\documentclass{article}
\usepackage{amsmath,amsfonts}
\usepackage{color}
\usepackage{xcolor}
\usepackage{graphicx}
\usepackage{subfigure}
\usepackage{amsthm}
\newcommand{\R}{\mathbb{R}}

\newcommand{\B}{\mathbb{B}}

\newcommand{\D}{\mathcal{D}}
\newcommand{\co}{\mathrm{co}}
\newcommand{\zero}{\mathbf{0}}
\newcommand{\range}{\mathrm{range}}

\newcommand{\sign}{\mathrm{sign}}

\newtheorem{theorem}{Theorem}
\newtheorem{definition}{Definition}
\newtheorem{corollary}{Corollary}
\newtheorem{example}{Example}

\newtheorem{assumption}{Assumption}
\newtheorem{lemma}{Lemma}
\newtheorem{remark}{Remark}
\DeclareMathOperator*{\esssup}{ess\,sup}

\title{\LARGE \bf 
Discontinuous  integro-differential equations    and sliding mode control
}
\author{Andrey Polyakov \thanks{Andrey Polyakov is with Inria Centre of the University of  Lille, CNRS CRIStAL, Centrale Lille, FR-59000, Lille, France, {andrey.polyakov@inria.fr}. 
		The work is supported by French National Research Agency  the project ANR-24-CE48-2771 SLIMDISC.
}}
\begin{document}
\maketitle

\begin{abstract}
    The paper deals with  analysis and design of sliding mode control  systems modeled by  finite-dimensional integro-differential equations.
    Filippov method and equivalent control approach are extended to a class of 
    nonlinear discontinuous integro-differential equations and to a class of control systems modeled by infinite-dimensional differential equations in Banach spaces. Sliding mode control  algorithms  are  designed for distributed input delay systems  and for a heat control system. 
\end{abstract}
\vspace{-2mm}
\section{Introduction}
\subsection{State of the art} 
Sliding Mode Control (SMC) is a robust control methodology, which uses discontinuous feedback laws for rejection of matched disturbances. The design and analysis  of SMC 
is well-developed for 
finite-dimensional systems \cite{Utkin1992:Book}, \cite{EdwardsSpurgeon1998:Book}, \cite{Shtessel_etal2014:Book}. Mathematical models of finite-dimensional  SMC systems are ill-posed discontinuous  ODEs (\textit{Ordinary Differential Equations}).
 Filippov method \cite{Filippov1988:Book}, \cite{Cortes2008:CSM} is aimed at regularization of discontinuous  ODEs.  The method consists in a minimal (in a set-theoretic sense) extension of an ill-posed discontinuous  ODE to a well-posed ODI (\textit{Ordinary Differential Inclusion}). The Filippov regularization has
 a rather simple mathematical formulation \cite{Filippov1988:Book}  based on the notion of a set of Lebesgue measure zero.
 Filippov method is the main  tool for regularization of  a SMC system \cite{Filippov1988:Book}, \cite{Utkin1992:Book}, \cite{EdwardsSpurgeon1998:Book}, \cite{Shtessel_etal2014:Book}, \cite{PolyakovFridman2014:JFI}, \cite{Utkin_etal2020:Book}
 while an analysis of the closed-loop dynamics of the SMC system is usually based on the so-called \textit{equivalent control method} 
 \cite{Utkin1992:Book}, \cite{EdwardsSpurgeon1998:Book}, \cite{Shtessel_etal2014:Book}.

Formally, the finite-dimensional Filippov method is inapplicable to discontinuous  \textit{infinite-dimensional systems} in abstract spaces \cite{DaleckyKrein1974:Book} (e.g., due to the absence of a notion of the set of Lebesgue measure zero in an abstract space). 
Traditionally (see, e.g., \cite{OrlovUtkin1985:DAN}, \cite{Orlov_elal2004:IJC},  \cite{Orlov2008:Book} and references therein), 
SMC design for infinite-dimensional systems  is based on an equivalent control method.  
Well-posedness analysis of  infinite-dimensional  SMC systems requires rather complicated constructions \cite{Levaggi2002:DIE},  \cite{Pisano_etal2011:SIAM_JCO},  \cite{Levaggi2013:DCDS}, \cite{Orlov2020:Book}.
 Infinite-dimensional differential inclusion  obtained from discontinuous infinite-dimensional differential equation in an  abstract  space (like in  \cite{Mohet_etal2023:SCL}) is a non-trivial object for research \cite{Tolstonogov2000:Book}. Some generalized solutions to discontinuous  functional differential equations of retarded type are introduced on \cite[page 58]{KolmanovskiiMyshkis1992:Book}.  The corresponding  regularization procedure consists in a construction of a convex hull of limits of all possible convergent sequences in a function space. Applicability of such non-constructive procedure to analysis and design of SMC systems seems questionable too.

\textit{Integro-Differential Equations} (IDEs) \cite{Lakshmikantham_etal1995:Book}, \cite{OReganMeehan1998:Book} model various finite and infinite dimensional systems, such as
epidemics \cite{KermackMcKendrick1991:BMB},  dispersive waves \cite{Debnath2005:Book},
systems with delays \cite{Hale1977:Book}, \cite{Fridman2014:Book}, \cite{Orlov2020:Book}, and others.  Input-output dynamics of some systems modeled by PDEs (\textit{Partial Differential Equations}) can also be described by IDEs (see Section \ref{sec:I.B}). 
Being a model of (possibly) infinite-dimensional system, an IDE is usually defined by  functions  in finite-dimensional spaces
(see, e.g., \cite{Lakshmikantham_etal1995:Book}, \cite{OReganMeehan1998:Book}). To the best of author's knowledge, \textit{finite-dimensional Filippov method} and \textit{SMC theory}  has never been developed
for  IDEs.
\vspace{-2mm}
\subsection{Motivation example: SMC design for a heat system}\label{sec:I.B}
Let us consider a control system modeled by the heat PDE\vspace{-1mm}
\begin{equation}\label{eq:heat}
	\tfrac{\partial x}{\partial t}\!=\!\nu \tfrac{\partial^2x}{\partial z^2}+
	\!\beta(z)(u(t)\!+\!\gamma(t)), \quad t\geq t_0\vspace{-1mm}
\end{equation}
with the homogeneous Dirichlet boundary conditions\vspace{-1mm}
\begin{equation}\label{eq:heat_boundary}
	x(t,0)=0, \quad x(t,1)=0, \vspace{-1mm}
\end{equation}
and the initial condition\vspace{-1mm}
\begin{equation}\label{eq:heat_ic}
	x(t_0,z)=x^0(z), \quad z\in [0,1],\vspace{-1mm}
\end{equation}
where $x(t,\cdot)$ is a distributed state of the system, $u(t)\in\R$ is the  control input,
$\beta: \R\to \R$  is a function that models an effect of the input to the system,  $\gamma\in L^{\infty}(\R;\R)$ is an unknown function (perturbation) with a known upper bound $\|\gamma\|_{L^{\infty}(\R;\R)}\leq \overline{\gamma}$, $\nu>0$ is a conductivity coefficient and $x^0$  is the initial state 
satisfying  the compatibility condition, e.g.,
\begin{equation}\label{eq:tilde_D}
 x^0\in \tilde{ \mathcal{D}}=\{x \in C^{2}([0,1];\R): x(0)=x(1)=0\}.
\end{equation}
A control system similar to  \eqref{eq:heat} also appears in the case of the  boundary control design using the Dirichlet lifting \cite{PrieurTrelat2019:TAC}, \cite{KatzFridman2021:EJC}.
 
Let a measured output of the  system be modeled as\vspace{-1mm}
\begin{equation}\label{eq:heat_y}
	y(t)\!=\!\int^1_0 \!\!\xi(z)x(t,z) dz, \quad \xi \!\in\! \tilde {\mathcal{D}}.\vspace{-1mm}
\end{equation}

\textit{The control aim is to steer the output  $y$ to zero in a finite time rejecting the effect of the unknown perturbation $\gamma$}. 

Let  a discontinuous output control law  be defined as \vspace{-1mm}
\begin{equation}\label{eq:heat_control}
	u(t)\!=\!q(t)\sign(y(t)), \quad q\in L^{\infty}([t_0,+\infty);\R). \vspace{-1mm}
\end{equation}
This is a SMC  provided that $q$  is properly selected  \cite{Orlov_elal2004:IJC}, \cite{Orlov2020:Book}.

The infinite-dimensional closed-loop control system \eqref{eq:heat} -- \eqref{eq:heat_control} is ill-posed in the SMC case.  A regularization of this discontinuous PDEs  with SMC is a nontrivial problem 
 (see, e.g., \cite[pages 24-28]{Orlov2020:Book}, \cite{Mohet_etal2023:SCL}). 
 However, the considered infinite-dimensional control system has the scalar  input $u$ and the scalar output $y$. The discontinuous feedback $u$ is a function of  output variable only. For such infinite-dimensional system, the well-posedness analysis can be reduced to a study of a finite-dimensional IDE that represents the input-output dynamics:
 \vspace{-2mm}
\begin{equation}\label{eq:heat_io_dynamics}
\! \,	\dot y(t)\!=\!p(t-t^0)+b(u(t)+\gamma(t))+\!\!\int^t_{t_0} \!\!\!\!\Phi(t-\tau)(u(\tau) +\gamma (\tau)) d\tau,\!\!\vspace{-2mm}
\end{equation}
where $y(0)=\int^1_0 \xi(z)x^0(z) dz$, $b=\int^1_0 \xi(z)\beta(z) dz$,  and $p, \Phi$ are some continuous functions 
(see Example \ref{ex:8} in Section \ref{sec:SMC_design}).

If the control law  $u$ is defined by the formula \eqref{eq:heat_control}, then the input-output dynamics of 
the system \eqref{eq:heat} - \eqref{eq:heat_control} 
should satisfy the discontinuous IDE \eqref{eq:heat_io_dynamics}, \eqref{eq:heat_control}. Due to discontinuity of $u$ on $y$, this IDE is ill-posed  \cite{Lakshmikantham_etal1995:Book},
 \cite{OReganMeehan1998:Book}.  However, the  IDE is defined in terms of finite-dimensional functions, which admit Filippov regularization \cite{Filippov1988:Book}.  
 This paper extends the Filippov method to a class of finite-dimensional IDEs, which covers the IDE \eqref{eq:heat_io_dynamics}, \eqref{eq:heat_control}  as a particular case. 
 A regularization of the IDE \eqref{eq:heat_io_dynamics}, \eqref{eq:heat_control}  allows  the original discontinuous PDE  \eqref{eq:heat} -- \eqref{eq:heat_control} to be easily regularized (see, e.g., \cite{Polyakov2025:CDC}).
Such an IDE-based approach to SMC design and analysis is applicable to a rather large class of systems modeled by differential equations in Banach spaces.\vspace{-3mm}
\subsection{Contributions of the paper}
\begin{itemize}
\item The Filippov method is extended to the IDE\vspace{-1mm}
\begin{equation}\label{eq:iODE_control}
\begin{split}
	\dot x(t)=&f(t,x(t),u(t,x(t)))\\
	&+\!\int^t_{t^0}\!\!\Phi(t,\tau) \tilde f(\tau,x(\tau),u(\tau,x(\tau))) d\tau,\vspace{-1mm}
\end{split}
\end{equation}  
where $t>t^0\in \R, x(t)\in \R^n$,  $f,\tilde f:\R\times \R^n\times \R^m\to \R^n$,{ $\Phi: \R\times \R\to \R^{n\times n}$}, $u:\R\times \R^n\to \R^m$ is a  (possibly) discontinuous  state feedback, which may be a SMC.
\item 
 The well-posedness analysis (existence, uniqueness and continuous dependence  of Filippov solutions on initial data)  is presented for the discontinuous IDE \eqref{eq:iODE_control} . 
\item The equivalent control method is developed for affine-in-control IDEs (i.e., for $f=a+bu$ and $\tilde f=\tilde a+\tilde bu$).
\item  A SMC  design for   control systems modeled by linear IDEs  is presented.
\item An IDE-based  SMC design and analysis for
some control systems modeled by differential equations in Banach spaces  is developed and illustrated on the system \eqref{eq:heat}-\eqref{eq:heat_control}.
\end{itemize}\vspace{-4mm}
\subsection{Organization of the paper}
The paper is organized as follows. First,  Filippov method for discontinuous ODEs is  briefly surveyed. The  surveyed results are utilized in the next section, where an extension of Filippov method  to IDEs is proposed. After that, the obtained theoretical results are applied to SMC design and analysis for finite-dimensional systems with distributed input delay and {for  infinite-dimensional systems modeled by differential equations in Banach spaces}. Finally, concluding remarks are given. 

\vspace{-2mm}
\section{Preliminaries: Discontinuous ODEs}\label{sec:hom}
 
 \subsection{Some classes of functions}
 	For $1\!\leq\! p\!\leq\! \infty$, $L^p$ denotes the well-known Lebesgue space.
 A function $(t,x)\in \R\times \R^n\mapsto f(t,x)\in \R^n$ is said to satisfy
 	\begin{itemize}
 	\item    \textit{Carath\'{e}odory Condition} if $f$ is
 	measurable in $t$ for all $x$, continuous in $x$ for almost all $t$ and  for any $\alpha,r>0$ there exists 
 	$m\in L^1((-\alpha,\alpha);\R)$ such that    $\left\|f(t,x)\right\| \leq m(t)$ for almost all $t\!\in\! (-\alpha,\alpha)$, for all $x\!\in\! \mathcal{B}(r)$, where $\mathcal{B}(r)$ is 
 	an open ball of the radius $r>0$, $\|\cdot\|$ - Euclidean norm.
  \item \textit{One-Sided Lipschitz Condition on $\mathcal{X}$} if
$\exists \ell\in L^{\infty}_{\rm loc}(\R;\R)$:\vspace{-1mm}
	\begin{equation}\label{eq:ch2_Lipschitz}
	(x-y)^{\top}(f(t,x)-f(t,y))\!\leq\! \ell(t)\|x-y\|^2,  \vspace{-1mm}
	\end{equation}
$\forall t\in \R, \forall x,y \!\in\! \mathcal{X}$, where 	$\ell$ may depend on $\mathcal{X}\subset \R^n$;
	  \item   \textit{Lipschitz Condition on $\mathcal{X}$} if $\exists \ell\in L^{\infty}_{\rm loc}(\R;\R)$:\vspace{-1mm}
	\begin{equation}\label{eq:ch2_Lipschitz}
		\|f(t,x)-f(t,y)\|\leq \ell(t)\|x-y\|,  \quad t\!\in\!\R,x,y \!\in\! \mathcal{X};\vspace{-1mm}
	\end{equation}	
	\item \textit{Filippov Condition} if
	$f$ is measurable and 
	for  any $\alpha,r\!>\!0$ there exists 
	$m\!\in\! L^1((-\alpha,\alpha);\R)$ :
	$\left\|f(t,x)\right\|\!\leq\!m(t)$ 
	for almost all $t\in (-\alpha,\alpha)$ and for almost all $x\in \mathcal{B}(r)$.
	\end{itemize}

{A measurable  function   $f: \R \times \R^{n} \to  \R^n$ is said to be \textit{locally (one-sided) Lipschitz} on $\mathcal{X}$ if it satisfies \textit{(One-Sided) Lipschitz Condition} on any compact from  $\mathcal{X}\subset \R^n$.  If $\mathcal{X}=\R^n$ we write shortly locally  (one-sided) Lipschitz.} 

A measurable  function   $(t,x)\in \R \times \R^{n} \mapsto f(t,x)\in \R^n$ is said to be \textit{piece-wise {(one-sided) Lipschitz} continuous in $x$}  if 
\begin{itemize}
	\item the space  $\R^{n}$ can be split into a finite number
	of open connected sets $G^j\subset \R^{n}, j=1,2,...,N$ with non-empty interiors  and continuous boundaries  $\partial G^j$;
	\item 
	$\bigcup_{i=1}^n \overline{G^j}=\R^n, G^i\bigcap G^j=\emptyset 
	\text{ for } i\neq j
	$ and the boundary set 
	$\mathcal{S}=\bigcup_{i=1}^N \partial G^j$ is
	of measure zero 
	\item  for each $j=1,...,N$,
	there exists  a locally (one-sided) Lipschitz continuous function $(t,x)\in (\R \times \R^{n} \mapsto f^j(t,x)\in \R^n)$ such that
  $f^j(t,x)\!=\!f(t,x)$ for  all $x\!\in\! G^j$ and almost all $t\!\in\! \R$.
\end{itemize}
\vspace{-1mm}
\subsection{Well-posed ODEs}
A problem is said to be \textit{well-posed} in Hadamard sense if 
\textit{1) it has a solution; 2) the solution is unique; 3) the solution  changes continuously with the initial data.}  
The  ODE\vspace{-1mm}
\begin{equation}\label{eq:ch2_generalsystem}
	\dot x(t) = f(t,x(t)), \quad f :\R \times \R^n\to  \R^n\vspace{-1mm}
\end{equation} 
defines a rule of an evolution of the state $x(t)\in \R^n$ of the system in time $t\in \R$.
Let $\mathcal{I}$ be one of the following intervals: $[a,b]$, $(a,b)$, $[a,b)$ or $(a,b]$, where $a,b\in[-\infty,+\infty],a<b$.
\begin{definition}[\small Classical solution of ODE]\; 
	\itshape 	A  function $ x\in C^1(\mathcal{I};  \R^n)$ is said to be a {classical solution of ODE} \eqref{eq:ch2_generalsystem} if
	it satisfies the equation \eqref{eq:ch2_generalsystem} everywhere on  $\mathcal{I}$.
\end{definition}

 The Initial Value Problem (IVP)  is to find  a solution  \eqref{eq:ch2_generalsystem}
for a given  initial state $x_0\in \R^n$ at a given initial time $t_0\in\R$:\vspace{-1mm}
\begin{equation} \label{eq:ch2_initialcondition}
	x(t^0)=x^0\in \R^n.\vspace{-1mm}
\end{equation} 
For the proofs of the classical theorems surveyed  in  this subsection we refer the reader, for example, to \cite{CoddingtonLevinson1955:Book}, \cite{Filippov1988:Book}.
\begin{theorem}[\small Peano Existence Theorem, 1890]\label{Thm:Piano} 
	\itshape If $f$ is continuous then 
	the IVP \eqref{eq:ch2_generalsystem}, \eqref{eq:ch2_initialcondition} has a classical solution defined, at least, locally in time.
\end{theorem}

The identity \eqref{eq:ch2_generalsystem} may hold almost everywhere  if $x : \mathcal{I} \mapsto \R^n$ is an {absolutely continuous function}\footnote{A function $x: \mathcal{I}\mapsto \R^n$ is said to be absolutely continuous if $\forall \varepsilon>0$, $\forall \delta>0$ such that whenever a finite sequence of pairwise disjoint intervals $(\alpha_k,\beta_k)\subset \mathcal{I}$ satisfies $ \sum_k(\beta_k-\alpha_k)<\delta$ then $\sum_k |x(\beta_k)-x(\alpha_k)|<\varepsilon$.}.
\begin{definition}[\small Carath\'{e}odory solution of ODE]\;\itshape
	An absolutely continuous function $ x : \mathcal{I}\subset \R^n \to \R^n$  is said to be  a  Carath\'{e}odory (or strong) solution of ODE \eqref{eq:ch2_generalsystem} if  it satisfies the equation \eqref{eq:ch2_generalsystem} almost everywhere on the time interval $\mathcal{I}$.
\end{definition}

The existence of Carath\'{e}odory solutions can be proven for ODEs with discontinuous in $t$ right-hand sides.
\begin{theorem}[\small Carath\'{e}odory Existence Theorem, 1918]\label{Theorem:Car}\hfill\newline
\itshape If $f$ satisfies Carath\'{e}odory Condition then the IVP \eqref{eq:ch2_generalsystem}, \eqref{eq:ch2_initialcondition} has a Carath\'{e}odory solution  defined, at least, locally in time.
\end{theorem}

The strong solutions allow some systems with deterministic noise to be modeled by ODEs, e.g.,  the system $\dot x(t)=-x(t)+\gamma(t)$ does not have a classical solution if $\gamma\in L^{\infty}(\R;\R)$ is a discontinuous signal, but it has the strong solution. 

Let $x(t,t^0,x^0)$ denote  a solution of the IVP  \eqref{eq:ch2_generalsystem}, \eqref{eq:ch2_initialcondition}. 

\begin{theorem}[\small Continuity of solutions on initial data {\cite{CoddingtonLevinson1955:Book},\,page\,58}]\label{thm:con_sol_strong_Rn}
\itshape If $f$ satisfies Carath\'{e}odory Condition, then $x(t,t^1,x^1)\to x(t,t^0,x^0)$ as $(t^1,x^1)\to (t^0,x^0)$ uniformly on $\mathcal{I}$ provided that
the
solution $x(t,t^0,x^0)$
 is  unique on $\mathcal{I}$.
\end{theorem}

Uniqueness of solutions can be guaranteed under additional restrictions to the right-hand side of \eqref{eq:ch2_generalsystem}. 
\begin{theorem}[Uniqueness of solutions]	\itshape \label{thm:uniqness_ODE}
If  a locally {(one-sided)} Lipschitz continuous function $f$ satisfies   Carath\'{e}odory  Condition
then the IVP \eqref{eq:ch2_generalsystem}, \eqref{eq:ch2_initialcondition} has a unique  Carath\'{e}odory solution defined, at least, locally in the (resp., forward) time.
\end{theorem}
\vspace{-1mm}
\subsection{Regularization of ill-posed ODE by Filippov method}

Systems with discontinuous in $x$ right-hand sides may have Carath\'{e}odory solutions in some cases \cite{Liberzon2003:Book}.
However, SMC theory deals with control systems modeled by \textit{ill-posed} ODEs with discontinuous right-hand sides \cite{Utkin_etal2009:Book, Liberzon2003:Book, GeligLeonovYakubovich2004:Book}. 
\begin{example}\itshape Let us consider the following SMC system\vspace{-1mm}
\begin{equation}\label{eq:ch2_motiv_example}
	\dot x(t)\!=\!u(t)\!+\!0.5\sin(3t),  \quad u(t)\!=\!-\sign(x(t)),  \vspace{-1mm}
\end{equation}
where $t\in \R$ is the time variable, $x(t)\in \R$ is the state variable, $u(t)\in \R$ is the control variable,  
$\sign(x(t))=1$ for $x(t)>0$, $\sign(x(t))=-1$ for $x(t)<0$ and $\sign(0)=0$. The ODE \eqref{eq:ch2_motiv_example} admits classical solutions on $\{x\!\in\! \R: x\!>\!0\}$ and on $\{x\!\in\! \R : x\!<\!0\}$. Each solution with $x(t_0)\!=\!x_0\!\neq\! 0$ reaches zero  in a finite time $t=t^0+T(t^0,x^0)$, since $\frac{d|x(t)|}{dt}
\leq-0.5$ for $x(t)\neq 0$. Hence, the natural continuation of a solution for $t>t^0+T(t^0,x^0)$ is $x(t)=0$, but neither Peano nor Carath\'{e}odory's existence theorem is applicable in this case. 
In addition, formally, if $x(t)=0$ for all $t>t^0+T(t^0,x^0)$ then $\dot x(t)=0=u(t)+0.5\sin(3t)$ for $t>t^0+T(t^0,x^0)$, but, by definition, $\sign(x(t))=0$ for $x(t)=0$.
\end{example}

To define a solution of the system \eqref{eq:ch2_generalsystem} with discontinuous (on $x$) right-hand side, a special regularization procedure known as the \textit{Filippov method} \cite{Filippov1988:Book} can be utilized. It extends  an ill-posed ODE to a well-posed ODI\vspace{-1mm}
\begin{equation}\label{eq:Filippov_DI}
	\dot x \in K[f](t,x),\vspace{-1mm}
\end{equation}
where $K[f]:\R\times \R^n\to 2^{\R^n}$ is defined as \cite[page 85]{Filippov1988:Book}:\vspace{-1mm}
	\begin{equation}\label{eq:Filippov_reg}
		K[f](t,x)=\bigcap_{\delta>0}\; \bigcap_{\mu(M)=0}  \;\overline{\co} f(t,x+\mathcal{B}(\delta)\backslash M),\vspace{-2mm}
	\end{equation}
	where $\overline{\co}$ denotes the convex closure, and the intersections 
	are taken over all sets $M\subset\R^n$ of measure zero and all 
	$\delta>0$  and $2^{\R^n}$ denotes the power set of $\R^n$.  The set-valued mapping $(t,x)\to K[f](t,x)$ is \textit{nonempty-valued, compact-valued, convex-valued and upper semi-continuous in $x$ for almost all $t$ }provided that $f$ is a measurable function (see, e.g., \cite{Filippov1988:Book}). Notice that \textit{Filippov regularization} \eqref{eq:Filippov_reg} does not modify continuous in $x$ function $f$.

	\begin{definition}[\small Filippov solution]\label{def:Filippov}
		\itshape An absolutely continuous function $x : \mathcal{I}\to\R^n$  is said to be
		a {Filippov solution} of \eqref{eq:ch2_generalsystem} if
 it satisfies the ODI \eqref{eq:Filippov_DI}, \eqref{eq:Filippov_reg}
		almost everywhere on $\mathcal{I}$.
	\end{definition}

Filippov method regularizes various discontinuous ODEs.
\addtocounter{example}{-1}
\begin{example}[continued]\itshape
Any solution of the discontinuous differential equation  \eqref{eq:ch2_motiv_example}  can be  prolonged 
by Filippov method for $t>t^0+T(t^0,x^0)$. The corresponding differential inclusion \vspace{-2mm}
	$$
	\dot x(t)\!\in\! -\overline{\sign}(x(t))+0.5\sin(3t), \quad \overline{\sign}(x)\!=\!\left\{
	\begin{smallmatrix}
		1& \text{if} &x>0,\\
		\mathrm{[}-1,1\mathrm{]} &\text{if} &  x=0,\\
		-1 & \text{if} & x<0,\\
	\end{smallmatrix}
	\right.\vspace{-1mm}
	$$
	obviously, admits the forward unique zero solution.
\end{example}
\begin{theorem}[Filippov Existence Theorem, 1960]\itshape
{If $f$ satisfies Filippov Condition}, then the IVP \eqref{eq:ch2_generalsystem}, \eqref{eq:ch2_initialcondition} has a Filippov solution  defined, at least, locally in time.
\end{theorem}

	Continuity of Filippov solutions on initial data  \cite[page 98]{Filippov1988:Book} repeats the classical result if solutions of  \eqref{eq:Filippov_DI} are unique.   
	\begin{theorem}[\small Continuity of Filippov solutions on initial data]\label{thm:con_sol_Filippov_Rn}
		\itshape Let $f $ satisfy Filippov Conditions
		and let a Filippov solution $x(t,t^0,x^0)$ of the IVP \eqref{eq:ch2_generalsystem}, \eqref{eq:ch2_initialcondition}  be defined and unique on some time interval $\mathcal{I}$.
		Then $x(t,t^1,x^1)\to x(t,t^0,x^0)$ as $(t^1,x^1)\to (t^0,x^0)$ uniformly on compacts from $\mathcal{I}$.
	\end{theorem}
\vspace{-1mm}
\subsection{Uniqueness of Filippov solutions of ODEs}
Lipschitz condition
	does not hold for discontinuous (on  state variable) ODEs.
	Filippov solution of a discontinuous ODE is understood as a solution of an ODI.  The uniqueness analysis for solutions of the ODI  is a non-trivial problem in the general case  \cite{Filippov1988:Book}, \cite{Tolstonogov2000:Book}. 
{However, \textit{One-Sided Lipschitz Condition} is still applicable for uniqueness analysis of Filippov solutions.
\begin{theorem}[Uniqueness of Filippov solutions, \cite{Filippov1988:Book}, \S10]	\itshape \label{thm:uniqne_Filippov}
If  a locally one-sided Lipschitz function $f$ satisfies  Filippov Condition then the IVP \eqref{eq:ch2_generalsystem}, \eqref{eq:ch2_initialcondition} has a unique  Filippov solution defined, at least, locally in the forward time.
\end{theorem}

Notice that \textit{One-Sided Lipschitz Condition} is still rather conservative.
For example, the right-hand side of the system \vspace{-1mm}
\begin{equation}\label{eq:ex_uniqness}
\dot x_1(t)=-\sign(x_1(t), \quad \dot x_2(t)=\sign(x_1(t))\vspace{-1mm}
\end{equation}  does not satisfy \textit{One-Sided Lipschitz Condition}, but its Filippov solution is unique. 
}
  Less conservative sufficient conditions for uniqueness of Filippov solution can be obtained for particular classes of discontinuous ODEs. 
  The Filippov regularization \eqref{eq:Filippov_reg}  of a piece-wise continuous system  gives\vspace{-1mm}
\begin{equation}\label{eq:K[f]}
	K[f](t,x)=\left\{
	\begin{smallmatrix}
		\{f(t,x)\} & \text{ if } & x\in \R^{n}\backslash \mathcal{S},\\
		\overline{\co}\left( \bigcup\limits_{j\in \mathcal{N}(x)}\left\{f^j(t,x)\right\} \right) & \text{ if } & x\in \mathcal{S},
	\end{smallmatrix}
	\right.\vspace{-1mm}
\end{equation}
where, $\mathcal{S}=\bigcup_{i=1}^N \partial G^j$ is the discontinuity set of $f$, but 
the set-valued function $\mathcal{N}: \mathcal{S} \to 2^{\{1,2,...,N\}}$ 
indicates the continuity domains $G^{j}$, which have a common boundary point $x\in\mathcal{S}$, i.e.,
$
\mathcal{N}(x)=\left\{j\in\{1,2,...,N\} : x\in \partial G^j \right\}.
$ 

Let $\mathcal{C}_{\Omega}(x)\subset \R^n$ denote the (Bouligand)  tangent cone to the set $\Omega\subset \R^n$ at $x\in \R^n$ defined as (see., e.g. \cite{Nagumno1942:PPMSJ}) \vspace{-1mm}
	\begin{equation}
		\mathcal{C}_{\Omega}(x)=\left\{z\in \R^n: \liminf_{h\to 0^+}\tfrac{\inf_{y\in \Omega}\|x+hz-y\|}{h}=0 \right\}.\vspace{-1mm}
	\end{equation}
Tangent cones characterize positively invariant sets\footnote{A set $M\!\subset\! \R^n$ is positively invariant if $x(t^0)\!\in\! M \Rightarrow x(t)\!\in\! M,t\!\geq\! t_0.$ } of dynamical systems (see, e.g., \cite{Clarke_etal1995:JDCS}, \cite{Blanchini1999:Aut}, \cite{AubinEkeland1984:Book} and references therein). 
\begin{assumption}\label{as:uniq_1}\itshape
	Let  a piece-wise one-sided Lipschitz continuous function $f$ be such that 
		$f^j(t,x)\!\notin\! \mathcal{C}_{\bar G^j}(x)$ for all  $x\!\in\! \partial G^j$,
		for almost all $t\!\in\!\R$ and  for all $j\!=\!1,2,...,N$.
\end{assumption}

 The following result is partially inspired by \cite[\S 10]{Filippov1988:Book}. 
\begin{theorem}[\small Uniqueness of solutions in sliding mode case]\label{thm:unique_SM}\itshape
	Let\vspace{-1mm} 
		\begin{equation}
		F_{\mathcal{S}}(t,x)= K[f](t,x)\cap\mathcal{C}_{\mathcal{S}}(x), \quad t\in \R, \quad x\in \mathcal{S}.\vspace{-1mm}
		\end{equation}
 If the set 
  $F_{\mathcal{S}}$ is a singleton, i.e., $F_{\mathcal{S}}(t,x)=\{f^0(t,x)\}$ for all $x\in \mathcal{S}$, for almost all $t\in \R$, with
 for some locally {one-sided} Lipschitz continuous function $f^0$, then,  {under  Assumption \ref{as:uniq_1},}  
  the IVP  \eqref{eq:ch2_generalsystem}, \eqref{eq:ch2_initialcondition}  has a Filippov solution 
   being unique,  at least, locally in the forward time. 	
\end{theorem}
\begin{proof}
A solution of \eqref{eq:ch2_generalsystem} coincides with the solution of the continuous ODE $\dot x(t)=f^j(t,x(t))$ as long as $x(t)\in G^j$.
 {Assumption \ref{as:uniq_1}} implies that 
the set $\R^n\backslash  G^j$ is strictly positively invariant 
for the  continuous ODE $\dot x(t)=f^j(t,x(t))$.
Hence, the system $\dot x(t)=f^j(t,x(t))$ (as well as the system \eqref{eq:ch2_generalsystem}) with the initial condition $x(t^0)=x^0\in \partial G^j$ do not have solutions belonging to  the open set $G^j$ for $t>t^0$.
Since this holds for each $G^j$, then the set $\mathcal{S}$ (being the  union of $\partial G^j$)  is  positively invariant for the system
\eqref{eq:ch2_generalsystem}.
{The condition $F_{\mathcal{S}}=\{f^0\}$} means that the ODE $\dot x(t)=f^0(t,x(t))$ defines the motion (known as \textit{sliding mode} \cite{Utkin1992:Book}, \cite{Filippov1988:Book}) of the system on $\mathcal{S}$.  {Theorem \ref{thm:uniqness_ODE}} ensures the uniqueness of this motion on the \textit{sliding set} $\mathcal{S}$.  
\end{proof}

{ Assumption \ref{as:uniq_1} guarantees that the discontinuity set $\mathcal{S}$ is the sliding set of the system. The dynamics on $\mathcal{S}$ is governed by \vspace{-1mm}
	\begin{equation}
		\dot x(t)\in F_{\mathcal{S}}(t,x(t)).\vspace{-1mm}
	\end{equation}
The system \eqref{eq:ex_uniqness} satisfies the above theorem with $F_{\mathcal{S}}=\{0\}$. }
The relay SMC system is the conventional example of discontinuous ODE satisfying the above theorem \cite{Utkin1992:Book}, \cite{Filippov1988:Book}.
\begin{example}\label{ex:2}\itshape
A simple example of discontinuous ODE with unique solutions can be  discovered in  \cite{Utkin1992:Book} for the SMC system\vspace{-1mm}
	\begin{equation}\label{eq:examp_SMC_relay}
	\dot x(t)=Ax(t)+Bu(t),\vspace{-1mm}
\end{equation} 
with $u(t)=-(CB)^{-1}CAx(t)+L \,\sign(Cx(t))$,
where $A\in \R^{n\times n}, B\in \R^{n\times m},L\in \R^{m\times m}$, $C\in \R^{m\times n}$, $\det(CB)\neq 0$ and $\sign$ of a vector is understood component-wisely. 
The dynamics of the sliding variable $\sigma=Cx$ is given by\vspace{-1mm}
\[
\dot \sigma(t)=CBL\sign(\sigma(t)).\vspace{-1mm}
\]
If $L\!=\!\rho (CB)^{-1}$, $\rho\! >\! 0$ then conditions of Theorem \ref{thm:unique_SM} are fulfilled and $f^0(t,x)\!=\!(A\!-\!B(CB)^{-1}CA)x$  for $Cx\!=\!\zero$.
\end{example}
The uniqueness of Filippov solutions of piece-wise continuous ODE without sliding sets can be proven under the assumption. 
\begin{assumption}\label{as:uniq_2}\itshape
	Let $f$ be piece-wise {one-sided Lipschitz} continuous
	and 
	$F_{\mathcal{S}}(t,x)=\emptyset$ for all $x\in \mathcal{S}$,
	for almost all $t\in\R$.
\end{assumption}
\begin{theorem}[\small Uniqueness of solutions in switching case]\label{thm:unique_no_SM}\itshape
	 Under Assumption \ref{as:uniq_2}, the IVP  \eqref{eq:ch2_generalsystem}, \eqref{eq:ch2_initialcondition}  has a Carath\'{e}odory  solution being unique,  at least, locally in the forward time. 	
\end{theorem}
\begin{proof}
The condition $	F_{\mathcal{S}}(t,x)=K[f](t,x)\cap\mathcal{C}(x)=\emptyset$ means that any  Filippov solution initiated on the discontinuity set $\mathcal{S}$ immediately leaves the set $\mathcal{S}$ by entering into  some $G^j$.
Since $f^j$  satisfies Theorem \ref{thm:uniqness_ODE} then any Filippov solution  of \eqref{eq:ch2_generalsystem} is a unique
Carath\'{e}odory solution.
\end{proof}
\begin{assumption}\label{as:uniq_3} \itshape
	Let $f$ be piece-wise one-sided Lipschitz continuous
	function  and there exists a closed set $\mathcal{M}\subset \mathcal{S}\subset \R^n$: 
	\begin{itemize}
		\item 
		$
		F_{\mathcal{S}}(t,x)=\emptyset$, $\forall x\!\in\! \mathcal{S}\backslash \mathcal{M}$, for
		almost all $t\!\in\!\R$;
		\item $f^j(t,x)\!\notin\! \mathcal{C}_{\bar G^j}(x)$ for all  $x\!\in\! \partial G^j\cap\mathcal{M}$,
		for almost all $t\!\in\!\R$ and  for all $j\!=\!1,2,...,N:\partial G^j\cap\mathcal{M}\neq \emptyset$.
	\end{itemize}
\end{assumption}
The following corollary combines both above theorems.
\begin{corollary}[\itshape \small Uniqueness of solution of switching sliding mode]\label{cor:uniqness_mix}\itshape
		Under Assumption \ref{as:uniq_3}, if 
	$F_{\mathcal{S}}(t,x)\!=\!\{f^0(t,x)\}$ is a singleton for all $x\!\in\! \mathcal{M}$ and  almost all $t\!\in\! \R$ and  $f^0$ is  locally one-sided Lipschitz continuous, then
		 the IVP  \eqref{eq:ch2_generalsystem}, \eqref{eq:ch2_initialcondition}  has a Filippov solution  being unique,  at least, locally in the forward time. 	
\end{corollary}

{Assumption \ref{as:uniq_3} ensures that $\mathcal{M}$ is the sliding set and the switching set $\mathcal{S}\backslash\mathcal{M}$ is crossed by trajectories without sliding.}
\begin{example}
	\itshape
	The system \eqref{eq:examp_SMC_relay} with 
	 $m=2$, $\sigma=(\sigma_1,\sigma_2)^{\top}$, $L=\rho(CB)^{-1} \left(\begin{smallmatrix} -1 & 2\\ -2 & -1\end{smallmatrix}\right)$ and  $\rho>0$ satisfy
	all conditions of Corollary \ref{cor:uniqness_mix} with $\mathcal{M}=\{0\}$.
\end{example}
\vspace{-2mm}
\subsection{Equivalent control method for systems modeled by ODEs}
Let us consider a control system modeled by ODE\vspace{-1mm}
\begin{equation}\label{eq:ODE_control}
	\dot x(t)=f(t,x(t),u(t,x(t))), \vspace{-1mm}
\end{equation}
where $t$ are $x$ are as before, $f\in \R\times \R^n\times \R^m\to \R^n$ and $u: \R\times \R^n\to\R^m$
is a measurable  function (e.g., SMC). In this case, some alternative methods
of regularization of discontinuous ODE 
can be proposed (e.g.,  \cite{Utkin1992:Book},  \cite{Filippov1988:Book}, \cite{PolyakovFridman2014:JFI}). 
\begin{definition}[\small Utkin solution]\label{def:Utkin}
	\itshape An absolutely continuous function $x : \mathcal{I}\to\R^n$  is said to be
	an Utkin solution of the first kind (resp., of the second kind) to \eqref{eq:ODE_control} if
	there exists a measurable function $u_{\rm eq}:\mathcal{I}\to \R^m$ such that \vspace{-1mm}
	\begin{equation}
		\dot x(t)= f(t,x(t), u_{\rm eq}(t)), \quad u_{\rm eq}(t)\in U(t,x(t))\vspace{-1mm}
	\end{equation}
	almost everywhere on $\mathcal{I}$, where $U=K[u]$ is defined by \eqref{eq:Filippov_reg} (resp., $U=(K[u_1], ..., K[u_m])^{\top}$ for $u=(u_1, u_2,...,u_m)^{\top}$).
\end{definition}

Since  $K[u](t,x(t))\subset (K[u_1](t,x(t)),...,K[u_m(t,x(t))])^{\top}$ then 
the existence of the Utkin solutions of the first  kind always implies the existnce of the Utkin solutions of the second kind. They always coincide for single input systems.
\begin{example}[Neimark 1960]\itshape A simple analysis  (see, e.g., \cite{Neimark1961:IFAC_Congress} or \cite[Example 2]{PolyakovFridman2014:JFI}) shows that the set of Utkin solutions of the first kind do not coincide with the set of Utkin solutions of the second kinds for the linear time-invariant control system \eqref{eq:examp_SMC_relay} 
	with	$x=(x_1,x_2)\in \R^2$, $A=\zero$, $B=I_2$ and 
	the static relay feedback law $u=(u_1,u_2)^{\top}, u_1=u_2=\sign(x_1)$.
\end{example}

The function $u_{eq}$ is the so-called \textit{equivalent control} \cite{Utkin1992:Book}, \cite{Filippov1988:Book}, which,
in SMC theory, is defined as a solution of a special algebraic equation (see below). 
The equivalent control method is the main tool for analysis and design of SMC systems \cite{Utkin1992:Book}, \cite{EdwardsSpurgeon1998:Book}, \cite{Shtessel_etal2014:Book}, \cite{Utkin_etal2020:Book}, \cite{Orlov2020:Book}.
The existence of $u_{\rm eq}$  follows from  \cite{Filippov1962:SIAM}.
\begin{lemma}[\small Filippov 1959]\label{lem:Filippov_selector}\itshape
	Let a function $(z,u)\in \R^p \times \R^m \mapsto w(z,u)\in \R^n$ be measurable in $z$ and continuous in $u$. Let a {compact}-valued mapping  $U :\R^{p} \to 2^{\R^m}$  
	and a function $v:\R^p\to \R^m$ be measurable. If $v(z)\in \{w(z,u): u\in U(z)\}$ almost everywhere 
	then $\exists u_{eq}:\R^p\to \R^m$ being measurable, $u_{\rm eq}(z)\in U(z)$   and  $v(z)=w(z,u_{eq}(z))$ almost everywhere. 
\end{lemma}
The set-valued mapping $t\mapsto U(t, x(t))$  from Definition \ref{def:Utkin} satisfies Lemma \ref{lem:Filippov_selector}, so an Utkin solution exists if the  ODI\vspace{-1mm}
\begin{equation}\label{eq:Ueq_DI}
	\dot x(t)\in f(t,x(t),U(t,x(t))), \quad t\in\R,\vspace{-1mm}
\end{equation}
has a solution (see, e.g., \cite{Filippov1988:Book} or \cite{PolyakovFridman2014:JFI}).
The mapping $(t,x)\mapsto  f(t,x,U(t,x))$ may be nonconvex-valued, so the conventional existence theorems for ODIs  \cite{Filippov1988:Book} are not applicable  \cite{BartoliniZolezzi1985:IEEE_TAC}. 

The equivalence of Filippov and Utkin solutions for affine-in-control systems was studied in the literature (see, e.g.,  \cite{Utkin1992:Book}).
\begin{theorem}\label{thm:UtkinZolezzi}\itshape
	Let  the system \eqref{eq:ODE_control} be affine on the input \vspace{-1mm}
	$$
	f(t,x)=a(t,x)+b(t,x)u(t,x),\vspace{-1mm}
	$$
	and  the locally essentially bounded\footnote{ A function $\xi : \mathcal{X}\subset \R^p\to \R^q$ is said to be locally essentially bounded  on $\mathcal{X}$ if $\esssup_{z\in\mathcal{Z}} \|\xi(z)\|<+\infty$ for any compact  $\mathcal{Z}\subset \mathcal{X}$.} functions 
	$a: \R\times \R^n \mapsto  \R^n$ and 
	$b: \R\times \R^n \mapsto  \R^{n\times m}$ be continuous in $x$ for almost all $t$ and measurable in $t$ for all $x$.
	If a measurable function $u:\R\times \R^{n}\mapsto \R^{m}$ is locally essentially bounded 
	then,  
 $\forall  t^0\in \R$, $\forall x^0\in \R^n$, the IVP \eqref{eq:ch2_generalsystem}, \eqref{eq:ch2_initialcondition}  has an Utkin  solution of the first kind defined, at least, locally in the time; any Utkin solution of the first kind is a Filippov solution and vice versa.
\end{theorem}
This follows from Filippov existence theorem and the identity\vspace{-1mm}
\[
K[f](t,x,u(t,x))=a(t,x)+b(t,x)K[u](t,x),\vspace{-1mm}
\]
being fulfilled for almost all $t$, due to  continuity of $a$, $b$ in $x$.

A sufficient condition of equivalence of Utkin solutions of the first  and of the second kinds is presented, for example, in \cite[Theorem 14, page 44]{Zolezzi2002:inBook} for affine-in-control systems.

If $\mathcal{S}=\{x\in \R^n: s(x)=\zero\}$, $s\in C^{1}(\R^n;\R^k)$ is a sliding set of the system  \eqref{eq:ODE_control} then the equivalent control 
on $\mathcal{S}$ is defined as a solution of the algebraic equation  \cite{Utkin1992:Book}:\vspace{-1mm}
\begin{equation}\label{eq:equation_for_u_eq}
\tfrac{\partial s(x)}{\partial x}f(t,x,u_{\rm eq})=\zero.\vspace{-1mm}
\end{equation}
Let  this algebraic equation have a  unique solution $u_{\rm eq}(t,x)$. If  $u_{\rm eq}(t,x)\in U(t,x)$ for all $x\in \mathcal{S}$ and for almost all $t\in \R$  and $f_0(t,x):=f(t,x, u_{eq}(t,x))$ satisfies Carath\'{e}odory existence theorem then  the system \eqref{eq:ch2_generalsystem} has an Utkin solution sliding on $\mathcal{S}$ and governed by the ODE $\dot x=f_0(t,x)$. If  \vspace{-1mm}
\begin{equation}\label{eq:con_utkin_unique}
	\det\left(\tfrac{\partial s(x)}{\partial x}b(t,x)\right)\neq 0, \quad \forall x\in \mathcal{S}, \quad \forall t\in \R,\vspace{-1mm}
\end{equation}
for the affine-in-control system from Theorem \ref{thm:UtkinZolezzi}, then \cite{Utkin1992:Book}:
\vspace{-1mm}
 \begin{equation}
 	u_{\rm eq}(t,x)=-\left(\tfrac{\partial s(x)}{\partial x}b(t,x)\right)^{-1}\tfrac{\partial s(x)}{\partial x}a(t,x), \quad x\in \mathcal{S}.\vspace{-2mm}
 \end{equation} 
 If $(t,x)\!\mapsto\!\!f^0(t,x)\!=\!a(t,x)\!-\!b(t,x)\!\left(\!\frac{\partial s(x)}{\partial x}b(t,x)\!\right)^{\!\!-\!1}\!\!\frac{\partial s(x)}{\partial x}a(t,x)$ is locally one-sided Lipschitz on $\mathcal{S}$ continuous function  then \textit{any Filippov solution belonging to the sliding set is unique}.
 This does not mean that all solutions with $x_0\in \mathcal{S}$ are unique, since  the sliding set $\mathcal{S}$ may not be strongly positively invariant.
 
 \vspace{-1mm}

\section{Discontinuous Integro-differential equations}\label{sec:IDE}
Let us consider  the integro-differential equation\vspace{-1mm}
\begin{equation}\label{eq:iODE}
	\dot x(t)\!=\!f(t,x(t))\!+\!\int^t_{t^0} \!\Phi(t,\tau) \tilde f(\tau,x(\tau)) d\tau, \quad t>t^0,\vspace{-1mm}
\end{equation}
where the functions $\Phi:\Delta\to  \R^{n\times n}$ with  $\Delta=\{(t,\tau)\in \R\times \R: t\geq \tau\}$
and $f,\tilde f:\R\times \R^n\to\R^n$ are measurable.  
\begin{assumption}\label{as:ex_IDE}\itshape
		Let $\Phi$ be locally essentially bounded on $\Delta$ and $f,\tilde f$ satisfy Filippov Condition. 
\end{assumption}
\vspace{-2mm}
\subsection{Well-posed  IDEs}
Continuous integro-differential equations of the form \eqref{eq:iODE} are well studied in the literature (see, e.g.,  \cite{Lakshmikantham_etal1995:Book}). Below  we use the notation $\mathcal{I}=[t^0,t^0+\alpha)$ with $\alpha>0$ or $\mathcal{I}=[t^0,+\infty)$.
\begin{definition}[\small Classical solution of IDE]\itshape
	A  function $ x\in C^1(\mathcal{I};  \R^n)$ with $x(t^0)=x^0$ is said to be a classical solution of the IVP   \eqref{eq:iODE},  \eqref{eq:ch2_initialcondition} if
	it satisfies  \eqref{eq:iODE} everywhere on  $\mathcal{I}$.
\end{definition}
The existence theorem of classical solutions of the integro-differential equation \eqref{eq:iODE} repeats the Peano existence theorem.
\begin{theorem}[see, e.g., \cite{Lakshmikantham_etal1995:Book}]\label{thm:peano_iODE}\itshape
If $f$, $\tilde f$, $\Phi$ are continuous  functions  then  the IVP \eqref{eq:iODE}, \eqref{eq:ch2_initialcondition} has a classical solution defined, at least, locally in the forward time.
	\end{theorem}
Carath\'{e}odory solutions of integro-differential equations are defined similarly to solutions of ODEs  \cite{OReganMeehan1998:Book}.
\begin{definition}[\small Carath\'{e}odory solution of IDE]\itshape
	An absolutely  continuous function $ x:\mathcal{I}\to \R^n$ with  $x(t^0)=x^0$ is said to be a Carath\'{e}odory (or strong) solution of the IVP   \eqref{eq:iODE},  \eqref{eq:ch2_initialcondition} if
	it satisfies  \eqref{eq:iODE} almost everywhere on  $\mathcal{I}$.
\end{definition}

Carath\'{e}odory solutions of IDEs are well-studied too \cite{OReganMeehan1998:Book}.
\begin{theorem}\label{thm:Caratheodory_iODE}\itshape
	Let Assumption \ref{as:ex_IDE} hold.  If  $f$,$\tilde f$ satisfy 
		Carath\'{e}\-odory  Condition then the IVP \eqref{eq:iODE}, \eqref{eq:ch2_initialcondition} has a Carath\'{e}\-odory  solution defined, at least, locally in the forward time.
\end{theorem}

Similarly to ODEs, a Lipschitz condition  is sufficient for uniqueness of solutions of IDEs (see, e.g., \cite{Lakshmikantham_etal1995:Book}, \cite{OReganMeehan1998:Book}).
\begin{theorem} \itshape \label{thm:uniqness_iODE} 	Let Assumption \ref{as:ex_IDE} be fulfilled and $f$, $\tilde f$ satisfy 
	Carath\'{e}\-odory  Condition. If $f$ is locally one-sided  Lipschitz  and $\tilde f$ is locally Lipschitz, then
  the IVP \eqref{eq:iODE}, \eqref{eq:ch2_initialcondition} has a  unique solution defined, at least, locally in the forward time. 
\end{theorem}	
\vspace{-2mm}

\subsection{Regularization of ill-posed IDEs by Filippov method}
If $f,\tilde f$ are  discontinuous in $x$ functions then the IVP \eqref{eq:iODE}, \eqref{eq:ch2_initialcondition}  may be ill-posed. Let us regularize it by Filippov method.
\begin{definition}[Filippov solution of IDE] \itshape
\quad An absolutely continuous function $x: \mathcal{I}\to \R^n$  is said to be a Filippov solution of the IVP  \eqref{eq:iODE}, \eqref{eq:ch2_initialcondition} if  $x(t^0)=x^0$ and, almost everywhere on $\mathcal{I}$, it satisfies  the integral differential inclusion\footnotemark\vspace{-1mm}
\begin{equation}\label{eq:Filippov_iODE}
	\begin{split}
\dot x(t)\in& K[f](t,x(t))+\int^t_{t^0} \Phi(t,\tau) K[\tilde f](\tau,x(\tau)) d\tau. \vspace{-1mm}
\end{split}
\end{equation}
\end{definition}
\footnotetext{The integral of the set-valued function is understood in the sense of Aumann \cite{Aumann1965:JMAA}.}
If  $\Phi=\zero$ or $\tilde f$ is independent of $x$, the definition is equivalent to the definition of Filippov solution of discontinuous ODE.
\begin{theorem}\label{thm:sol_iODE}\itshape
If Assumption \ref{as:ex_IDE} is fulfilled, then
for any $x^0\in \R^n$ and any $t^0\in \R$,  the IVP \eqref{eq:iODE}, \eqref{eq:ch2_initialcondition} has a Filippov solution defined, at least, locally in the forward time. 
\end{theorem}
\begin{proof}  1) The set-valued mappings $K[f]$ and $K[\tilde f]$  are  nonempty-valued, convex-valued, compact-valued and upper semi-continuous  in $x$ for almost all $t$ (see, e.g., \cite[page 85]{Filippov1988:Book}).
Let us define the functions 
$f^{\eta},\tilde f^{\eta}:\R\times \R^{n}\to \R$  as follows\vspace{-1mm}
\[
	f^{\eta}(t,x)\!=\!\tfrac{1}{\eta}\!\!
	\int\limits_{\mathcal{B}(\eta)} \!\!\!f(t,x+y) dy, \;\;\;
	\tilde f^{\eta}(t,x)\!=\!\tfrac{1}{\eta}\!\!
	\int\limits_{\mathcal{B}(\eta)} \!\!\tilde f(t,x+y) dy,\vspace{-1mm}
\] 
where $\eta>0$.
In  \cite[Theroem 8, page 85]{Filippov1988:Book} it is shown that these functions satisfy Carath\'{e}odory Condition
and \vspace{-1mm}
\[
\begin{split}
	f^{\eta}(t,x)\in&F^{\eta}(t,x):=\overline{\co} \bigcup_{y\in \mathcal{B}(\eta),z\in \mathcal{B}(\eta)}K[f](t,x+y)+z,  \\
	\tilde f^{\eta}(t,x)\in &\tilde F^{\eta}(t,x):=\overline{\co} \bigcup_{y\in \mathcal{B}(\eta),z\in \mathcal{B}(\eta)}K[\tilde f](t,x+y)+z,  \\ 
\end{split} \vspace{-1mm}
\] 
for all $x\in \R^n$ and almost all $t\in \R$. 
By Theorem \ref{thm:Caratheodory_iODE}, the IDE\vspace{-1mm}
\begin{equation}\label{eq:IDE_eta}
\dot x^{\eta}(t)=f^{\eta}(t,x)+\int^t_{t^0} \Phi(t,\tau) \tilde f^{\eta}(\tau,x^{\eta}(\tau)) d\tau\vspace{-1mm}
\end{equation}
has a Carath\'{e}odory solution.  {By definition of Aumann's Integral \cite{Aumann1965:JMAA}, for almost all $t\in\mathcal{I}$}, this solution satisfies \vspace{-1mm}
\begin{equation}\label{eq:IDI_eta}
\dot x^{\eta}(t)\in F^{\eta}(t,x^{\eta}(t))+\int^t_{t^0} \Phi(t,\tau) \tilde F^{\eta}(\tau,x^{\eta}(\tau)) d\tau.\vspace{-1mm}
\end{equation}
2) 
Since $\Phi$ is locally essentially bounded on $\Delta$ then, for any $\alpha>0$ there exists $ \overline{\Phi}>0$ such that \vspace{-1mm}
\[
\esssup_{t_0\leq \tau\leq t\leq t_0+\alpha}\|\Phi(t,\tau)\|\leq \overline{\Phi}.\vspace{-1mm}
\]
Let $\bar \eta=\min\{1/4,1/(4\overline{\Phi}\alpha)\}$.
Since, {by Filippov Condition},  for  any $\alpha>0$ and  any $r>0$ there exist  a Lebesgue integrable  function $m:[t^0,t^0+\alpha] \to \R_+$ such that $\|f(t,x)\|\leq m(t)$ and  $\|\tilde f(t,x)\|\leq m(t)$  for almost all $t\in [t^0,t^0+\alpha]$ and almost all $x\in x^0+\mathcal{B}(r+\bar \eta)$  then (see, e.g., \cite[page 76]{Filippov1988:Book})\vspace{-1mm}
$$
K[f](t,x)\subset \mathcal{B}(m(t)), \quad K[\tilde f](t,x)\subset \mathcal{B}(m(t))\vspace{-1mm}
$$
for almost all $t\in \mathcal{I}=[t^0,t^0+\alpha]$ and for all $x\in x^0+\mathcal{B}(r+\bar \eta)$, and (see, e.g., \cite[page 82]{Filippov1988:Book})\vspace{-1mm}
$$
F^{\eta}(t,x)\subset \mathcal{B}(m(t)+\eta), \quad \tilde F^{\eta}(t,x)\subset \mathcal{B}(m(t)+\eta)\vspace{-1mm}
$$
for almost all $t\!\in\! \mathcal{I}$, all $x\!\in\! x^0+\mathcal{B}(r+\bar \eta)${, all $\eta\!\in\! (0,\bar \eta)$}. 
Hence, 
\vspace{-1mm}
$$
\|\dot x^{\eta}(t)\|\leq m(t)+\eta+\overline{\Phi}\int^t_{t^0}m(\tau)+\eta \,d\tau
\vspace{-1mm}$$
for almost  all $t\in \mathcal{I}$ and all $\eta\in (0,\bar \eta]$.
Since  $m$ is Lebesgue integrable on $\mathcal{I}$ then 
$M:=\int^{t^0+\alpha}_{t^0} m(\tau) d\tau<+\infty$ and 
for any $\epsilon>0$ there exist $$\delta\in (0, \epsilon \min \{1,1/(4\overline{\Phi}M)\})$$
such that 
for all $[a,b]\subset \mathcal{I}: |b-a|\leq \delta$ and 
all $\eta\in \left(0, \bar \eta \right)$  it holds\vspace{-1mm}
\[
\begin{split}
	\|x^{\eta}(b)-x^{\eta}(a)\|\leq& \left\|\int^b_{a}  {\dot x}^{\eta}(s) ds \right\| \leq \int^b_{a}  \|{\dot x}^{\eta}(s)\| ds\\
\leq & \int^b_{a} m(s)ds+(b-a) (\eta+\overline{\Phi}M +\overline{\Phi}\alpha \eta) \\
\leq &\frac{\epsilon}{4}+\delta \eta+\delta \overline{\Phi}M+\delta \overline{\Phi}\alpha \eta\leq \epsilon.\vspace{-1mm}
\end{split}
\]
This means that the family of function $\{x^{\eta}\}, \eta\in (0,\bar \eta]$ is equicontinuous and uniformly bounded for all $t\in \mathcal{I}$,\vspace{-1mm}
\[
\|x^\eta(t)\|-\|x_0\|\leq \int^t_{t_0}\|\dot x^{\eta}(\tau)\| d\tau\leq M +\alpha  \eta+\overline{\Phi}(M\alpha+\alpha^2 \eta), \vspace{-1mm}
\] so, by Arzel\`a--Ascoli theorem, it contains a sequence $x^{\eta_k}$ uniformly  convergent to a continuous function $x$ as $\eta_k\to \zero$.
The above inequalities imply that $x$ is absolutely continuous.

{3)} The set-valued mappings $F^{\eta}$ and $\tilde F^{\eta}$ are nonempty-valued, compact-valued, convex-valued and upper semi-continuous continuous in $x$ for almost all $t$ (see, e.g., \cite[page 86]{Filippov1988:Book}).
{
 For any solution $x^{\eta}$ of the IDE \eqref{eq:IDE_eta}, the set-valued mappings $t\to F^{\eta}(t,x^{\eta}(t))$ and $\tau\to \Phi(t,\tau)\tilde F^{\eta}(t,x^{\eta}(t))$ are compact-valued, convex-valued, measurable and bounded by the integrable function $t\mapsto 
\bar\Phi(m(t)+\eta)$. 
}	
By {\cite[Theorems 1, 2 and 4]{Aumann1965:JMAA}}\footnote{The requirement of the paper \cite{Aumann1965:JMAA} about Borel graph of the integrated function  has  been relaxed 
	(see, e.g., \cite[Chapter 3]{CastaingValadier1977:Book})
	using Kuratowski--Ryll-Nardzewski measurable selection theorem published in the same year \cite{KuratowskiRyll-Nardzewski1965:BAPC}.},
the right-hand side of \eqref{eq:IDI_eta} \vspace{-1mm}
\[
 Q^{\eta}(t):=F^{\eta}(t,x(t))+\int^t_{t^0} \Phi(t,\tau) \tilde F^{\eta}(\tau,x(\tau)) d\tau \subset \R^n\vspace{-1mm}
\]
is convex-valued and compact-valued {for almost all $t$.
Moreover, by construction, for almost all $t$ and for all $x$, we have  $K[f](t,x)\subset F^{\eta}(t,x)$,  $K[\tilde f](t,x)\subset \tilde F^{\eta}(t,x)$, $\forall \eta\in [0,\overline{\eta}]$, and  $F^{\eta}(t,x)\to K[f](t,x)$, $\tilde F^{\eta}(t,x)\to K[\tilde f](t,x)$  as $\eta \to 0$, where the limits are understood in the Hausdorff topology (see, e.g., \cite[\S 5]{Filippov1988:Book}).
By \cite[Theprem 5]{Aumann1965:JMAA}, for almost all $t\in \mathcal{I}$, we conclude that $Q^{\eta_{k}}(t)\to Q(t)$ as $\eta_k\to 0$,} 
where    \vspace{-1mm}
\[
Q(t)=K[f](t,x(t))+\int^t_{t^0} \Phi(t,\tau) K[\tilde f](\tau,x(\tau)) d\tau \subset \R^n\vspace{-1mm}
\]
is  a convex and compact set.
Hence, for any 
 $v\in \R^n$ and for almost all $t\in\mathcal{I}$ we have\vspace{-1mm}
{ $
v^{\top} \dot x^{\eta_k}(t) \leq \sup_{\ell \in Q^{\eta_{k}}(t)}v^{\top}\ell
 $ and }\vspace{-1mm}
\begin{equation}\label{eq:temp_proof}
 \phi(t):=\limsup_{\eta_k\to 0} v^{\top} \dot x_{\eta_k}(t)\leq \sup_{\ell \in Q(t)}v^{\top}\ell,\vspace{-2mm}
\end{equation}
Moreover, for any $[a,b]\subset \mathcal{I}$, it holds\vspace{-1mm}
\[
v^{\top}(x^{\eta_k}(b)-x^{\eta_k}(a))\!=\!\int^b_a v^{\top} \dot x^{\eta_k}(s) ds\!\leq\! \int^b_a \sup\limits_{j\geq k}v^{\top} x^{\eta_j}(s) ds.\vspace{-1mm}
\]
Since {$v^{\top}(x^{\eta_k}(b)-x^{\eta_k}(a))=\int^b_a v^{\top} \dot x^{\eta_k}(s) ds \to v^{\top}(x(b)-x(a))=\int^b_a v^{\top} \dot x(s) ds$ as $\eta^k\to 0$  and }$\sup_{j\geq k}v^{\top} x^{\eta_j}(s)$ does not increase as $k$ grows, then {we derive the inequality } \vspace{-1mm}
\[
\int^b_a v^{\top} \dot x(s) ds=v^{\top}(x(b)-x(a))\leq \int^b_a \phi(s) ds \vspace{-1mm}
\]
{by tending $\eta_k\!\to\! 0$.}
The segment $[a,b]$ is arbitrary selected, so\vspace{-1mm}
\[
v^{\top} \dot x(t)\le\sup_{\ell \in Q(t)}v^{\top}\ell, \quad \forall v\in \R^n,\vspace{-1mm}
\]
almost everywhere on $[t^0,t^0+\alpha]$. This  means $\dot x(t)\in Q(t)$ for almost all $t\in \mathcal{I}$. The proof is complete.
\end{proof}

Under some assumptions,  the integro-differential equation \eqref{eq:iODE} can be interpreted as a functional differential equation of retarded type  \cite{Hale1977:Book}, \cite{Fridman2014:Book}.
 Indeed, if, there exist $d>0$ such that $\Phi(t,\tau)=0$ for $t-\tau>d$, then
	$\dot x(t)$ depends only on $x(\tau)$ with $\tau\in [\max\{t^0, t-d\},t]$, $t\geq t^0$.
	Generalized solutions of  discontinuous retarded functional differential equations are introduced in  \cite[page 58]{KolmanovskiiMyshkis1992:Book} by means of a limiting procedure in a function space, which does not have such a simple formalization by the Filippov regularization \eqref{eq:Filippov_reg}. A study of links between the generalized solutions and the Filippov solutions is an interesting problem for the future research.

\begin{theorem}[On continuous dependence on parameters]\itshape\label{thm:continuity_iODE}
Let Assumption \ref{as:ex_IDE} be fulfilled 
	 and a family of functions $f_\mu, \tilde f_\mu$ be  parameterized by 
	$\mu\in \mathcal{M}$, where $\mathcal{M}$ is an open set of a complete metric space. Let $f=f_{\mu_0}$, $\tilde f=\tilde f_{\mu_0}$ for some $\mu_0\in \mathcal{M}$.
	Let  for any $\alpha>0$, any $r>0$ and any $\mu\in\mathcal{M}$, there exist $m_{\mu} \in L^{\infty}((-\alpha,\alpha);\R_+)$   such that 
 the set \vspace{-1mm}
	\begin{equation}\label{eq:set_E}
	\! \,\,E(t,x)\!=\!\left\{\!z\!\in\! \R^n: \!\begin{smallmatrix}
		\|x-z\|\leq m_{\mu}(t), \,
		\|f_\mu(t,x)-f(t,z)\|\leq m_{\mu}(t)\\
		 \text{ and }\|\tilde f_\mu(t,x)-\tilde f(t,z)\|\leq m_{\mu}(t)
	\end{smallmatrix}\!\right\}\!\!\!\vspace{-1mm}
\end{equation}
	is of positive measure for almost all $t\in(-\alpha,\alpha)$ and almost all $x\!\in\! \!B(r)$. 
	Let 
	$\sup_{\mu\in \mathcal{M}}\|m_{\mu}\|_{L^{\infty}}<+\infty$ and 
	$\|m_{\mu}\|_{L^1}\to 0$ as $\mu \to \mu_0$.
	 If, for a given $t^0\in\R$ and a given $x^0\in \R^n$, 	all Filippov solutions of the IVP \eqref{eq:iODE}, \eqref{eq:ch2_initialcondition}  exist on  $[t^0,t^0+T]$ then, 
	for any  $\varepsilon>0$, there exists $\delta\in (0,\varepsilon)$ such that 
	any Filippov solution $x_{\mu}$ of the IVP\vspace{-1mm}
	\begin{equation}\label{eq:iODE_mu}
		\dot x_{\mu}(t)\!=\!f_{\mu}(t,x_{\mu}(t))\!+\!\!\int^t_{t^0_{\mu}} \!\!\Phi(t,\tau) \tilde f_{\mu}(\tau,x_{\mu}(\tau)) d\tau, \;\;x_{\mu}(t^0_{\mu})\!=\!x^0_{\mu}\vspace{-1mm}
	\end{equation}
	with $t>t^0_{\mu}\in[t^0,t^0+\delta]$, $x_{\mu}(t^0_{\mu})=x^0_{\mu}\in x_0+\mathcal{B}(\delta)$   and 
	$\|m_{\mu}\|_{L^1(t^0,t^0+T)}\!\leq\! \delta$ differs from 
	some solution $x$ of the  IVP \eqref{eq:iODE}, \eqref{eq:ch2_initialcondition} for less than $\varepsilon$, i.e., $\sup\limits_{t\in [t^0+\delta,t^0+T]}\|x(t)-x_{\mu}(t)\|\leq \epsilon$.
\end{theorem}
\begin{proof}
	Let $\bar r>0$ be such that any Filippov solution of \eqref{eq:iODE}, \eqref{eq:ch2_initialcondition}
	belongs to the ball $x^0+\mathcal{B}(\bar r)$ for all $t\in[t^0,t^0+T]$.
	Let us study the case $t^0_{\mu}=t^0$ and define
\[
F^{*}_{\mu}(t,x)=\overline{\co} \bigcup_{y\in \mathcal{B}(m_{\mu}(t)),z\in\mathcal{B}(m_{\mu}(t))}K[f](t,x+y)+z
\]
\[
\tilde F^{*}_{\mu}(t,x)=\overline{\co} \bigcup_{y\in \mathcal{B}(m_{\mu}(t)),z\in\mathcal{B}(m_{\mu}(t))}K[\tilde f](t,x+y)+z.\vspace{-1mm}
\]
For almost all $t\in (t^0,t^0+T)$ and almost all $x\in x^0+\mathcal{B}(\bar r)$, by construction $K[f](t,x)\subset F^{*}_{\mu}(t,x)$, $K[\tilde f](t,x)\subset \tilde F^{*}_{\mu}(t,x)$, and
$K[f_{\mu}](t,x)\subset F^{*}_{\mu}(t,x)$, $K[\tilde f](t,x)\subset \tilde F^{*}_{\mu}(t,x)$ due to positive measure of $E(t,x)$ (see, \cite[page 98]{Filippov1988:Book} for more details).
Since $r_{\mathcal{M}}=\sup_{\mu \in \mathcal{M}}\|m_{\mu}\|_{L^{\infty}(t^0,t^0+T)}<+\infty$ then, due to Filippov Condition, 
there exists $m\in L^{1}((t^0,t^0+T),\R_+)$ such that $\sup_{y\in F^*_{\mu}(t,x)}\|y\|\leq m(t)+m_{\mu}(t)$ for almost all $t\in (t^0,t^0+T)$, for all $x\in x^0+\mathcal{B}(\bar r+r_{\mathcal{M}})$ and for all $\mu\in \mathcal{M}$.
Hence, any Filippov solution of \eqref{eq:iODE_mu}  satisfies \vspace{-1mm}
\begin{equation}
	\dot x_{\mu}(t)\in F^*_{\mu}(t,x_{\mu}(t))+\int^t_{t^0} \Phi(t,\tau)\tilde F^{*}_{\mu}(\tau,x_{\mu}(\tau))d\tau.\vspace{-1mm}
\end{equation}
and  $\|\dot x_{\mu}(t)\|\in m(t)+m_{\mu}(t)+\overline{\Phi}\int^t_{t^0} m(\tau)+m_{\mu}(\tau) \, d\tau$ almost everywhere,
where $\overline{\Phi}\geq 0$ is defined in the proof of Theorem \ref{thm:sol_iODE}. 
Hence,  for any $[a,b]\subset (t^0,t^0+T)$, we have\vspace{-1mm}
\[\vspace{-1mm}
\begin{split}
	\|x_{\mu}(b)\!-\!x_{\mu}(a)\|
	\!\leq\! & \left\|\int^b_{a}  {\dot x}^{\eta}(s) ds \right\| \leq \int^b_{a}  \|{\dot x}^{\eta}(s)\| ds\\
	\leq&\int^b_{a} \!\! m(s)ds\!+\!(b-a) (r_{\mathcal{M}}\!+\!\overline{\Phi}M\!+\!T\overline{\Phi} r_{\mathcal{M}}), 
\end{split}
\]
where $M=\int^{t^0+T}_{t^0} m(\tau) d\tau<+\infty$. This implies the equicontinuity of the family $\{x_{\mu}\}_{\mu\in\mathcal{M}}$  
of absolutely continuous functions (see, the step 2 of the proof of Theorem \ref{thm:sol_iODE} for more details).  Similarly, for all $t\in (t^0,t^0+T)$, we derive \vspace{-1mm}
\[\vspace{-1mm}
\begin{split}
\|x_{\mu}(t)\|\leq& \|x_0\|+\delta+\int^{t}_{t^0} \|\dot x_{\mu}(\tau)\| d\tau\\
\leq&  \|x_0\|+\delta+M +T (r_{\mathcal{M}}+\overline{\Phi}M\!+\!T\overline{\Phi} r_{\mathcal{M}}),
\end{split}
\]
i.e., the family $\{x_{\mu}\}_{\mu\in\mathcal{M}}$ is uniformly bounded.
 Since, by assumption, 
 $\|m_{\mu}\|_{L^1}\to 0$ as $\mu\to \mu_0$, then,
 
 for almost all $t$ and for all $x$, we have   $F^*_{\mu}(t,x)\to K[f](t,x)$, $\tilde F^*_{\mu}(t,x)\to K[\tilde f](t,x)$  as $\mu \to \mu_0$, where the limits are understood in the Hausdorff topology.
 By \cite[Theprem 5]{Aumann1965:JMAA}, for almost all $t\in \mathcal{I}$, we conclude that $Q^*_{\mu}\to Q(t)$ as $\mu\to \mu_{0}$,
where    \vspace{-1mm}
\[\vspace{-1mm}
\begin{split}
Q^*_{\mu}(t)=&F^*_{\mu}(t,x_{\mu}(t))+\int^t_{t^0} \Phi(t,\tau)\tilde F^{*}_{\mu}(\tau,x_{\mu}(\tau))d\tau\subset \R^n,\\
Q(t)=&K[f](t,x(t))+\int^t_{t^0} \Phi(t,\tau) K[\tilde f](\tau,x(\tau)) d\tau \subset \R^n
\end{split}
\]
are  convex and compact sets. Therefore,  for any uniformly convergent sequence of solutions $x_{\mu_i}\to x_{\mu_0}$ as $\mu_i\to \mu_0$, the limit $x_{\mu_0}$ is a Filippov solution of  the IVP \eqref{eq:iODE}, \eqref{eq:ch2_initialcondition}
(see, the step 3 of the proof of Theorem \ref{thm:sol_iODE} for more details).
Applying 
\cite[Lemma 6, page 10]{Filippov1988:Book} 
we complete the proof. 
\end{proof}

Continuous dependence of solutions  of IDEs on initial data (as in Theorem \ref{thm:con_sol_strong_Rn}) follows from the above result.

\vspace{-1mm}

\subsection{Uniqueness of Filippov  solutions of IDEs}

	For continuous $\tilde f$, the uniqueness analysis of Filippov solutions repeats the  analysis of Charath\'eodory  solutions.
	\begin{theorem}\itshape\label{thm:uniq_one_side_IDE}
		Under Assumption \ref{as:ex_IDE},
		if $f$ is locally one-sided Lipschitz  and $\tilde f$ is locally Lipschitz  then the IVP \eqref{eq:iODE}, \eqref{eq:ch2_initialcondition} has\\ a  unique solution defined, at least, locally in the forward time. 
	\end{theorem}
\begin{proof}
	If $x_1,x_2$ are two solutions of the IVP \eqref{eq:iODE}, \eqref{eq:ch2_initialcondition} then, for $e(t)=x_1(t)-x_2(t)$  and for almost all $t$, it holds  \vspace{-1mm}
	\begin{equation}\label{eq:proof_unique}\small
\begin{split}
		\!\!\!\!\!\tfrac{d \|e(t)\|^2}{dt}\!=&2e^{\!\top}\!(t)\! \left(\! f_1(t)\!-\! f_2(t)\!+\!\!\!\int^t_{t^0}\!\!\Phi(t,\tau) \left(\tilde f_1(\tau)\!-\! \tilde f_2(\tau)\right) d\tau\!\right)\\[-2mm]
	\leq& 2\ell(t)\|e(t)\|^2+2\bar \Phi \|e(t)\|\int^t_{t^0} \tilde \ell(\tau)\|e(\tau)\| d\tau\\[-2mm]
	\leq& 2\ell(t)\|e(t)\|^2\!+\!(\bar \Phi \|e(t)\|)^2\!+\!\left(\int^t_{t^0} \!\!\tilde \ell(\tau)\|e(\tau)\| d\tau\right)^2\\[-2mm]
	\leq &2\ell(t)\|e(t)\|^2+\bar \Phi^2\|e(t)\|^2+\|\tilde \ell\|_{L^2}^2\!\!\int^t_{t^0}\!\! \|e(\tau)\|^2 d\tau, 
	\end{split}\!\!\!\!\vspace{-1mm}
	\end{equation}
	where $f_1(t)\in K[f](t,x_1(t)), f_2(t)\in K[f](t,x_2(t))$, $\tilde f_1(\tau)=\tilde f(\tau,x_1(\tau))$, $\tilde f_2(\tau)=\tilde f(\tau,x_2(\tau))$, $\overline{\Phi}\geq 0$ is defined in the proof of Theorem \ref{thm:sol_iODE}. \textit{One-Sided  Lipschitz Condition} for $f$ (and, consequently, for $K[f]$, see \cite[page 106]{Filippov1988:Book}) with $\ell\in L^{\infty}_{\rm loc}$, \textit{Lipschitz Condition}  for $\tilde f$ with $\tilde \ell\in L^{\infty}_{\rm loc}$ and Cauchy--Schwartz inequality in $L^2((0,t);\R)$ have been utilized to derive \eqref{eq:proof_unique}. Since, $e(t^0)=0$ then the obtained  inequality implies $e(t)\!=\!0$ for $t\!>\!t^0$ (see \cite{Pachpatte1975:JMAA}).\!\!\!\!
\end{proof}

Let $F_{\mathcal{S}}(t,x,\gamma)\!=\! \left(K[f](t,x)+\gamma\right)\,\cap\,\mathcal{C}_{\mathcal{S}}(x)$, where $\gamma\in \R^n$, $\mathcal{S}\!\subset\!\R^n$ is discontinuity set of $f$  and $\mathcal{C}_{\mathcal{S}}$ is tangent cone for $\mathcal{S}$.
\begin{theorem}[\small Uniqueness of solution in sliding mode case]\label{thm:unique_SM_iODE}\hfill \newline\itshape
	Let Assumption \ref{as:ex_IDE} be fulfilled. Let    $\tilde {\mathcal{S}}\subset \R^n$ (resp., $\mathcal{S}\subset \R^n$) be a discontinuity set of the piece-wise 	(resp., one-sided) Lipschitz continuous  function $\tilde f$ (resp., $f$). Let a measurable function $(t,x,\gamma)\in \R\times \R^n\times \R^n \mapsto f^0(t,x,\gamma)\in \R^n$ be
 locally one-sided Lipschitz in $x$ and locally Lipschitz in $\gamma$.
	If \vspace{-1mm}
	\begin{itemize} 
	\item[1)]	$f$  satisfies
 Assumption \ref{as:uniq_1} and 
$\mathcal{S}\cap \tilde {\mathcal{S}}=\emptyset$;
\item[2)]  $\forall x\in \mathcal{S}$, $\exists \hat \epsilon\!>\!0$ :  $\left(K[f](t,x)\!+\!\gamma\right)\,\cap\,\mathcal{C}_{\mathcal{S}}(x)\!=\!\{f^0(t,x,\gamma)\}$ is a singleton for almost all $t\in \R$ and 
all $\gamma\in \mathcal{B}(\hat \epsilon)$;
\item[3)]  $K[f](t,x)\cap \mathcal{C}_{\tilde{\mathcal{S}}}(x)\!=\!\emptyset$ for all $x\!\in\! \tilde{\mathcal{S}}$ and almost all $t\!\in\! \R$;\vspace{-1mm}
\end{itemize}
then	the IVP  \eqref{eq:iODE}, \eqref{eq:ch2_initialcondition} has a Filippov solution being unique,  at least, locally in the forward time. 	
\end{theorem}
\begin{proof}
	Let us study the case $x^0\in \mathcal{S}$. 	This means that $x^0\in \partial G^j$ for all $j\in \mathcal{N}(x^0)$, where, as before, $\mathcal{N}:\mathcal{S}\to 2^{\{1,2,...,N\}}$ is the indicator function, i.e., $j\in \mathcal{N}(x_0)$ if $x_0\in \partial G^{j}$.
	Since the function $f$ is piece-wise continuous then  there exists $\bar \epsilon>0$ 
	such that $\mathcal{N}(x^0+\delta)\subset \mathcal{N}(x^0)$ for $x_0+\delta\in S: \|\delta\|\leq \bar \epsilon$.

	By 
	Assumption \ref{as:uniq_1}, the vector $f^j(t,x^0)$ belongs to the set $\R^n\backslash \mathcal{C}_{\bar  G^{j}}(x^0)$ 
	for all $t\in \R$.
	Since $\mathcal{C}_{\bar G^{j}}(x^0)$ is a positive closed cone then its complement $\R^n\backslash \mathcal{C}_{\bar G^{j}}(x^0)$ is a positive open cone.
	Hence, taking into account,
	$f^j\in C(\R\times \R^n;\R^n)$  we conclude that there exists $\epsilon^j>0$ such that \vspace{-1mm}
	$$
    f^j(t^0+s,x^0+\mathcal{B}(\epsilon^j))+\mathcal{B}(\epsilon^j) \subset  \left(\R^n\backslash {\mathcal{C}}_{\bar G^{j}}(x^0)\right), 	\; \forall s\in (0,\epsilon^j),\vspace{-1mm}
	$$
 where $\mathcal{B}(\epsilon^j)\subset \R^n$ is an open  ball of the radius $\epsilon^j>0$.
	
	Let $\epsilon =\min\left\{\bar \epsilon,\hat \epsilon, \min_{j\in \mathcal{N}(x^0)}\{\epsilon^j\}\right\}$.
	Repeating the proof of Theorem \ref{thm:unique_SM},  we conclude  that any solution 
	$x^{\epsilon}$
	of the IVP\vspace{-1mm}
	\begin{equation}\label{eq:ODI_temp_unique_SM}
		\dot x^{\epsilon}(t)\in K[f](t,x^{\epsilon}(t))+\mathcal{B}(\epsilon), \quad x^{\epsilon}(t^0)=x^0\in \mathcal{S}\vspace{-1mm}
	\end{equation}
	belongs to $\mathcal{S}$	as long as $x^{\epsilon}(t)\in x^0+\mathcal{B}(\epsilon)$. 
	For any continuous function $t\in [t^0,t^0+\epsilon]\to x(t)\in \R^n$, the set-valued function \vspace{-1mm}
	$$
	t\in [t^0,t^0+\epsilon]\mapsto \tilde Q(t):=\int^t_{t^0} \Phi(t,\tau) K[\tilde f](\tau,x(\tau)) d\tau\subset \R^n\vspace{-1mm}
	$$ is bounded as follows (see, the proof of Theorem \ref{thm:sol_iODE})\vspace{-1mm}
	\[
	\sup_{y\in \tilde Q(t)} \|y\|\leq \overline{\Phi}\int^t_{t^0}m(\tau) +\eta\, d\tau\vspace{-1mm}
	\]
	for some $\eta>0, \overline{\Phi}>0$ and some integrable function $m:[t^0,t^0+\epsilon]\to \R$. This means that $\exists t'\in (0,\epsilon)$ such that \vspace{-1mm}
	\[
	\sup_{y\in \tilde Q(t)} \|y\|\leq\epsilon, \quad \forall t\in [t^0,t^0+t'].\vspace{-1mm}
	\]
	Any solution the system \eqref{eq:iODE} with $x(t^0)=x^0\in \mathcal{S}$ is a solution of \eqref{eq:ODI_temp_unique_SM} for  $t\in[t^0,t^0+t']$. Since $\mathcal{S}\cap \tilde {\mathcal{S}}=\emptyset$  and  $\tilde f$ satisfies Lipschitz condition on any compact from $\mathcal{S}$  then  the IDE\vspace{-1mm}
	\begin{equation*}
\dot x(t)\!=\!f^0\left(t,x(t),\gamma(t)\right)+\gamma(t), \quad \gamma(t)\!=\!\int^t_{t^0}\!\!\!\Phi(t,\tau) \tilde f(\tau,x(\tau)) d\tau\vspace{-1mm}
	\end{equation*}
   defines the motion of the system on the sliding surface. 
  {For the obtained system, repeating the proof of Theorem  \ref{thm:uniq_one_side_IDE}, we derive the estimate similar to \eqref{eq:proof_unique}, which insures a local in-time uniqueness of solutions.  	For $x^0\in  G^j\backslash \tilde{\mathcal{S}}$, the local-in-time  uniqueness of solution  follows from  Theorem \ref{thm:uniqness_iODE}.  Finally, since $K[f](t,x)$ and $\mathcal{C}_{\mathcal{S}}(x)$ are compact sets  for all $x$ and for almost all $t$, then the condition 3) implies that 
  	$\forall x\in \tilde {\mathcal{S}}$, $\exists \hat \epsilon\!>\!0$ such that  $\left(K[f](t,x)+\gamma \right)\cap \mathcal{C}_{\tilde{\mathcal{S}}}(x)=\emptyset$  for almost all $t\in \R$ and 
  	all $\gamma\in \mathcal{B}(\hat \epsilon)$.
  	So, any solution  initiated on $\tilde S$ immediately enters into  domain $\R^n\backslash\mathcal{S}$, i.e., the estimate \eqref{eq:proof_unique} holds almost everywhere and solution is unique, at least, locally in time.}
\end{proof}

Theorem \ref{thm:unique_SM_iODE} asks $\tilde f$ to be continuous on the discontinuity set of $f$. For IDEs with  sliding  modes, this condition cannot be violated to ensure the uniqueness of Filippov solutions.
\begin{example}\itshape\label{ex:5}
	A Filippov solution of the IVP\vspace{-1mm}
	\begin{equation}\label{eq:ex_5}\vspace{-1mm}
	\left\{
	\begin{split}
		\dot x_1(t)=&\;u(x(t)),  	\quad\quad\quad \quad\quad \quad \quad \quad \quad x_1(0)=0,\\[-1mm]
		\dot x_2(t)=&\int^t_{0} \sin(t+\tau)u(x(\tau)) d\tau, \quad\quad  x_2(0)=0,\\[-1mm]
		u(x)=&-\sign(x_1), \quad \quad \quad \;\;x=(x_1,x_2)^{\top}\in \R^2
	\end{split}
	\right.
	\end{equation}
 is  $x_1(t)=0$ and $x_2(t)=\int^t_0\int^s_0
  \sin(s+\tau)q(\tau) d\tau ds$, where $q:\R\mapsto [-1,1]$ is any measurable function. 
\end{example}
The IDE \eqref{eq:iODE} without sliding mode may have unique solutions if an  intersection of  discontinuity sets of $f$ and $\tilde f$ is non-empty.
\begin{theorem}[\small Uniqueness of solution in switching case]\label{thm:unique_no_SM_iODE}\hfill \newline\itshape
		Let Assumption \ref{as:ex_IDE} be fulfilled. Let    $\tilde {\mathcal{S}}\subset \R^n$ (resp., $\mathcal{S}\subset \R^n$) be a discontinuity set of the piece-wise  (resp., one-sided) Lipschitz continuous  function $\tilde f$ (resp., $f$). If 
	 $\tilde {\mathcal{S}}\subset \mathcal{S}$ and
	 $f$  satisfies Assumption \ref{as:uniq_2} then
	the IVP  \eqref{eq:iODE}, \eqref{eq:ch2_initialcondition} has a Filippov solution being unique,  at least, locally in the forward time. 	
\end{theorem}	
\begin{proof}
	For $x_0\notin \mathcal{S}$,  Filippov solution of  the IVP \eqref{eq:iODE}, \eqref{eq:ch2_initialcondition} coincides Carath\'eodory solution as long as $x(t)\notin \mathcal{S}$. This solution is unique due to Theorems \ref{thm:Caratheodory_iODE} and  \ref{thm:uniqness_iODE}.
	Since $K[f](t,x)$ and $\mathcal{C}_{\mathcal{S}}(x)$ are compact sets  for all $x$ and for almost all $t$, then Assumption \ref{as:uniq_2} implies that for any $x\in \mathcal{S}$ there exists $\hat \epsilon>0$ such that $K[f](t,x)+\gamma\cap \mathcal{C}_{\mathcal{S}}(x)=\emptyset$ for all $\gamma\in \R^n: \|\gamma\|\leq \hat \epsilon$. A Filippov solution $x$ of the IVP \eqref{eq:iODE}, \eqref{eq:ch2_initialcondition} initiated 
	at $x^0\in \mathcal{S}$ immediately enters in 
	a domain $G^{j}$, where the solution is unique   due to Theorems \ref{thm:Caratheodory_iODE} and \ref{thm:uniqness_iODE}.
\end{proof}	

The corollary given below follows from above theorems.
\begin{corollary}\label{cor:unique_mix_iODE}\itshape
		Let Assumption \ref{as:ex_IDE} hold. Let    $\tilde {\mathcal{S}}\subset \R^n$ (resp., $\mathcal{S}\subset \R^n$) be a discontinuity set of the piece-wise 	 (resp., one-sided) Lipschitz continuous  function $\tilde f$ (resp., $f$). Let a measurable function $(t,x,\gamma)\in \R\times \R^n\times \R^n \mapsto f^0(t,x,\gamma)\in \R^n$ be
	locally one-sided Lipschitz in $x$, locally Lipschitz in $\gamma$.
	If \vspace{0mm}
		\begin{itemize} 
		\item[1)]	$f$  satisfies
		 Assumption \ref{as:uniq_3} with
		$\mathcal{M}\subset \mathcal{S} : \mathcal{M}\cap \tilde {\mathcal{S}}=\emptyset$;
		\item[2)]  $\forall x\!\in\! \mathcal{M}$, \!$\exists \hat \epsilon\!>\!0$:\!  $\left(K[f](t,x)\!+\!\gamma\right)\,\!\cap\,\mathcal{C}_{\mathcal{M}}(x)\!=\!\{f^0(t,x,\gamma)\}$ is a singleton for almost all $t\in \R$ and 
		all $\gamma\in \mathcal{B}(\hat \epsilon)$;
		\item[3)]  $K[f](t,x)\cap \mathcal{C}_{\tilde{\mathcal{S}}}(x)\!=\!\emptyset$ for all $x\!\in\! \tilde{\mathcal{S}}$ and almost all $t\!\in\! \R$;\vspace{0mm}
	\end{itemize}
then
	the IVP  \eqref{eq:iODE}, \eqref{eq:ch2_initialcondition} has a Filippov solution being unique,  at least, locally in the forward time. 	
\end{corollary}

{
		Theorems \ref{thm:unique_SM_iODE},  \ref{thm:unique_no_SM_iODE} and Corollary \ref{cor:unique_mix_iODE}  prove only a local-in-time uniqueness of solutions. 
		For ODEs, due to the \textit{semigroup property of solutions} (see, e.g., \cite{CoddingtonLevinson1955:Book}), the local-in-time uniqueness implies  the uniqueness on the whole interval of the existence of solution. For IDEs, the semigroup property is not valid, so, a Filippov solution of IDE may be unique on a short interval of  time and non-unique on a long interval of time.   Utkin solutions introduced below solve this problem in some cases. 	
}
\subsection{Equivalent control method for systems modeled by IDEs}
Let us consider a control system modeled by the IDE \eqref{eq:iODE_control} with
measurable functions $f$ and $\tilde f$ being continuous in $(x,u)\in \R^n\times \R^m$ for almost all $t\in \R$.
\begin{definition}[Utkin solution of IDE]\label{def:Utkin_sol_iODE}
\itshape An absolutely continuous function $x: \mathcal{I}\to \R^n$  is said to be an Utkin solution of the first kind (resp., of the second kind) of the IVP  \eqref{eq:iODE_control}, \eqref{eq:ch2_initialcondition} if $x(t^0)=x^0$ and there exists a measurable function $u_{\rm eq}:\mathcal{I}\to \R^m$ such that $u_{\rm eq}(t)\in U(t,x(t))$ and\vspace{-1mm}
\begin{equation}\label{eq:iODE_Utkin}
\!\,	\dot x(t)\!=\! f(t,x(t), u_{\rm eq}(t))\!+\!\!\int^t_{t^0}\!\!\!\Phi(t,\tau)\tilde f(\tau,x(\tau),u_{\rm eq}(\tau)) d\tau, \!\vspace{-1mm}
\end{equation}
almost everywhere on $\mathcal{I}$, where $U=K[u]$ is defined by \eqref{eq:Filippov_reg}
(resp., $U=(K[u_1], ..., K[u_m])^{\top}$ for $u=(u_1, u_2,...,u_m)^{\top}$).
\end{definition}

Similarly to ODEs,  
the set of Utkin solutions of the first  kind is a subset of the set of the Utkin solutions of the second kind, since $K[u](t,x(t))\subset (K[u_1](t,x(t)),...,K[u_m](t,x(t)))^{\top}$. They always coincide for single input systems
(i.e., if $m=1$).
Utkin solutions always exist for affine-in-control IDEs.
\begin{theorem}\label{thm:sol_iODE_Utkin}\itshape 
 Let $f$ and $\tilde f$ in \eqref{eq:iODE_control} have the form\vspace{-1mm}
\begin{equation*}
	f(t,x,u)\!=\!a(t,x)\!+\!b(t,x)u,\quad  \tilde f(t,x,u)\!=\!\tilde a(t,x)\!+\!\tilde b(t,x)u,\vspace{-1mm}
\end{equation*}
where $t\in \R, x\in \R^n, u\in \R^m$ and 
measurable functions $a,\tilde a:\R\times \R^n \to \R^n$, $b,\tilde b: \R\times \R^n\to \R^{n\times m}$ are locally essentially bounded and continuous in $x$.
If  a measurable function $u:\R\times\R^n\to \R^n$ is locally essentially bounded then,
for any $t^0\in \R$ and any $x^0\in \R^n$,
the IVP \eqref{eq:iODE_control},\eqref{eq:ch2_initialcondition}  has an Utkin solution defined, at least, locally in the forward time.
For the affine-in-control system, any Utkin solution of the first kind  is a Filippov solution.
\end{theorem}
\begin{proof}
{
	Let us denote $f^{\eta}\!=\!a+bu^{\eta}, \tilde f^{\eta}\!=\!\tilde a+\tilde bu^{\eta}$, \vspace{-1mm}
	\[
   u^{\eta}(t,x)\!=\!\tfrac{1}{\eta}\int_{\|y\|\leq \eta} \!\!\!\!\!\!u(t,x+y) dy, \vspace{-1mm}
	\]
\[
U^{\eta}(t,x)=\overline{\co} \bigcup_{y,z\in \mathcal{B}(\eta)}  U(t,x+y)+z,\vspace{-1mm}
\]
\[
\left( 
\begin{smallmatrix}
	F^{\eta}(t,x)\\
	\tilde F^{\eta}(t,x)
\end{smallmatrix}
\right)=  \bigcup_{\hat u\in U^{\eta}(t,x)} \left( 
\begin{smallmatrix}
	a(t,x)+b(t,x)\hat u\\
	\tilde a(t,x)+\tilde b(t,x)\hat u
\end{smallmatrix}
\right),\vspace{-1mm}
\]
where $U=K[u]$ or $U=(K[u_1],....,K[u_m])$.
By construction, $u^{\eta}(t,x)\in U^{\eta}(t,x)$, $f^{\eta}(t,x)\in F^{\eta}(t,x)$, $\tilde f^{\eta}(t,x)\in\tilde F^{\eta}(t,x)$ for almost all $t$ and all $x$. Since $a,\tilde a,b,\tilde b$ are continuous in $x$ then, for any $\eta>0$, the functions $f^{\eta}$ and $\tilde f^{\eta}$ satisfy Caratheodory Condition and the set-valued mappings   
$F^{\eta}$ and $\tilde F^{\eta}$ are compact-valued, convex-valued and upper-semicontinuous in $x$. 
Let us denote $\xi^{\eta}(t)=\int^t_{t^0} u^{\eta}(\tau,x^{\eta}(\tau)) d\tau$.
Repeating the step 2 of the proof of Theorem \ref{thm:sol_iODE} 
we conclude that the family of pairs $(x^{\eta},\xi^{\eta})$ 
is uniformly bounded and equicontinuous on some time interval $\mathcal{I}$ and  there exists a sequence $\eta_i :\lim_{i\to \infty}\eta_i=0$ and a pair of absolutely continuous functions $(x,\xi)$ such that $(x^{\eta_i},\xi^{\eta_i})\to (x,\xi)$ as $i\to +\infty$ uniformly on $\mathcal{I}$.
 Since $a,\tilde a,b,\tilde b$ are continuous in $x$, then, repeating the step 3) of the proof of Theorem \ref{thm:sol_iODE} we derive  $\dot \xi(t)\in  U(t,x(t))$ and 
 $\dot x(t)\in   a(t,x(t))+b(t,x(t))U(t,x(t))+ \int^t_{t^0}\Phi(t,\tau) \tilde a(\tau,x(\tau)) d\tau +\int^t_{t^0}\Phi(t,\tau) b(\tau,x(\tau))\dot \xi(\tau)d\tau$
for almost all $t\!\in\! \mathcal{I}$. So, $x$ is an Utkin solution with  $u_{\rm eq}\!=\!\dot \xi$.}
Due to continuity  of $f$ and $\tilde f$ in $x$, it holds  
$ K[f]=a+bK[u]
$ and  
$K[\tilde f]=\tilde a+\tilde bK[u]$.
So, any Utkin solution of the first kind is a Filippov solution.
\end{proof}	

  A set of Filippov solutions of the affine-in-control { discontinuous IDE} may be larger than a set of Utkin solutions of the first kind. This is a specific feature of IDEs. For affine-in-control discontinuous ODEs ($\Phi=\zero$), these sets are identical. 
\addtocounter{example}{-1}
\begin{example}[Continued]\itshape 
	The IVP \eqref{eq:ex_5} has the unique Utkin solution of the first kind: $x_1(t)=x_2(t)=0,\forall t\geq 0$ with the equivalent control $u_{\rm eq}(t)=0$. 
\end{example}

At the first sight, in view of the above example, the Utkin solutions of the first kind seems to be more appropriate 
for control systems design. Indeed, the integral part in IDE \eqref{eq:iODE_control} contains the control signal defined in past, i.e., for $\tau \leq t$. At time $t$,  this control is already determined, so a non-uniqueness of  a control signal in past  looks artificial. However,  Utkin solutions of the second kind (as well-as Filippov solutions) may correspond to a more realistic  dynamics of a MIMO system with  inputs generated by independent actuators.
 \begin{example}\itshape 
 The Utking solution of the first kind  for\vspace{-1mm}
\begin{equation}\label{eq:ex_7}\small
	\left\{
	\begin{split}
		\dot x_1(t)\!=&u_1(x_{1}(t)),  	\quad\quad\quad \quad\quad \;\;\; x_1(0)=0,\\[-1mm]
		\dot x_2(t)\!=&\smallint^t_{0} u_2(x_{1}(\tau)) d\tau, \quad\quad  x_2(0)=0,\\[-1mm]
		u_i(x)\!=&-\sign(x_1), \quad i\!=\!1,2, \quad x\!=\!(x_1,x_2)^{\top}\!\in\!\R^2
	\end{split}
	\right.\vspace{-1mm}
\end{equation}
is unique and given by $x_1(t)=x_2(t)=0,\forall t\geq 0$. 
The set of Filippov solutions of the IVP \eqref{eq:ex_7} 
coincides with the sets of Utkin solutions of the second kind of the IVP \eqref{eq:ex_7}  and 
coincides with the set of Filippov solutions of the IVP  \eqref{eq:ex_5} (see, Example \ref{ex:5}). If control inputs $u_1$
and $u_2$ are assumed to be independent (e.g., realized by different actuators), then, in practice, the input signals cannot be identical. In this particular case, the Utkin solution of the second kind provides a more realistic description of the system motion than  the Utkin solution of the first kind.
\end{example}

{Utkin solutions of the first kind for an affine-in-control IDE with a smooth sliding surface may be unique globally in time.}

	\begin{corollary}[On equivalent control on smooth surface]\itshape\label{eq:uniq_Utkin_IDE}
		Let $\mathcal{S}=\{x\in \R^{n}:s(x)=0\}, s\in C^1(\R^n;\R^m)$ be the discontinuity  set of
		the piece-wise  Lipschitz continuous function $u:\R\times \R^n\to \R^m$. Let conditions of  Theorem \ref{thm:sol_iODE_Utkin} be fulfilled, the matrix $G(t,x)=\frac{\partial s(x)}{\partial x}b(t,x)$ be non-singular for almost all $t\in \R$ and for all $x\in \R^n$ and $(t,x)\to G^{-1}(t,x)$ be locally essentially bounded on $\R\times \R^{n}$.\vspace{-1mm}
		\begin{itemize}
		\item[i)]  If $t\in [t^0,t^1] \mapsto x(t)\in \R^n$ with $\mathcal{I}=[t^0,t^1]$ is an Utkin solution of  \eqref{eq:iODE_control}, \eqref{eq:ch2_initialcondition}  such that $x(t)\notin \mathcal{S}$ for $t\in [t^0,t^*]$ and $x(t)\in \mathcal{S}$ for $t\in[t^*,t^1]$ then the equivalent control $u_{\rm eq}:[t^0,t^1]\to \R^m$  is given by \vspace{-1mm}
		\begin{equation}\small
		u_{\rm eq}(t) =\left\{
		\begin{array}{lcc}
			u(t,x(t)) & \text{ for } & t\in[t^0,t^*),\\
			\sum^{\infty}_{i=0}(\mathcal{K}^{i}g)(t) & \text{ for } & t\in[ t^*,t^1],
			\end{array}
		\right.\vspace{-1mm}
		\end{equation}
	where  $\mathcal{K}:L^{\infty}((t^*,t^1);\R^m)\to L^{\infty}((t^*,t^1);\R^m) $ is a linear bounded operator \vspace{-1mm}
	\begin{equation}\small 
	(\mathcal{K}\tilde u)(t)\!=\!\!\int^t_{t^*}\!\!\!G^{-1}(t,x(t))\Phi(t,\tau) \tilde b(\tau,x(\tau)) \tilde u(\tau) d\tau,\vspace{-1mm}
	\end{equation}
	and $g\!\in\! L^{\infty}((t^*,t^1);\R^n)$ is as follows\vspace{-1mm}
					\begin{equation*}\small 
			\!\!g(t)\!=\!- G^{-1}(t,x(t))\left.\tfrac{\partial s(y)}{\partial y}\right|_{y=x(t)} \!\left(a(t,x(t))\!+\!\tilde g(t)\right),\vspace{-2mm}
		\end{equation*}
	\begin{equation*}\small 
			\tilde g(t)\!=\!\int^{t^*}_{t^0}\!\!\!\!\!\Phi(t,\tau) \left(a(\tau,x(\tau)) \!+\!b(\tau,x(\tau)) u(t,x(\tau)\right) d \tau.\vspace{-1mm}
	\end{equation*}
			
            \item[ii)] If $a,b,\tilde a,\tilde b, \frac{\partial s}{\partial x}$ 
             are locally Lipschitz continuous  and $$G(t,y_1)=G(t,y_2),\tilde b(t,y_1)=\tilde b(t,y_2),  \quad \forall y_1,y_2\in \mathcal{S}$$
            then any Utkin solution of   \eqref{eq:iODE_control}, \eqref{eq:ch2_initialcondition} belonging to $\mathcal{S}$ is unique. 
					\end{itemize} 
	\end{corollary}
\begin{proof}
	i)  If $x(t)\!\in\! \mathcal{S}$ then $\frac{d s(x(t))}{dt}\!=\!\left.\frac{\partial s(y)}{\partial y}\right|_{y=x(t)}\dot x(t)\!=\!0$ and, taking into account,\vspace{-1mm}
		\begin{equation}\label{eq:IDE_affine_SMC_u_eq}\small
	\begin{split}
		\dot x(t)\!=&a(t,x(t))+b(t,x(t))u_{\rm eq}(t)\\[-1mm]
		&+\!\int^t_{t^0}\!\!\!\Phi(t,\tau)\left(\tilde a(\tau,x(\tau))\!+\!\tilde b(\tau,x(\tau))u_{\rm eq}(\tau) \right)d\tau,\vspace{-2mm}
	\end{split}
	\end{equation}
	we conclude that, in the sliding mode, $u_{\rm eq}$ is the unique solution of  
	the integral Volterra equation of the second kind\vspace{-1mm}
	\begin{equation}\label{eq:u_eq_Volterra}\small 
		u_{\rm eq} (t)\!=\!g(t)\!+\!(\mathcal{K}u_{\rm eq})(t), \quad t\in [t^*,t^1].\vspace{-1mm}
	\end{equation}
This integral  equation has a unique solution since \vspace{-1mm}
	 \begin{equation*}\small
M\!=\!\!\esssup_{t^*\leq \tau\leq t\leq t^1}\| G^{-1}(t,x(t))\Phi(t,\tau) \tilde b(\tau,x(\tau))\| d\tau\!<\!+\infty.\vspace{-2mm}
	\end{equation*}
	Indeed, since $\|\mathcal{K}^n\|_{\mathcal{L}(L^{\infty};L^{\infty})}\leq \frac{M^n(t^1-t^*)^n}{n!}$, then  
	the Neumann series $\sum^{\infty}_{i=0}\mathcal{K}^{i}$ is convergent in the  Banach space $\mathcal{L}(L^{\infty};L^{\infty})$ of bounded linear operators $L^{\infty}\to L^{\infty}$, and  $u_{\rm eq}=\sum^{\infty}_{i=0}\mathcal{K}^{i} g$ is the solution of the integral equation \eqref{eq:u_eq_Volterra}.\\
	ii) For $t\in[t^*,t^1]$, 
	  the dynamics  is given by \eqref{eq:IDE_affine_SMC_u_eq}.
	Let $x_1(t), x_2(t)$ satisfy the above identity for $t\geq t^0$, $x_1(t)=x_2(t)=x(t)\notin \mathcal{S}$ for $t^0\leq t< t^*$ and $x_1(t)\in \mathcal{S}, x_{2}(t)\in \mathcal{S}$ for $t\in [t^*,t^1]$. 
	By assumption, the operator $\mathcal{K}$ is independent of $x(t)$.
For $j=1,2$, let $g_i$ be defined by replacing $x$ with  $x_j$ in the formula for $g$, respectively. Notice than $\tilde g$ is identical for both $x_1$ and $x_2$.
	Then, using Lipschitz continuity  in $x$ and local essential boundedness in $t$ of the functions  $a,b,\tilde a,\tilde b, \frac{\partial s}{\partial x}$,  for $e(t)=x_1(t)-x_2(t)$, we derive
	$
	\|g_1(t)-g_2(t)\|\leq \ell_1 \|e(t)\|
	$
	and 
	$
	\|Kg_1(t)-Kg_2(t)\|\leq \ell_2\int^t_{t^*}\|e(\tau)\| d\tau\leq \ell_2\sqrt{t-t^*}\|e\|_{L^2}
	$
	for some $\ell_1, \ell_2\!>\!0$, where $\|e\|^2_{L^2(t^*,t)}=\int^t_{t^*} e(\tau) d\tau$. Moreover, for $i\geq 1$, it holds
	$
	\|(K^ig_1)(t)-(K^ig_2)(t)\|\leq \tfrac{(2\ell_2)^i(t-t^*)^{i}}{(2i-1)!\sqrt{t-t^*}}\|e\|_{L^2}.$	
Therefore, for some $\ell_3,\ell_4,\ell_5\!>\!0$ and for all $t\!\in\! [t^*,t^1]$ we have \vspace{-3mm}
	\[\vspace{-1mm}\small
\begin{split}
	\tfrac{1}{2}\!\tfrac{d \|e(t)\|^2}{dt}\!\leq&\ell_3\|e(t)\|^2\!+\ell_4\|e(t)\|\sum_{i=1}^{\infty} \left\|\mathcal{K}^i_1g_1(t)\!-\!\mathcal{K}^i_2g_2(t) \right\|\\[-1mm]
	&\!\!\!+\!\ell_5 \|e(t)\|\!\int^t_{t^*}\!\! \|e(\tau)\|\!+ \!\!\sum_{i=1}^{\infty}\! \left\|\mathcal{K}^i_1g_1(\tau)\!-\!\mathcal{K}^i_2g_2(\tau) \right\| d\tau\\[-1mm]
	\leq& \ell_3\|e(t)\|^2+ \ell(t-t^*)\|e(t)\| \cdot \|e\|_{L^2(t^*,t)},\\
\end{split}
\]
where $s\mapsto \ell(s)\!=\!\ell_5\sqrt{s}\!+\!(\ell_4\!+\!\ell_5)
\sum_{i=1}^{\infty}\tfrac{(2\ell_2)^is^{i}}{(2i-1)!\sqrt{s}}$.
Since   $\sum_{i=1}^{\infty}\tfrac{(2\ell_2)^is^{i}}{(2i-1)!\sqrt{s}}=\sqrt{2\ell_2}(\exp(\sqrt{2\ell_2 s})-1)$  and $e(t^*)=0$ then the obtained estimate yields $e(t)\!=\!0,\forall t\in [t^*,t^1]$.
	\end{proof}
\begin{table}[b]
	\centering
		\caption{Types of solutions of the IDE \eqref{eq:iODE_control} and conditions to $f,\tilde f,u$}
	\begin{tabular}{c|c|c|c}
		\textbf{Solution} & \!\!\textbf{Class of $f,\tilde f$}\!\! & \textbf{Class of $u$} & \textbf{Uniqueness}\\ 
		\hline
		Classical    & continuous &  continuous& Lipschitz\\
		\cite{Lakshmikantham_etal1995:Book}& in $(t,x,u)$ & in $(t,x)$ &  \\
		\hline
		\!\!\!Carath\'eodory\!\!
		& \!\!discontin. in $t$;\!\! & \!\!discontin. in $t$;\!\!&    Lipschitz\\
		\cite{OReganMeehan1998:Book}  & \!\!contin. in $(x,u)$\!\!  & \!\!continuous in $x$\!\! &  \\
		\hline
		Filippov & \!\!discontinuous\!\! & discontinuous & piece-wise Lipschitz;\\
		 & in $(t,x,u)$ & in $(t,x)$ & \!$\tilde f$ is independent of $u$\\
		&&& \!\! and has no common \\
		& & & discontinuity  with $f$\\
		\hline
		Utkin & \!\!discontin. in $t$;\!\!& discontinuous & piece-wise Lipschitz;  \\
		& continuous&  in $(t,x)$ &  \!\!affine-in-control system;\!\!\!\\
		& in $(x,u)$&&\!smooth {and invariant}  \\
		& & & sliding surface
	\end{tabular}
	\label{tab:sol_IDE}
	\vspace{-3mm}
\end{table}

The table \ref{tab:sol_IDE} summarizes 
classes of solutions of the IDE \eqref{eq:iODE_control} and some conditions to its right-hand side, which are required to ensure existence/uniqueness.

\begin{remark}[On explicit Euler method for IDEs]\label{rem:1}
	
	If $f,\tilde f$ are piece-wise locally Lipschitz continuous functions in $x$ then approximate (numerical solution) of the IDE  \eqref{eq:iODE} can be found using the explicit Euler method  \vspace{-1mm}
	\begin{equation}\label{eq:explicit_Euler_IDE}
		\vspace{-1mm}
		\begin{split}
			x_{k+1}=&x_{k}+hf(t^0+kh,x_{k})+hI_k,\\
			I_k=&h\sum^{k}_{i=0} \Phi(t^0+kh,t^0+ih)\tilde f(t^0+ih,x_i), 
		\end{split}
	\end{equation}
	where $h>0$ is the sampling period, $k=0,1,2,...$, $x_k$ approximates $x(t^0+kh)$ and
	$I_k$ approximates the integral $\int^t_{t^0}\Phi(t,\tau)\tilde f(\tau,x(\tau)) d\tau$ by means of the method of rectangles. Indeed, since, on continuity domains, both functions $f,\tilde f$ are locally  Lipschitz continuous in $x$, any solution $x:\mathcal{I}\to \R^n$ of IDE  \eqref{eq:iODE} belonging to  this domain is unique and  the sequence of the Euler approximates
	\[
	x_{h}(t)\!=\!\tfrac{t^0\!+(k+1)h-t}{h}x_k+\tfrac{t-t^0\!-kh}{h}x_{k+1}, t\!\in\![t^0\!+kh,t^0\!+(k+1)h], 
	\]
	(being uniformly bounded and equicontinuous on $\mathcal{I}$) converges to $x$ uniformly on $\mathcal{I}$ as $h\to 0$. This conclusion remains valid  if the discontinuous IDE  \eqref{eq:iODE} has no sliding mode.
	
	In the sliding mode case, similarly to ODEs,
	the explicit Euler method correctly approximates the Filippov/Utkin solution of the IDE away an $\epsilon$-neighborhood of a sliding set $\mathcal{S}$, where $\epsilon\to 0$ as $h\to 0$.      
	A consistent discrete-time approximation of the motion on the sliding surface requires development
	of advanced numerical schemes \cite{AcaryBrogliato2010:SCL}, \cite{Polyakov_etal2019:SIAM_JCO}. This goes out of the scope of the paper.
\end{remark}
\section{Application of discontinuous IDEs for SMC design and analysis}\label{sec:SMC_design}
\subsection{SMC for a distributed  input delay system}
Let us consider a control system modeled by the IDE\vspace{-1mm}
\begin{equation}\label{eq:LTI_IDE}
	\begin{split}
		\dot x(t)=&Ax(t)+B(u(t)+\gamma(t))+p(t)\\
		&
		+\int^t_{t^0} \Phi(t,\tau) \tilde B (u(\tau)+\gamma(\tau))d
		\tau,
	\end{split}\vspace{-1mm}
\end{equation} 
where, $t>t_0\in \R, x(t)\in \R^n$, $u(t)\in \R^{m}$, $\Phi :\Delta\to \R^{n\times n}$ be locally essentially bounded on $\Delta$, $\gamma\in L^{\infty}(\R;\R^m)$ is a matched perturbation, $p\in L^{\infty}(\R;\R^n)$ is a possibly mismatched perturbation 
and $A\in \R^{n\times n}, B,\tilde B\in \R^{n\times m}$ are  known matrices. This IDE can be considered as a linear system with a  \textit{distributed input delay}. SMC design for a linear  system with a \textit{distributed state delay} can be found in \cite{Gouaisbaut_etal2001:TDS}. The aim is to stabilize the output \vspace{-1mm}
\begin{equation}\label{eq:LTI_y}
	y(t)=Cx(t), \quad C\in \R^{m\times n}, \quad \det (CB)\neq0\vspace{-1mm}
\end{equation}
to zero in a finite time. 
For  ODEs (i.e. for $\Phi=\zero$), the SMC theory (see, \cite{Utkin1992:Book}, \cite{EdwardsSpurgeon1998:Book}, \cite{Shtessel_etal2014:Book}) suggests  the  feedback\vspace{-1mm}
\begin{equation}\label{eq:SMC_LTI}
	u(t)=-\rho(CB)^{-1}\frac{y(t)}{\|y(t)\|_{P}}, \quad \rho >0,\vspace{-1mm}
\end{equation}
 where $\|y\|_{P}=\sqrt{y^{\top}Py}$ is the weighted Euclidean norm with a positive definite matrix $P\succ 0$. Let $\|Q\|_{P}=\sup_{y\neq 0} \frac{\|Qy\|_{P}}{\|y\|_{P}}$, $Q\in \R^{m\times m}$ be the matrix norm is induced by the weighted Euclidean norm $\|\cdot\|_{P}$.
 {The condition $\det(CB)\neq 0$ means that the output has the relative degree 1 with respect to the input. This is the standard  assumption for the first order sliding mode control design \cite{Utkin1992:Book}, \cite{Utkin_etal2020:Book}.
If the output has a higher relative degree with respect to the input, then, to solve the considered regulation problem, the high order sliding mode control design algorithms  \cite{Levant2003:IJC}, \cite{Shtessel_etal2014:Book} must be extended  to the IDE \eqref{eq:LTI_IDE}. This extension goes out of the scope of the paper.}

The single-input closed-loop system \eqref{eq:LTI_IDE}, \eqref{eq:SMC_LTI}, obviously, satisfies
Theorem \ref{thm:sol_iODE}. So, for any initial condition \eqref{eq:ch2_initialcondition}, it has a local-in-time  Filippov solution. 
So, \textit
{the closed-loop system is well-posed}  provided that solutions are unique. Let us show that, under some restrictions,
the feedback \eqref{eq:SMC_LTI} is a global
 finite-time stabilizer of $y$ to zero.
\begin{theorem}\label{thm:SMC_design_LTI_IDE}\itshape
	If $CA\!=\!\Lambda C$  with $\Lambda\!\in\! \R^{m\times m}$ such that \vspace{-1mm}

	\begin{equation}\label{eq:LMI}
		\Lambda^{\top}P+P\Lambda\preceq 0, \quad 0\prec P=P^{\top}\in \R^{m\times m},\vspace{-1mm}
	\end{equation}
	\begin{equation}\label{eq:con_rho}
		\rho \!\geq\!\tfrac{ \left\|C(B \gamma(t)+p(t))\right\|_{P}+\int^t_{t^0}\! \|C\Phi(t,\tau)\tilde B\gamma(\tau)\|_{P} d\tau+\delta}{1-M}, \vspace{-1mm}
	\end{equation}
	\begin{equation}\label{eq:con_Phi}
\small 	\esssup_{t\geq t^0}\int^t_{t^0}\|C\Phi(t,\tau)\tilde B(CB)^{-1}\|_{P} d\tau\leq M<1,\;\; \delta\!>\!0,\vspace{-1mm}
\end{equation}
	then, any  Filippov solution of the closed-loop system \eqref{eq:LTI_IDE}, \eqref{eq:LTI_y}, \eqref{eq:SMC_LTI} exists and defined for all $t\geq t^0$. Moreover, 
	$\frac{d}{dt} \|y(t)\|_{P}\leq -\delta$ for $\|y(t)\|_{P}\neq 0, t\geq t^0$ and 
	there exists
	$T\in \left[t^0,t^0+ |Cx(t^0)|/\delta\right]$ such that
	\begin{equation}
		y(t)=Cx(t)=\zero , \quad \forall t\geq T.
	\end{equation}
\end{theorem}  
\begin{proof}
	For any Filippov solution of the closed-loop system \eqref{eq:LTI_IDE}, \eqref{eq:LTI_y}, \eqref{eq:SMC_LTI} it holds\vspace{-1mm}
	\[\small
	\begin{split}
		\tfrac{d \|x(t)\|}{dt}\!\leq & \|Ax(t)\|\!+\!\rho\| B(CB)^{-1}y\|/\|y\|_{P}\!+\!\|B\gamma\|_{L^{\infty}}\!+\!\|p\|_{L^{\infty}}\!\\
		&\!\!+\!
		\left(\!\rho \| \tilde B(CB)^{-1}y\|/\|y\|_{P}\!+\!\|\tilde B\gamma\|_{L^{\infty}}\!\right)\!\!\int^t_{t^0}\!\!\|\Phi(t,\tau)\| d\tau.
	\end{split}\vspace{-1mm}
	\]
	Since $\Phi$ is measurable and  locally bounded on $\Delta$ then the function $$t\mapsto \int^t_{t^0}\|\Phi(t,\tau)\| d\tau$$ is measurable and locally bounded on $[t_0,+\infty)$. So, using Gr\"onwall--Bellman inequality, we conclude that $\|x(t)\|<+\infty$ for any finite  $t\geq t^0$.
	
	Assumption $CA=\Lambda C$ implies that the output $y$ satisfies \vspace{-1mm}
	\begin{equation}\label{eq:dyn_z_LTI}\small
		\begin{split}
			\dot y(t)\in&\,  \Lambda y(t)+CB\gamma(t)-\tfrac{\rho y(t)}{\|y(t)\|_{P}}+Cp(t)\\
			& +\!\int^t_{t^0}C\Phi(t,\tau)\tilde B
			\left( \gamma(\tau)-\tfrac{\rho (CB)^{-1} y(\tau)}{\|y(\tau\|)_{P}}\right)d
			\tau
		\end{split}\vspace{-1mm}
	\end{equation}
	for almost all $t\geq t^0$. Using \eqref{eq:con_Phi}, \eqref{eq:con_rho}  and \eqref{eq:LMI}, we derive \vspace{-1mm}
	\[
	\begin{split}
		\tfrac{d \|y(t)\|_{P}}{dt}=&\tfrac{y^{\top}(t)P\Lambda y(t)}{\|y(t)\|_{P}} +\tfrac{y^{\top}(t)PC(B\gamma(t)+p(t))}{\|y(t)\|_{P}}-\rho\tfrac{y^{\top}(t)Py(t)}{\|y(t)\|_{P}^2}\\
		&+\tfrac{y^{\top}(t)P}{\|y(t)\|_{P}}\int^t_{t^0}C\Phi(t,\tau)\tilde B
		\left( \gamma(\tau)-\tfrac{\rho (CB)^{-1} y(\tau)}{\|y(\tau\|)_{P}}\right)d
		\tau\\
		\leq &\|C(B\gamma(t)+p(t))\|_{P}\!-\!\rho \!+\rho M\\
		&+\int^t_{t^0} \|C\Phi(t,\tau)\tilde B\gamma(\tau)\|_{P} d\tau\leq-\delta<0\\
	\end{split} \vspace{-1mm}
	\]
	as long as $|y(t)|\neq 0$, so $y(t)=0$ for all $t\geq T$.
\end{proof}

If $\|\gamma\|_{L^{\infty}}\leq \bar \gamma$ and $\|p\|_{L^\infty}\leq \bar p$ then the inequality (51) holds for $\rho =(\sqrt{\|P\|}(\|CB\|\bar\gamma +\|C\|\bar p+M\|CB\|\bar \gamma)+\delta)/(1-M)$. Such a design based on upper estimates on uncertain parameters of the system perfectly follows SMC theory \cite{Utkin1992:Book}, \cite{EdwardsSpurgeon1998:Book} and appreciated by SMC practice \cite{Utkin_etal2009:Book}.

The above theorem presents a sufficient condition under which  $Cx=0$ is a sliding surface of the closed-loop system \eqref{eq:LTI_IDE}, \eqref{eq:LTI_y}, \eqref{eq:SMC_LTI}.  
The SMC  theory \cite{Utkin1992:Book} also studies a motion of the system on this surface. 
Similarly to  ODEs, the sliding motion can be studied using  equivalent control method.

\begin{corollary}\label{cor:Utkin_sol_LTI_SMC}
	\itshape
	Under conditions of Theorem \ref{thm:SMC_design_LTI_IDE}, the closed-loop system 
	\eqref{eq:LTI_IDE}, \eqref{eq:LTI_y}, \eqref{eq:SMC_LTI} has the unique Utkin solution with\vspace{-1mm}
	\begin{equation} \label{eq:LTI_IDE_u_eq}
		u_{\rm eq}(t)\!=\!\left\{
		\begin{array}{cl}
			\!-\rho\frac{(CB)^{-1}y(t)}{\|y(t)\|_{P}}, & \!\!t^0\!\leq\! t\!<\!T,\\
			\!\!\tilde u_{\rm eq}(t)-\gamma(t), &  \!\!t\!\geq \! T,
		\end{array}
		\right.  \;\tilde u_{\rm eq}\!=\!\sum_{i=0}^{\infty} \!\mathcal{K}^ig,\,
		\vspace{-1mm}
	\end{equation}
	where $T\in [t^0,t^0+T^{\max}]$ is, as before, the reaching time, and  the linear bounded operator  $\mathcal{K}: L^{\infty}((T,t);\R)\to L^{\infty}((T,t);\R)$ is given for $s\in[T,t]$ by\vspace{-1mm}
	\begin{equation}\label{eq:K_operator}	
		(\mathcal{K}g)(s)\!=\!-(CB)^{-1}\!\int^{s}_{T}\!\!\!C\Phi(t,\tau)\tilde B g(\tau)d\tau,  
		\vspace{-1mm}
	\end{equation}
	\begin{equation}
		\!g(t)\!=-(CB)^{-1}\left(Cp(t)\!+\tilde g(t)\right),\vspace{-1mm}
	\end{equation}
\begin{equation}
		\tilde g(t)=
		\int^{T}_{t^0}\!\!C\Phi(t,\tau)\tilde B\left(\gamma(\tau)-\tfrac{\rho (CB)^{-1}y(\tau)}{\|y(\tau)\|_{P}}\right)\,d\tau.\vspace{-1mm}
	\end{equation}	
	For $t\!\geq\! t^0\!+\!T$, the dynamics of 	\eqref{eq:LTI_IDE}, \eqref{eq:LTI_y}, \eqref{eq:SMC_LTI}   is governed by \vspace{-1mm}
	\begin{equation}\label{eq:dyn_Utkin_LTI_SMC}
\begin{split}
		\!\!\!\!\!\dot x(t)\!=&\,Ax(t)\!+\Pi p(t)
		\!+\!\!\int^t_{t^0} \!\!\!\!\Pi\Phi(t,\tau) \tilde B(u_{\rm eq}(\tau)\!+\!\gamma(\tau))d
			\tau,\\
			0=&\, Cx(t),
		\end{split} \vspace{-1mm}
	\end{equation}
{where $\Pi=\left(I_n\!-\!B(CB)^{-1}C\right)$ is a projector on $Cx=0$.}
\end{corollary} 
\begin{proof}
	By Theorem \ref{thm:sol_iODE_Utkin},  the system has an Utkin solution coinciding with some Filippov solution.  Since, by Theorem \ref{thm:SMC_design_LTI_IDE}, any Filippov solution belongs a sliding set for all $t\geq T$, then, by Corollary \ref{eq:uniq_Utkin_IDE}, any Utkin solution is unique   and \vspace{-1mm}
	\begin{equation}\label{eq:tilde_u_eq_LTI}
	\begin{split}
		0=&CB(u_{\rm eq}(t)+\gamma(t))+Cp(t)\\
		&+\int^t_{t^0}C\Phi(t,\tau)\tilde B
		\left(u_{\rm eq} (\tau)+\gamma(\tau)\right)d \tau,
	\end{split}\quad  t\geq T.\vspace{-1mm}
	\end{equation}
Hence, we derive \eqref{eq:dyn_Utkin_LTI_SMC}.
\end{proof}	

Under some additional conditions the dynamics  \eqref{eq:dyn_Utkin_LTI_SMC} of the system in the sliding mode can be defined by an ODE, e.g.,
	if  $\Phi(t,\tau) \tilde B\in \range(B)$  then 
$
		\Pi\Phi(t,\tau) \tilde B=\zero.$

Let $\mathrm{spec}(\cdot)$ denotes the spectrum of a matrix.
The  case of a bounded  input delay is treated by the following corollary.
\begin{corollary}\itshape 
Let there exist  $d>0$ such that\vspace{-1mm}
 \begin{equation}\label{eq:Phi_d}
		\Phi(t,\tau)=\zero \quad  \text{for all} \quad  (t,\tau)\in \Delta: t-\tau\geq d.\vspace{-1mm}
	\end{equation} 
	If, under condition of Theorem \ref{thm:SMC_design_LTI_IDE}, the set
	$\mathrm{spec(A)}\backslash\mathrm{spec}(\Lambda)$ belongs to the left complex half-plane and\vspace{-1mm}
	\begin{equation}\label{eq:p_vanishing}
		\esssup_{s\geq t} \|p(s)\|\to 0 \text{ as } t\to +\infty,\vspace{-1mm}
	\end{equation}
	then for any Utkin solution of  the closed-loop system 
	\eqref{eq:LTI_IDE}, \eqref{eq:LTI_y}, \eqref{eq:SMC_LTI}  it holds
$
		x(t)\to 0 \text{ as } t\to +\infty.
$
\end{corollary}
\begin{proof}
	Since $\Phi(t,\tau)=0$ for all $t\geq T+d$ and all $\tau\in [t^0,T]$ then, for all $ t\geq T+d$, it holds $Cx(t)=0$ and\vspace{-1mm}
	\begin{equation}
\begin{split}
	\!\!\!\!\!\dot x(t)\!=&\,Ax(t)\!+\Pi p(t)
	\!+\!\Pi\!\int^t_{t-d} \!\!\!\!\Phi(t,\tau) \tilde B\tilde u_{\rm eq}(\tau)d
	\tau,\\
	0=&\,CB\tilde u_{\rm eq}(t)+Cp(t)+\int^t_{t-d}C\Phi(t,\tau)\tilde B
	\tilde u_{\rm eq} (\tau)d \tau,\vspace{-1mm}
\end{split}
\end{equation}
Hence, the condition \eqref{eq:p_vanishing} implies that $\esssup_{s\geq t}\|\tilde u_{\rm eq}(t)\|\to 0$ as $t\to +\infty$ and $$\esssup_{s\geq t}\left\|\int^t_{t-d} \Phi(t,\tau) \tilde B\tilde u_{\rm eq}(\tau)d
\tau\right\|\to 0 \text{
as }t\to +\infty.$$
Since the sliding surface $Cx=0$ is an eigenspace of $A$  then there exists $C^{\top}\in \R^{n-m\times n}$ such that $C^{\top}C^{\bot}=0$, $\mathrm{rank}{C^{\bot}}=n-m$ and $AC^{\bot}=\tilde \Lambda C^{\bot}$, where $\tilde \Lambda\in \R^{(n-m) \times (n-m)}: \mathrm{spec}(\tilde \Lambda)=\mathrm{spec(A)}\backslash\mathrm{spec}(\Lambda)$. Hence,  $z(t)=C^{\bot}x(t)$ satisfies\vspace{-1mm}
\[
\dot z(t)=\tilde \Lambda z(t)+C^{\bot}\Pi p(t)
\!+\!C^{\bot}\Pi\!\int^t_{t-d} \!\!\!\!\Phi(t,\tau) \tilde B\tilde u_{\rm eq}(\tau)d
\tau,\vspace{-1mm}
\]
for all $ t\geq T+d$.
Since, by assumption, the matrix $\tilde \Lambda$ is Hurwitz,  then
$z(t)=C^{\bot}x(t)\to 0$ as $t\to+\infty$.
\end{proof}

\begin{example}\itshape
	Let us consider the system \eqref{eq:LTI_IDE}, \eqref{eq:LTI_y}, \eqref{eq:SMC_LTI} with 
	\[
	A=\left(
	\begin{smallmatrix}
		-2 & 4 & 2\\
		0 & -3 & 1\\
		-1 & 2 & 1
	\end{smallmatrix}
	\right), \quad B=\tilde B=\left(\begin{smallmatrix}
		0\\
		0\\
		1
	\end{smallmatrix}
	\right),  \quad  C=\left(
	\begin{smallmatrix}
		1 & 0 & -2
	\end{smallmatrix}
	\right),
	\]
	$p=\zero$ and $\Phi(t,\tau)=\xi(t-\tau)I_n$, where $\xi(t)=1$ for $0\leq t-\tau< 1$ and $\xi(t)=0$ for $t-\tau \geq 1$. In this case, the condition \eqref{eq:con_Phi} is fulfilled with $M=1$ and $|CB|=2$. 
	Therefore, by Theorem \ref{thm:SMC_design_LTI_IDE},
	the SMC  \eqref{eq:SMC_LTI} with	 $ \rho>6 \|\gamma\|_{L^{\infty}}$
	stabilizes the output $y=Cx$ of the system \eqref{eq:LTI_IDE} to zero in a finite time.
	The results of numerical simulations for $\rho=4$, $\gamma(t)=0.5\cos(2t)$ and $t^0=0$ are depicted on Figures 
	\ref{fig:1}, \ref{fig:2} and \ref{fig:3}. The simulations have been done by means of explicit Euler method \eqref{eq:explicit_Euler_IDE}
	with the sampling period $h=0.001$. The integral is computed using the method of rectangles using the same sampling period.
	The applicability of this method  is discussed in Remark \ref{rem:1}.
	
	To  compare the SMC for IDE with the SMC for ODE , we also present the simulation results of the system \eqref{eq:LTI_IDE}-\eqref{eq:SMC_LTI}  with $\Phi=\zero$ (please see Figures \ref{fig:5}, \ref{fig:6} and \ref{fig:7}). The simulations show a faster convergence of the ODE ($\Phi=\zero$) system to zero. 
	
	The control signal depicted on Figure \ref{fig:3} is the typical signal generated by SMC, which has a chattering during  the sliding mode (after the time instant $t=T\approx 0.55$). The chattering is caused by the use of the  explicit method, which correctly approximates the dynamics of the SMC system only away from the sliding surface (see Remark \ref{rem:1} or \cite{AcaryBrogliato2010:SCL}).
	Recall that (see, e.g., \cite{Utkin1992:Book}), in practice, the equivalent control can be estimated by filtering the discontinuous input signal $u(t)$:\vspace{-1mm}
	\begin{equation}\label{eq:filter}
		\epsilon\dot u_{\epsilon}(t)=-u_{\epsilon}(t)+u(t), \quad t>t^0, \quad u_{\epsilon}(t^0)=0, \vspace{-1mm}
	\end{equation} 
	where $\epsilon>0$ is a small enough constant parameter.  For $t\geq t^0+T$, the SMC engineering   considers
	$u_{\epsilon}(t)\approx u_{\rm eq}(t)$.
	In the view of \eqref{eq:tilde_u_eq_LTI}, the function $\delta: [t^0, +\infty)\to \R$ given by\vspace{-1mm}
	\begin{equation*}
		\delta(t)=CB(u_{\epsilon}(t)+\gamma(t))+\int^t_{t^0}C\Phi(t,\tau) \tilde B(u(\tau)+\gamma(\tau)) d\tau\vspace{-1mm}
	\end{equation*}
	can be considered as a practical sliding mode indicator. 
		Indeed,  if   the Utkin solution  is an adequate description of a motion of  the considered discontinuous control system  and 
		the explicit Euler method gives a precise enough approximate of Utkin solution of IDE then the approximate identity  $\sigma(t)\approx 0$   may  be treated  as an indication  of a practical sliding mode for $t\geq T$. The evolution of the signal $\delta$ computed for $\epsilon=0.01$ is depicted on Figure \ref{fig:4}. The numerical computation of $\delta$ has been done with the sampling period $h=0.0001$. The obtained $\delta(t)$ is, indeed, close to zero for $t\geq 0.6$. This may be interpreted as some heuristic confirmation of the validity  of Utkin solutions for description of a motion of discontinuous SMC control systems modeled by IDEs.     A small delay in detection of SMC is  caused by an inertia  of the linear filter \eqref{eq:filter} and by  numerical errors of the explicit Euler method. Development of a rigorous numerical analysis for IDEs with SMC is the important problem for the future research. 
			\begin{figure}
				\vspace{-2mm}
		\centering
		\includegraphics[height=43mm,width=0.37\textwidth]{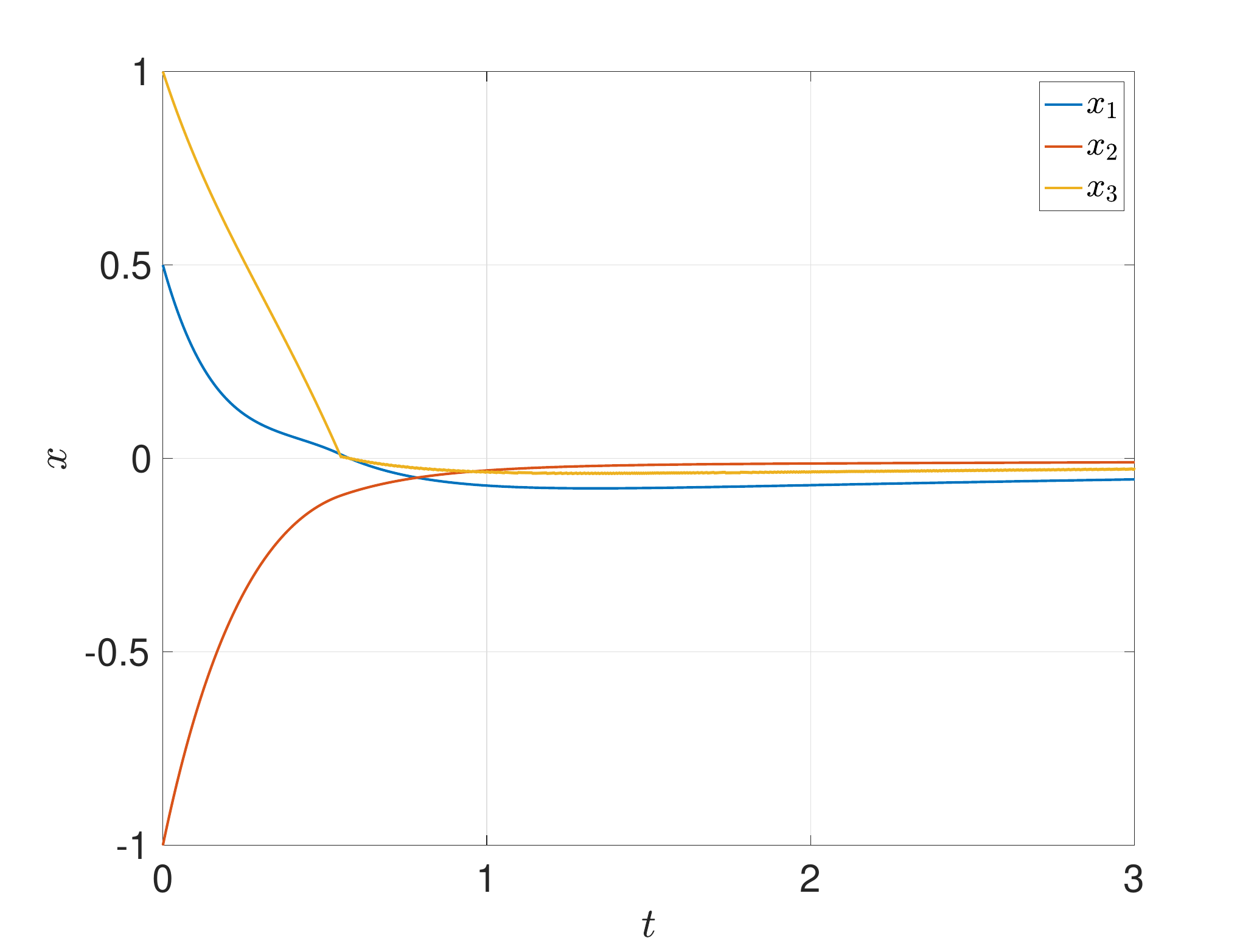}\vspace{-4mm}
		\caption{The states of the closed-loop system \eqref{eq:LTI_IDE}-\eqref{eq:SMC_LTI}}
		\label{fig:1}
	\end{figure}
	\begin{figure}
		\vspace{-2mm}
		\centering
		\includegraphics[height=43mm,width=0.37\textwidth]{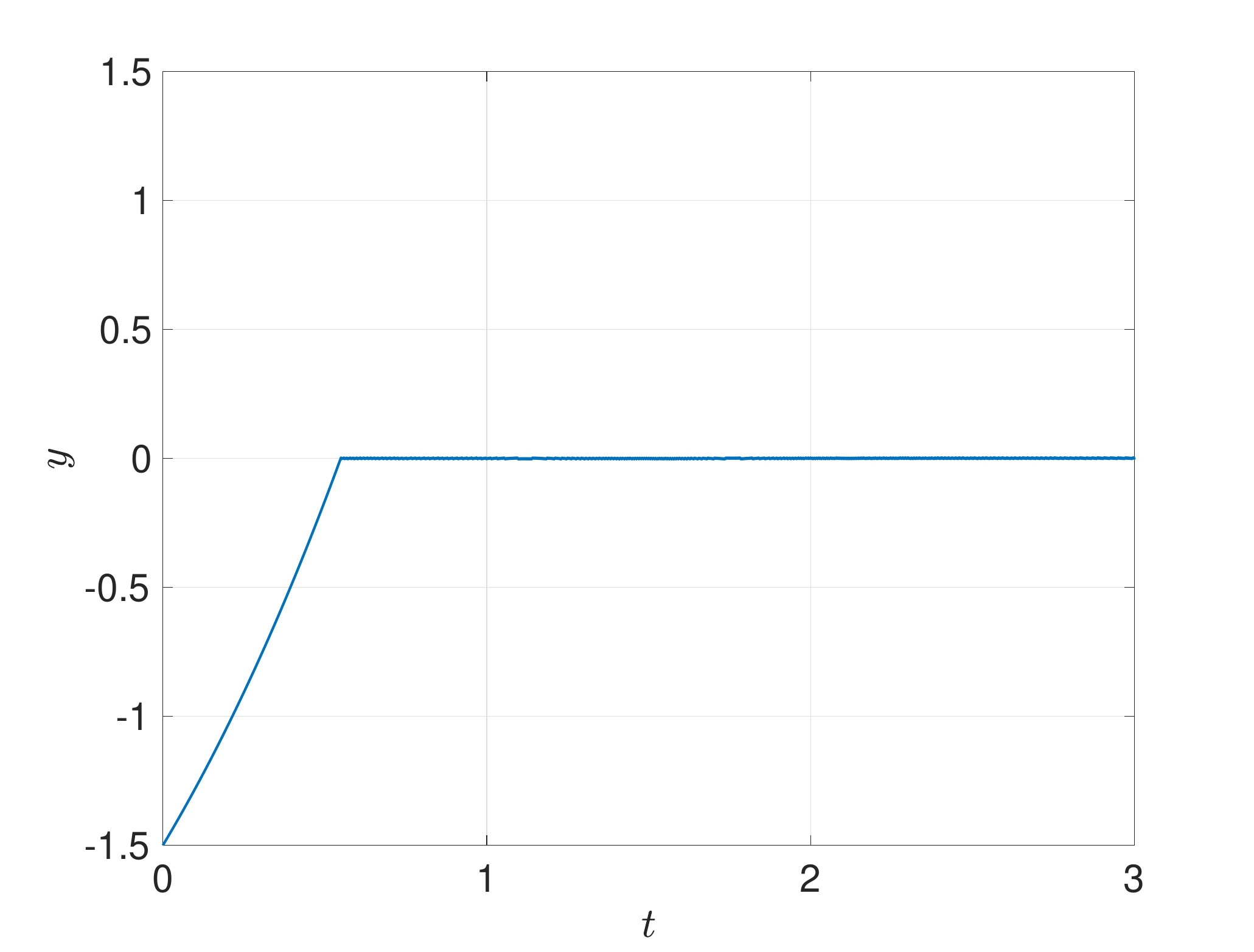}\vspace{-4mm}
		\caption{The output of the closed-loop  system \eqref{eq:LTI_IDE}-\eqref{eq:SMC_LTI}}
		\label{fig:2}\vspace{-5mm}
	\end{figure}
	\begin{figure}
		\vspace{-2mm}
			\centering
		\includegraphics[height=43mm,width=0.37\textwidth]{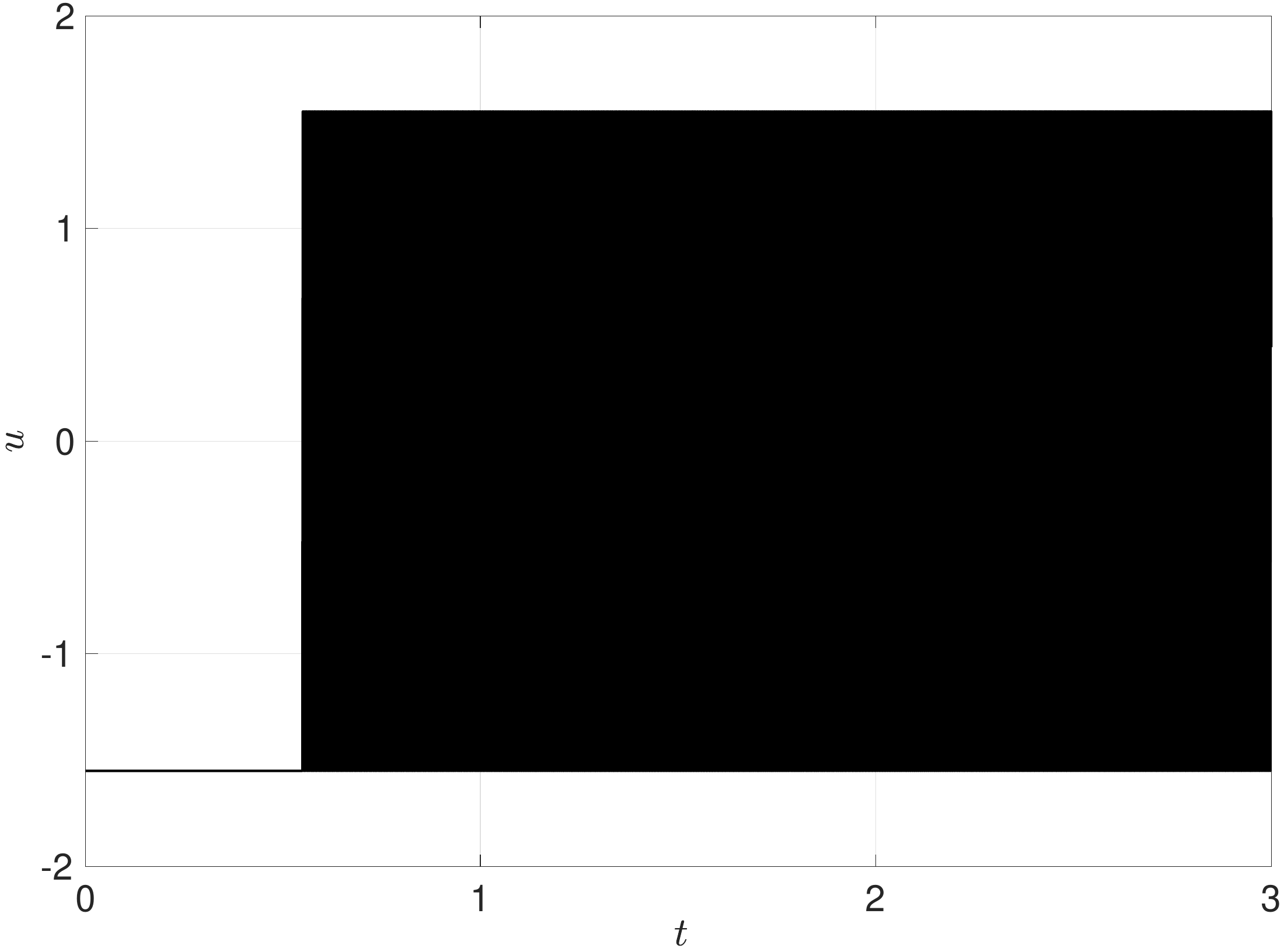}\vspace{-4mm}
		\caption{The control input of the closed-loop system \eqref{eq:LTI_IDE}-\eqref{eq:SMC_LTI}}
		\label{fig:3}
	\end{figure}
\end{example}

 \subsection{SMC for dynamical systems in Banach spaces}
 {Let us consider a control system modeled by a differential equation in a Banach space \cite{DaleckyKrein1974:Book}, \cite{Yosida1980:Book}, \cite{Pazy1983:Book}\vspace{-1mm}
 	\begin{equation}\label{eq:EE}
 		\dot x(t)\!=\!\!Ax(t)\!+\!B(u(t)+\gamma(t)), \,t\!>\!t_0, \, x(t_0)\!=\!x^0\!\!\in\! \mathcal{D}(A), \vspace{-1mm}
 	\end{equation}
 	where $x(t)\in \B$ is the system state at time $t$, $A: \mathcal{D}(A)\subset \B\to \B$ is a linear closed densely defined (possibly) unbounded operator, which generates a strongly continuous semigroup $\Psi(s):\B\to \B$, $s\geq 0$ of linear bounded operators  on a \textit{real Banach space} $\B$, $B:\R^m\to \B$ is a linear bounded operator, $u:\R\to \R^m$ is the finite-dimensional control input, $\gamma\in L^{\infty}(\R;\R^m)$ is the matched perturbation   and $\mathcal{D}(A)$ is the domain of $A$. 
 	Let the output of the system be defined as \vspace{-1mm}
 	\begin{equation}\label{eq:EE_output}
 		y(t)=Cx(t)\in \R^k,\vspace{-1mm}
 	\end{equation}
 	where $C: \B\to \R^k$ is a linear bounded operator. 
 	For any $u\in L^1_{\rm loc}([t^0,+\infty);\R^m)$, the open-loop systems \eqref{eq:EE} is well-posed and has the so-called \textit{mild solution} 
 	\cite[page 106]{Pazy1983:Book}\vspace{-1mm}
 	\begin{equation}\label{eq:EE_solution}
 		x(t)=\Psi(t-t^0)x^0+\int^t_{t^0}\Psi(t-\tau)B(u(\tau)+\gamma(\tau)) d\tau,\vspace{-1mm}
 	\end{equation}
 	where the integral is understood in the sense of Bochner \cite{Yosida1980:Book}.
 	The mild solution $x: [t^0,t^0+T)\to \B$ with $0<T\leq +\infty$ is the  \textit{classical} (resp., \textit{strong}  \cite[page 109]{Pazy1983:Book}) solution of \eqref{eq:EE}  if $x(t)\!\in\! \D(A)$ satisfies \eqref{eq:EE} for (almost) all $t\in [t^0,t^0+T)$.
 	
 				\begin{figure}
 		\vspace{-2mm}
 		\centering
 		\includegraphics[height=43mm,width=0.37\textwidth]{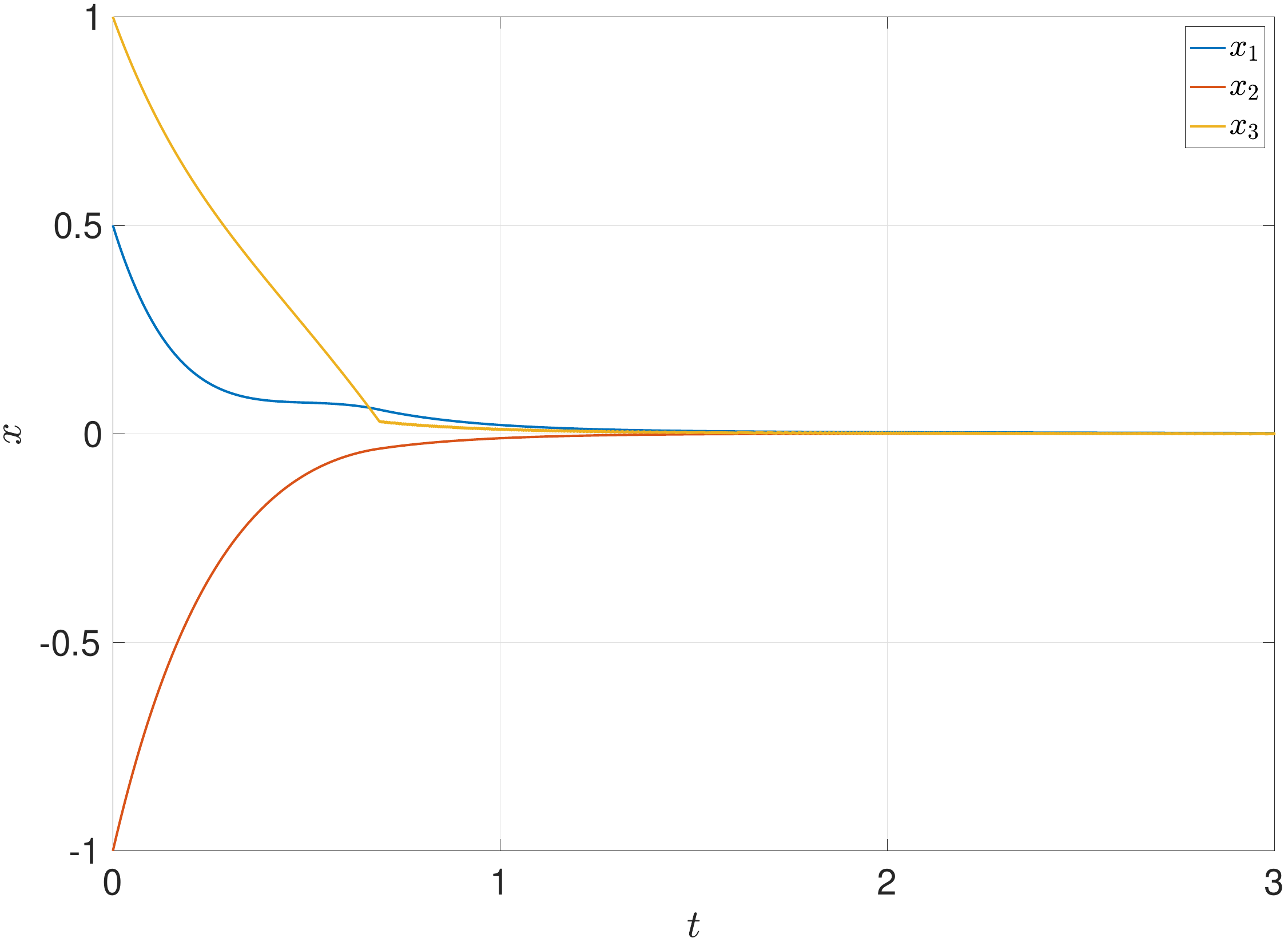}\vspace{-4mm}
 		\caption{The states of the closed-loop system \eqref{eq:LTI_IDE}-\eqref{eq:SMC_LTI} with $\Phi=\zero$}
 		\label{fig:5}
 	\end{figure}
 	\begin{figure}
 		\vspace{-2mm}
 		\centering
 		\includegraphics[height=43mm,width=0.37\textwidth]{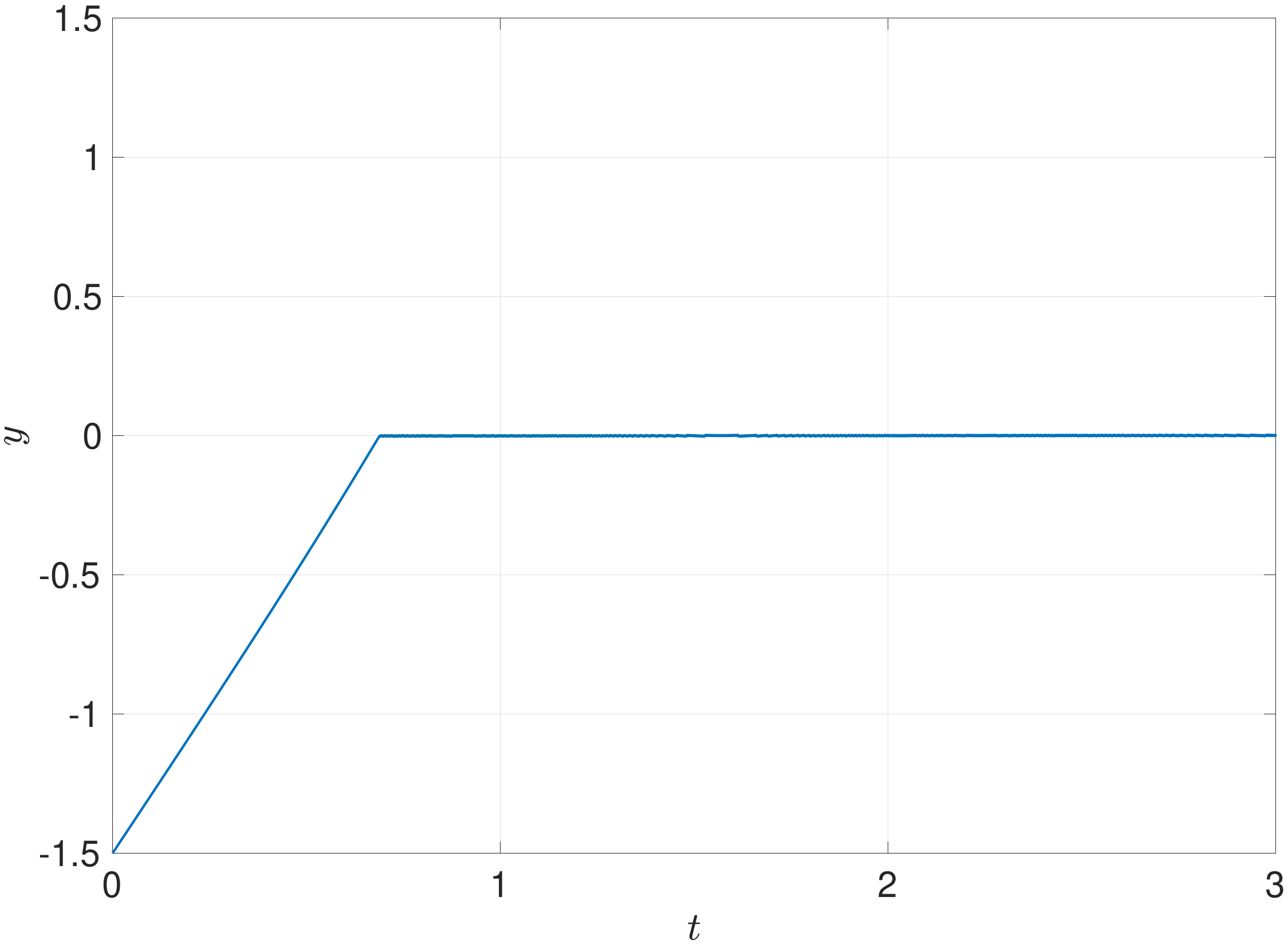}\vspace{-4mm}
 		\caption{The output of the closed-loop  system \eqref{eq:LTI_IDE}-\eqref{eq:SMC_LTI} with $\Phi=\zero$}
 		\label{fig:6}
 	\end{figure}
 	\begin{figure}
 		\centering
 		\includegraphics[height=43mm,width=0.37\textwidth]{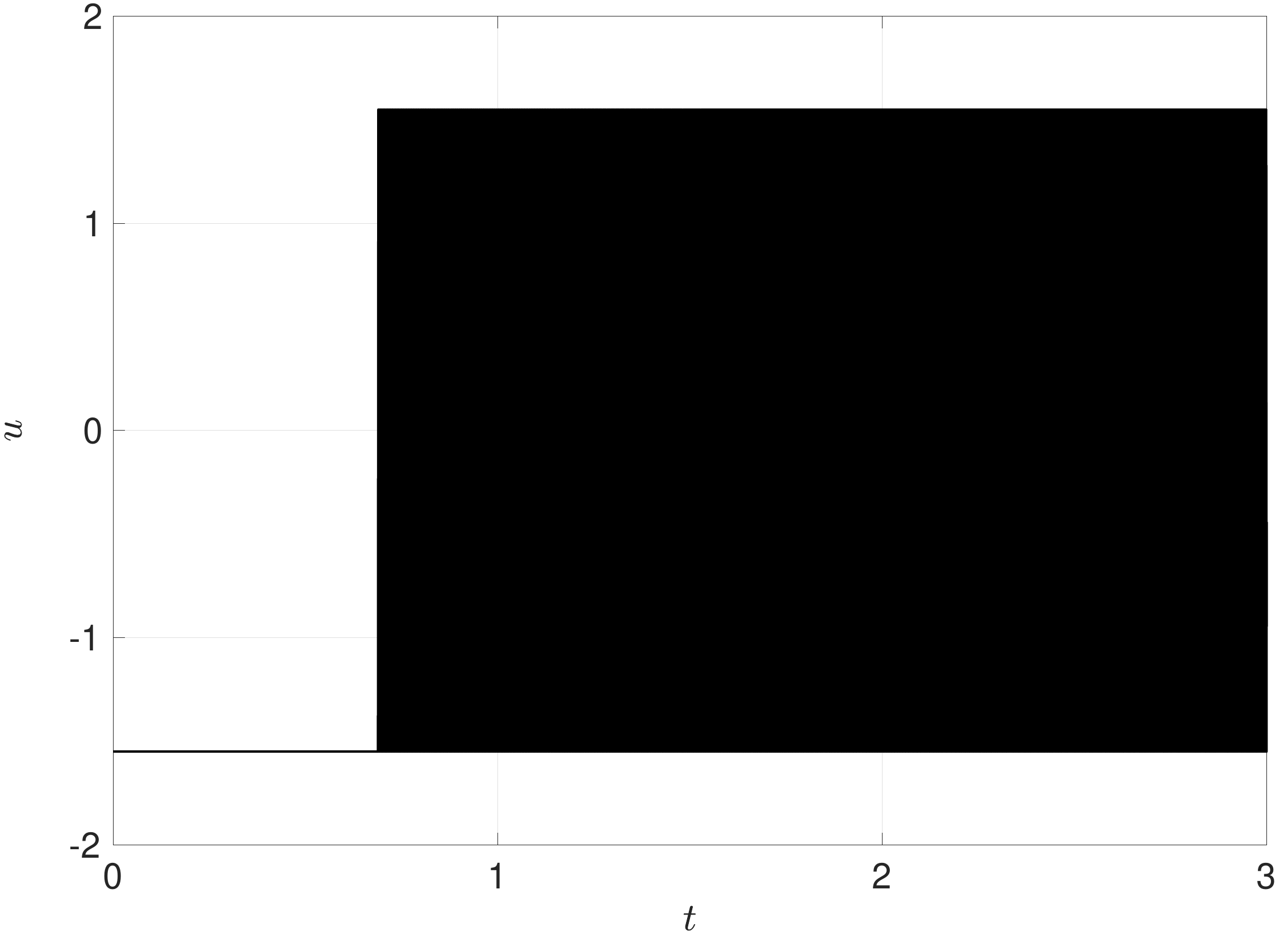}\vspace{-4mm}
 		\caption{The control input of the closed-loop system \eqref{eq:LTI_IDE}-\eqref{eq:SMC_LTI} with $\Phi=\zero$}
 		\label{fig:7}
 	\end{figure}
 	
 	Let  a feedback law be  defined as follows\vspace{-1mm}
 	\begin{equation}\label{eq:EE_control}
 		u(t)=\tilde u(t,y(t)),\quad \tilde u: [t^0,+\infty)\times \R^k \to \R^{m},\vspace{-1mm}
 	\end{equation}
 	where $\tilde u$ is a measurable function.
 	Even if a  function  $\tilde u$ is continuous in $y$, the non-linearity in the right-hand side of the closed-loop system may not 
 	be locally Lipschitz continuous, so, the conventional theory of nonlinear abstract differential equations
 	(see, e.g., \cite{Pazy1983:Book}, \cite{EngelNagel2000:Book}) is not applicable for its analysis.
 	In the case of SMC design \cite{Orlov_elal2004:IJC}, the closed-loop system  \eqref{eq:EE}, \eqref{eq:EE_output}, \eqref{eq:EE_control} is ill-posed.  Let us regularize  the system \eqref{eq:EE}, \eqref{eq:EE_output} following the  concept of Utkin solutions \cite{Orlov2008:Book}, \cite{PolyakovFridman2014:JFI}.
 	\begin{definition}\label{def:Fillipov_solution_EE}\label{def:Filippov_sol_EE}\itshape 
 		A continuous function $x:[t_0,t^0+T) \to \B$	with $x(t^0)\!=\!x^0$ and $0<T\leq+\infty$ is said to be an Utkin solution of the first (resp., second) kind for the system  \eqref{eq:EE}, \eqref{eq:EE_output}, \eqref{eq:EE_control} if 
 			 $x$ satisfies \eqref{eq:EE_solution} with $u=u_{\rm eq}\in L^1((t_0,t^0+T);\R^m)$ such that
 		$u_{\rm eq}(t)\in U(t,Cx(t))$ for almost all $t\!\in\! (t_0,t^0+T)$, where $U\!=\!K[\tilde u]$  (resp., $U=(K[\tilde u_1], ..., K[\tilde u_m])^{\top}$ for $\tilde u=(\tilde u_1, ...,\tilde u_m)^{\top}$).
 	\end{definition}
In other words, an Utkin solution $x$ of   \eqref{eq:EE}, \eqref{eq:EE_output}, \eqref{eq:EE_control}  is  a mild solution  with a properly selected $u(t)\!\in\! U(t,Cx(t))$.

If 
 $\range(B)\!\subset\! \mathcal{D}(A)$, then,
 by \cite[Theorem 2.9, Chapter 4]{Pazy1983:Book},
 any mild solution is a strong  solution. By  \cite[Theorem 2.4, Chapter 1]{Pazy1983:Book},
 the function \vspace{-1mm}
 $$
 s\in [0,+\infty)\mapsto \Phi(s):=C\Psi(s) AB\in  \R^{k\times m}
 \vspace{-1mm}
 $$ is \textit{continuous}. 
 Moreover, taking into account $\frac{d\Psi(s)B}{ds}=\Psi(s)AB$ and $\frac{d\Psi(s)x^0}{ds}=\Psi(s)Ax^0, \forall s> 0,\forall x^0\in \D(A)$  (see, \cite[Theorem 2.4, Chapter 1]{Pazy1983:Book}), we derive \vspace{-1mm}
\[\small
\begin{split}
	\tfrac{d}{dt}y(t)=\tfrac{d}{dt}Cx(t)=&\frac{d}{dt}C\Psi(t-t)x^0\\
	&+\frac{d}{dt}\int^t_{t^0}\underbrace{C\Psi(t-\tau)B}_{\in \R^{k\times m}}(u(\tau)+\gamma(\tau))d\tau\\
	=&C\Psi(t-t)Ax^0+CB(u(t)+\gamma(t))\\
	&+\int^t_{t^0}\underbrace{C\Psi(t-\tau)AB}_{\Phi(t-\tau)}(u(\tau)+\gamma(\tau))d\tau
\end{split}
\] 
almost everywhere. Therefore,
 the input-output dynamics of the system \eqref{eq:EE}, \eqref{eq:EE_output}  is given by the IDE\vspace{-1mm}
 \begin{equation}\label{eq:EE_IDE}
 	\dot y(t)=p(t-t^0)+CB(u(t)+\gamma(t))+\int^t_{\sigma} \!\!\Phi(t,\tau)(u(\tau)+\gamma(\tau))d\tau, \vspace{-1mm}
 \end{equation}
 where $t> t^0$, $CB\in \R^{k\times m}$  and \vspace{-1mm}
 \begin{equation}\label{eq:x_to_y}
 	y(t^0)\!=\!Cx^0,\quad p(t-t^0)\!=\!C\Psi(t-t^0)Ax^0.\vspace{-1mm}
 \end{equation}
Notice that the function $s\mapsto p(s)$ is  \textit{continuous} if $x_0\in \D(A)$. 
	\begin{figure}
	\vspace{-2mm}
	\centering
	\includegraphics[height=43mm,width=0.37\textwidth]{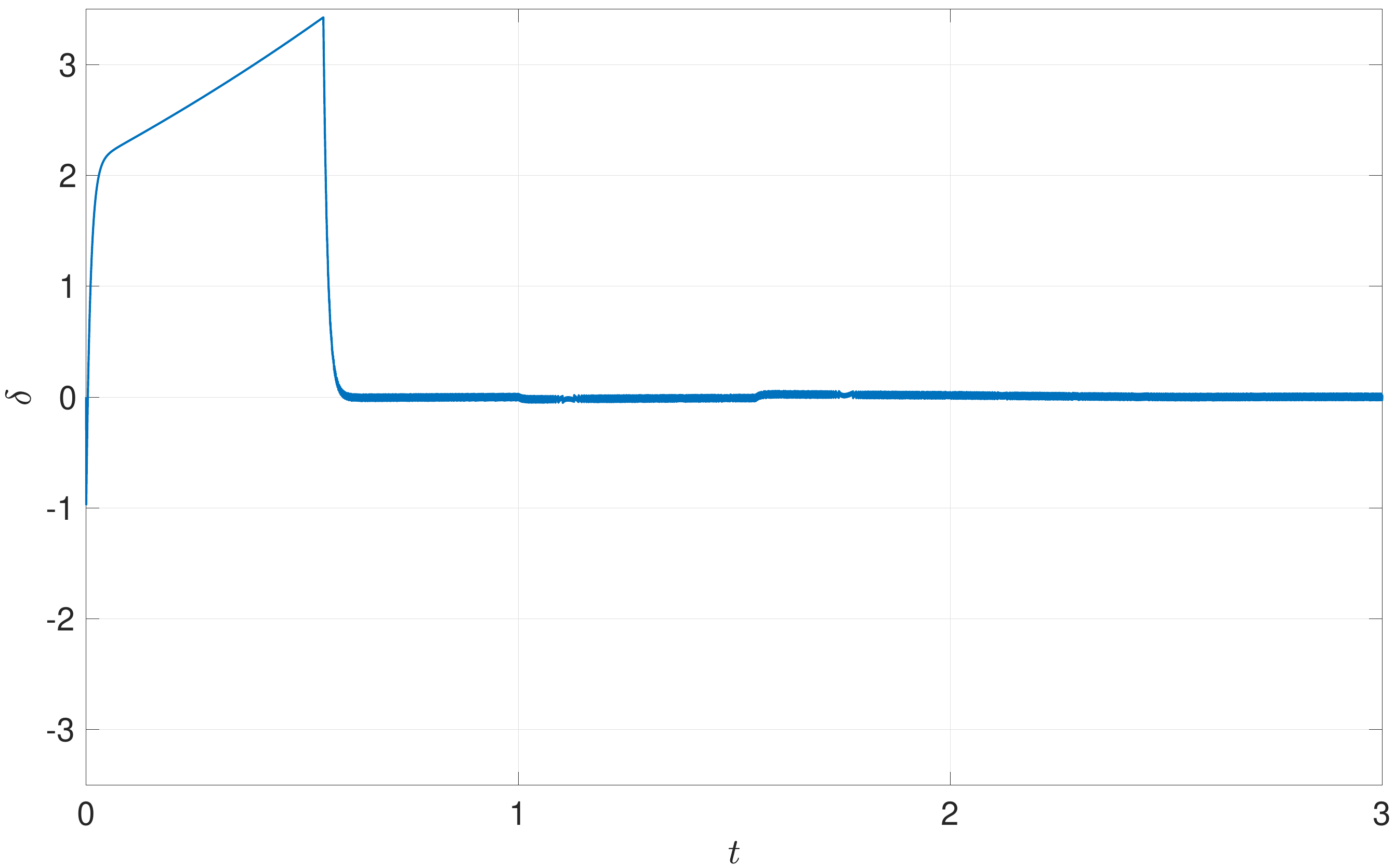}\vspace{-4mm}
	\caption{The sliding mode  indicator $\delta$ of the closed-loop system \eqref{eq:LTI_IDE}- \eqref{eq:SMC_LTI}}
	\label{fig:4}
\end{figure}

Let  $y_{\rm eq}$ be an Utkin solution of the IDE \eqref{eq:EE_IDE},  \eqref{eq:x_to_y} with the feedback law \eqref{eq:EE_control},
and let $u_{\rm eq}$ be the corresponding equivalent control. By Theorem \ref{thm:sol_iODE_Utkin}, the pair $(y_{\rm},u_{\rm eq})$  always exists if $\tilde u$ is a locally essentially bounded measurable function.  If $x_{\rm eq}$ is a mild solution of \eqref{eq:EE} given by \eqref{eq:EE_solution} with $u\!=\!u_{\rm eq}$, then\vspace{-1mm}
\[\small
\begin{split}
\tfrac{d}{dt}Cx_{\rm eq}(t)=&C\Psi(t-t)Ax^0+CB(u_{\rm eq}(t)+\gamma(t))\\
&+\int^t_{\sigma} \!\!\Phi(t-\tau)(u_{\rm eq}(\tau)+\gamma(\tau))d\tau,
\end{split}\vspace{-1mm}
\]
i.e., $y_{\rm eq}(t)=Cx_{\rm eq}(t)$ and $u_{\rm eq}(t)\in U(t,Cx(t))$ almost everywhere. By Definition \ref{def:Filippov_sol_EE}, this means,  that the function $x_{\rm eq}$, constructed by solving the IDE \eqref{eq:EE_IDE}, is an Utkin solution of the nonlinear (possibly discontinuous) differential equation \eqref{eq:EE}, \eqref{eq:EE_control} in a Banach space. As a trivial corollary of Theorem \ref{thm:sol_iODE_Utkin} (resp., Theorem \ref{thm:Caratheodory_iODE}), we derive the  following result.
 	\begin{corollary}\label{cor:Utkin_EE}\itshape
 		If, $\range(B)\subset \mathcal{D}(A)$ and $\tilde u$ satisfies    Filippov  Condition, then, for any $x_0\in \mathcal{D}(A)$,  the system \eqref{eq:EE}, \eqref{eq:EE_output}, \eqref{eq:EE_control} has an Utkin solution  solution $x :[t^0,t^0+T)\to \D(A)$, which is defined, at least, locally in time and  satisfies  the differential equation \eqref{eq:EE} with $u=u_{\rm eq}$ almost everywhere, where $0< T\leq +\infty$. If, additionally, $\Phi\in L^{\infty}((0,+\infty);\R^{k\times m})$  and 
 		$\esssup_{t\in \R,y\in \R^k}\|\tilde u(t,y)\|<+\infty$, then $T=+\infty$.
 	\end{corollary}
The uniqueness of Utkin solution of the IDE \eqref{eq:EE_IDE}, \eqref{eq:EE_control}
 implies the uniqueness of Utkin solution of the system \eqref{eq:EE}, \eqref{eq:EE_output}, \eqref{eq:EE_control}.
 	\begin{corollary}\itshape 
 		Let $\range(B)\subset \mathcal{D}(A^2)$ and $\tilde u$ satisfy Filippov (resp., Carath\'eodory) Condition. If Utkin  solution $x(t,t^0,x^0)$ of  \eqref{eq:EE}, \eqref{eq:EE_output}, \eqref{eq:EE_control} is  unique  on $\mathcal{I}$ then it depends continuously on initial data: $x(t,t^1,x^1)\to x(t,t^0,x^0)$ as $|t^1-t^0|+\|x^1-x^0\|_{\B}+\|Ax^1-Ax^0\|_{\B}\to 0$ unformly on compacts from $\mathcal{I}$.
 	\end{corollary}
 	\begin{proof}
 	Let $x(t,t^0,x^0)$ be  an Utkin solution  with the corresponding equivalent control $u_{\rm eq}$. Let us denote $\xi(t,t^0,x^0)=\int^t_{t^0}u_{\rm eq}(\tau,t^0,x^0) d\tau$. 
 		Since $\range(B)\subset \mathcal{D}(A^2)$ then $\frac{d\Psi(s)AB}{ds}=\Psi(s)A^2B, \forall s> 0$ and
 		\vspace{-1mm}
 		\[\small
 		\begin{split}
 			-\xi(t,t^0,x^0)=&\,\Phi(t-t^0)\xi(t^0,t^0,x^0)-\Phi(0)\xi(t,t^0,x^0)\\
 			=&\int^{t-t^0}_{0} \frac{d}{d\sigma}\left(\Phi(\sigma) \xi(t-\sigma,t^0,x^0)\right)d \sigma\\ =&\int^{t-t^0}_{0}C\Psi(\sigma)A^2B\xi(t-\sigma,t^0,x^0) d\sigma\\
 			&-\int^{t-t^0}_0 \Phi(\sigma) \dot \xi(t-\sigma,t^0,x^0)d \sigma
 		\end{split}\vspace{-1mm}
 		\]
 		Hence, $\int^t_{t^0} \Phi(t-\tau) \dot \xi(\tau,t^0,x^0)d \tau =\xi(t,t^0,x^0)
 		+\int^t_{t^0}C\Psi(t-\tau)A^2B\xi(\tau,t^0,x^0) d\tau$ and
 		the pair $y(t,t^0,x^0)=Cx(t,t^0,x^0)$, $\xi(t,t^0,x^0)$ is an Utkin solution of the  IDE\vspace{-1mm}
 		\[\small
 		\begin{split}
 			\dot y(t,t^0,x^0)\in &\,p(t,t^0,x^0)+CB(U(t,y(t,t^0,x^0))+\gamma(t))\\
 			& + \int^t_{t^0} \Phi(t-\tau) \gamma(\tau )d \tau+\xi(t,t^0,x^0)\\
 			&+\int^t_{t^0}C\Psi(t-\tau)A^2B\xi(\tau,t^0,x^0) d\tau,\\
 			\dot \xi(t,t^0,x^0)\in&\, U(t,y(t,t^0,x^0)), \quad \xi(t^0)=0, \quad y(t^0)=Cx^0,
 		\end{split}\vspace{-1mm}
 		\]
 		where $p(t,t^0,x^0)=C\Psi(t-t^0)Ax^0$.  Since $p(t,t^0,x^0)=C\Psi(t-t^0)Ax^0$ depends continuously on $Ax^0$, then,  using Theorem \ref{thm:continuity_iODE}, 
 		we conclude that $\xi(t,t^0,x^0)$ depends continuously on $(t^0,x^0,Ax^0)$. 
 		Since $\mathrm{range}(B)\!\subset\! \D(A^2)\!\subset\! \D(A)$ then \vspace{-1mm}
 		\[
 		\begin{split}
 		x(t,t^0\!,x^0)\!=\,&\Psi(t\!-\!t^0)x^0\!+\!\!\int^t_{t^0}\!\!\Psi(t\!-\!\tau) B(\dot \xi(\tau,t^0\!,x^0)\!+\!\gamma(\tau))d\tau\\
 		=&\Psi(t-t^0)x^0\!+\!B\xi(t,t^0\!,x^0)\!+\!\!\int^t_{t^0}\!\!\!\Psi(t\!-\!\tau) B\gamma(\tau)d\tau\\
 		&+\int^t_{t^0}\Psi(t-\tau) AB \xi(\tau,t^0,x^0)d\tau,
 		\end{split}\vspace{-1mm}
 		\]
 		where the identity $\frac{d\Psi(s)B}{ds}=\Psi(s)AB,\forall s> 0$ and the integration by parts has been utilized.
 	Continuous dependence of $\xi(t,t^0,x^0)$ on $(t^0,x^0,Ax^0)$ implies continuous dependence of $x(t,t^0,x^0)$ on $(t^0,x^0,Ax^0)$.
 	\end{proof}
The discontinuous IDE \eqref{eq:EE_IDE} is utilized  above for  well-posedness analysis of the differential equation in the Banach space.  This IDE, can also  be used for SMC design by means of  Theorem \ref{thm:SMC_design_LTI_IDE}. However,  since, by Corollary \ref{cor:Utkin_EE}, any Utkin solution satisfies \eqref{eq:EE} with $u(t)=u_{\rm eq}(t)\in U(t,Cx(t))$ almost everywhere, then a control design as well as stability/completeness/robustness analysis can also be based on the original differential equation \eqref{eq:EE} in the Banach space. 
 \begin{example}\itshape\label{ex:8}
 1) \textbf{Regularization of PDE using IDE}.	 The example \eqref{eq:heat}--\eqref{eq:heat_control}  admits the representation \eqref{eq:EE} with 
 $A\!=\!\nu\frac{\partial^2}{\partial z^2}$, $D(A)\!=\!H^2((0,1);\R)\cap H^1_0((0,1);\R)$, $(Bx)(z)\!=\!\beta(z) x(z)$, where $H^2$, $H^1_0$ are Sobolev spaces. 
 The closed densely defined operator $A: \D(A)\subset L^2((0,1);\R)\to L^2((0,1);\R)$ is a generator of the strongly continuous semigroup of linear bounded operators on $L^2((0,1);\R)$ (see, e.g., \cite[page 212]{Pazy1983:Book}).
The output \eqref{eq:EE_output} is given by $y(t)=Cx(t)=\int^1_0\xi(z)x(t,z)dz$. The control law \eqref{eq:EE_control}  is defined by the formula  \eqref{eq:heat_control}.  Since $\beta\in \D(A)$, then, by Corollary \ref{cor:Utkin_EE}, the system  \eqref{eq:heat}--\eqref{eq:heat_control}  has an Utkin solution $x(t)\!\in\! \D(A)$, $\forall x^0\in \D(A)$, $\forall  
 q\!\in\! L^{\infty}((t^0,+\infty);\R)$ in \eqref{eq:heat_control}.
 
 2) \textbf{SMC design using PDE}.	
 The closed-loop heat control system \eqref{eq:heat}-\eqref{eq:heat_control} is forward complete, i.e., all Utkin solutions exist globally in time. 
Indeed, considering $V(t)\!=\!\int^1_0 \!x^2(t,z) dz$ and using integration by parts, we derive\vspace{-1mm}
 \[\vspace{-1mm}\small 
 \begin{split}
 \dot V(t)=&\,2\int^1_0 \nu x(t,z)\tfrac{\partial^2 x(t,z)}{\partial z^2 } +x(t,z) \beta(z) (u_{\rm eq}(t)+\gamma(t))dz\\[-2mm]
=& -2\int^1_0 \nu \left(\tfrac{\partial x(t,z)}{\partial z } \right)^2 +x(t,z) \beta(z) (u_{\rm eq}(t)+\gamma(t))dz,
\end{split}
 \]
 where $u_{\rm eq}(t)\in q(t)\,\overline{\sign}(Cx(t))$.
Since $x(t,\cdot)\in \D(A)$ almost everywhere then the function $z\mapsto x(t,z)$ is continuously differentiable in $z$ for almost all $t$. Using Wirtinger and Cauchy-Schwartz inequalities,  we come up with the estimate
\[
 \dot V(t)\leq -\nu 2\pi^2  V(t)+2(\|q\|_{L^{\infty}}+\|\gamma\|_{L^{\infty}})\|\beta\|_{L^2} \sqrt{V(t)},
\]     
that is fulfilled almost everywhere.
Therefore,  any Utkin solution of the closed-loop system is bounded and  \vspace{-1mm}
 \begin{equation}\label{eq:heat_est_x}
 \limsup_{t\to +\infty} \|x(t,\cdot)\|_{L^2}\leq \tfrac{(\|q\|_{L^{\infty}}+\|\gamma\|_{L^{\infty}})\|\beta\|_{L^2}}{\nu \pi^2}. \vspace{-1mm}
 \end{equation}
Since $CB=\int^1_0 \xi(z)\beta(z) d z$ and $Cx(t)=\int^1_0 \xi(z)x(t,z) dz$ then, due to $\xi(0)=\xi(1)=x(t,0)=x(t,1)=0$ and $\xi\in C^2([0,1];\R)$, the 
 integration by parts yields \vspace{-1mm}
 \begin{equation}\label{eq:input_output_heat_temp}
 \begin{split}
 \tfrac{dCx(t)}{dt}=&\int^1_0  \nu  \xi(z) \tfrac{\partial^2 x(t,z)}{\partial z^2} dz +CB(u_{\rm eq}(t) +\gamma(t))\\
 =&\int^1_0  \nu  \tfrac{\partial^2 \xi(z)}{\partial z^2} x(t,z)dz  + CB(u_{\rm eq}(t) +\gamma(t))\\
 =& -\nu \lambda Cx(t)+\nu \int^1_0  \!\!\left(\tfrac{\partial^2 \xi(z)}{\partial z^2}\!+\!\lambda \xi(z)\right)x(t,z)dz \\
 & + CB(u_{\rm eq}(t) +\gamma(t)), \quad \quad \forall \lambda\geq 0.
 \end{split} \!\vspace{-1mm}
 \end{equation}
 Taking into account that  $u_{\rm eq}(t)\in q(t)\overline{\sign}(Cx(t))$, we conclude that, 
 for $Cx(t)\!\neq\! 0$, it holds $u_{\rm eq}(t)\!=\! q(t)\sign(Cx(t))$,  \vspace{-1mm}
 \[
 \begin{split}
\tfrac{d  |Cx(t)|}{dt}\leq&\, -\nu \lambda |Cx(t)|
+CBq(t)\\
&+\nu \left \|\tfrac{\partial^2 \xi}{\partial z^2}+\lambda \xi\right\|_{L^2} \!\|x(t,\cdot)\|_{L^2}+\|CB\gamma\|_{L^{\infty}}
\end{split} \vspace{-1mm}
 \]
 almost everywhere.
If $CB=\int^1_0\xi(z)\beta(z) dz \neq 0$ 
then, taking $\delta>0$ and  $q(t)=-\frac{\rho
}{CB}$ with $\rho>0$,
we have  due to  \eqref{eq:heat_est_x}: \vspace{-1mm}
$$
 \|x(t,\cdot)\|_{L^2}\leq \tfrac{(\rho+\|\gamma\|_{L^{\infty}}|CB|)\|\beta\|_{L^2}}{\nu \pi^2|CB|}+\delta \vspace{-1mm}
$$
for almost all  $t\geq t^0+\tilde T(x_0)$ provided that $\tilde T(x_0)\geq 0$ is large enough.
If $\exists \lambda\geq 0$ such that  \vspace{-1mm}
\begin{equation}\label{eq:cond_lambda}
	\sqrt{\!\int^1_0 \!\!\! \left(\! \tfrac{\partial^2 \xi(z)}{\partial z^2}\!+\!\lambda \xi(z)\!\right)^{\!2} \!\!\!dz \!\int^{1}_0\!\!\!\!\beta^2(z)dz} \!<\! \pi^2\! \left|\int^1_0\!\!\!\!\xi(z)\beta(z) dz\right| \vspace{-1mm}
\end{equation}
then, for  \vspace{-1mm}
\[
\rho\!>\!\tfrac{|CB|\left( \pi^2\bar \gamma |CB|+\left(\nu\delta\pi^2+\bar \gamma\|\beta\|_{L^2}\right)\left \|\frac{\partial^2 \xi}{\partial z^2}+\lambda\xi\right\|_{L^2}\right)}{ \pi^2\left|CB\right|-\|\beta\|_{L^2} \left\|\frac{\partial^2 \xi}{\partial z^2}+\lambda \xi \right\|_{L^2} }, \quad\bar \gamma\!\geq\! \|\gamma\|_{L^{\infty}}, \vspace{-1mm}
\]
we derive the following estimate\vspace{-1mm}
\[\small\vspace{-1mm}
\begin{split}
\tfrac{d  |Cx(t)|}{dt}\!\leq& \nu \left \| \tfrac{\partial^2 \xi}{\partial z^2}\!+\!\lambda \xi\right\|_{L^2}\! \left(\tfrac{(\rho+\|\gamma\|_{L^{\infty}}|CB|)\|\beta\|_{L^2}}{\nu \pi^2|CB|}\!+\!\delta\right) \\
&+|CB|\cdot \|\gamma\|_{\infty}-\rho<0
\end{split}
\]
for almost all $t\geq t^0+\tilde T(x^0)$ such that $Cx(t)\neq 0$. 
Therefore,  there exists a finite number $T\geq 0$ such that $Cx(t)\to 0$ as $t\to T$. Moreover, since  $t\mapsto Cx(t)$ is an absolutely continuous function defined for all $t\geq t^0$ then the implication 
$|Cx(t)|\neq0 \Rightarrow \frac{d |Cx(t)|}{dt}<0$ implies that $Cx(t)=0$ for all $t\geq T$.
Therefore, the system \eqref{eq:heat}, \eqref{eq:heat_boundary}, \eqref{eq:tilde_D}, \eqref{eq:heat_y}, \eqref{eq:heat_control} has a sliding mode and the identity  $\frac{dCx(t)}{d t}=0$ implies  \vspace{-1mm}
  \[
 u_{\rm eq}(t)=-\gamma(t)-\tfrac{\nu \int^1_0  \frac{\partial^2 \xi(z)}{\partial z^2} x(t,z)dz}{CB} \vspace{-1mm}
 \]
 for almost all  $t\geq T$. The dynamics of the original system in the sliding mode is defined by the equation \vspace{-1mm}
 \begin{equation}\label{eq:heam_SMC_dynamics1}
\tfrac{\partial x(t,z)}{\partial t}\!=\!\nu\tfrac{\partial^2 x(t,z)}{\partial z^2}\!-\!\tfrac{\nu \int^1_0 \frac{\partial^2 \xi(z)}{\partial z^2} x(t,z)dz}{\int^1_0 \xi(z)\beta(z)dz}\beta(z), 
 \vspace{-1mm}
 \end{equation}
with the homogeneous boundary conditions \eqref{eq:heat_boundary} and  \vspace{-1mm}
\begin{equation}\label{eq:heam_SMC_dynamics2}
\int^1_0 \xi(z) x(t,z)=0 \vspace{-1mm}
\end{equation}
for almost all $ t\geq T$.

3) \textbf{SMC design using IDE}.		The same result can be obtained  considering the input-output dynamics of the heat control system. Using the modal decomposition  \vspace{-1mm}
	\begin{equation}
	 		x(t,z)=\sum_{i=1}^{\infty} x_i(t)\phi_i(z), \quad x_i(t)\in \R,\vspace{-1mm}
	 	\end{equation}
 	where $\phi_i(z)=\sqrt{2}\sin(\sqrt{\lambda_i} z), \lambda_i=\pi^2 i^2$,  
 	the system  \eqref{eq:heat} - \eqref{eq:heat_ic} can be rewritten (see, e.g. \cite{Balas1988:JMAA}, \cite{PrieurTrelat2019:TAC},  \cite{KatzFridman2021:EJC}) in the form \vspace{-1mm}
	\begin{equation}\label{eq:heat_model}
	 		\dot x_i(t)=-\nu \lambda_ix_i(t)+b_{i}(u(t)+\gamma(t)), \quad t>t_0,\vspace{-1mm}
	 	\end{equation}
 	where $x_i(0)=x^0_i=\int^1_0\!\phi_i(z) x^0(z)dz$, $b_{i}=\int^1_0\!\phi_i(z) \beta(z) dz$.
 	 By Cauchy formula, we have\vspace{-1mm}
 	\[
 	x_i(t)\!=\!e^{-\nu \lambda_i(t-t_0)}x^0_i+\int^t_{t_0} \!\!e^{-\nu \lambda_i(t-\tau)}b_{i}(u(
 	\tau)+\gamma(\tau)) d\tau.\vspace{-1mm}
     \]
Therefore, the input-output dynamics \eqref{eq:input_output_heat_temp} is given by\vspace{-1mm}
 \begin{equation}\label{eq:input_output_heat}
 \begin{split}
	 \dot y(t)=&-\nu\lambda y(t)+p(t,t^0)+CB(u(t)+\gamma(t))\\
	 &+\int^{t}_{t^0}\Phi(t,\tau)(u(\tau)+\gamma(\tau)) d\tau,
	 \end{split}\vspace{-2mm}
 \end{equation}
 \[
 p(t{\color{blue}-}t^0)=\nu\sum_{i=1}^{\infty}e^{-\nu\lambda_i(t-t^0)}x_i^0\int^1_0  \left(\tfrac{\partial^2 \xi(z)}{\partial z^2}+\lambda \xi(z)\right) \phi_i(z)dz,\vspace{-2mm}
 \]
 \[
 \Phi(t{\color{blue}-}\tau)=\nu \sum_{i=1}^{\infty}e^{-\nu\lambda_i(t-\tau)}b_i \int^1_0  \left(\tfrac{\partial^2 \xi(z)}{\partial z^2}+\lambda \xi(z)\right) \phi_i(z)dz.\vspace{-1mm}
 \]
 Since $|\Phi(t{\color{blue}-}\tau)|\leq \nu e^{-\nu \pi^2(t-\tau)} \|\beta\|_{L^2}\left\|\frac{\partial^2 \xi}{\partial z^2}+\lambda \xi\right\|_{L^2} $ then 
 the inequality \eqref{eq:cond_lambda} yields $M=\frac{\esssup_{t\geq t^0}\int^t_{t^0}\|\Phi(t{\color{blue}-}\tau)| d\tau}{|CB|}\!<\!1$. 
 
  Let $q(t)=-\frac{|CB| \bar \gamma+\bar p+M\bar \gamma+\delta}{(1-M)CB}$ with  $\delta>0$ and $ \bar p\geq |p(t-t^0)|$.
 By Theorem \ref{thm:SMC_design_LTI_IDE} and Corollary \ref{cor:Utkin_sol_LTI_SMC},  
 the system \eqref{eq:input_output_heat}, \eqref{eq:heat_control} has the unique Utkin solution and $y(t)=0$ for all $t\geq T$, where $T\geq t^0$ is some instant of time. This means that   the Utkin solution of the system  \eqref{eq:heat}--\eqref{eq:heat_control}  satisfies  \eqref{eq:heam_SMC_dynamics2}  for all $t\geq T$.
 
 4) \textbf{Stability analysis of the heat system in the sliding mode}.	
Let us show that the system  \eqref{eq:heam_SMC_dynamics1}, \eqref{eq:heam_SMC_dynamics2}, \eqref{eq:heat_ic} is globally exponentially  stable. Considering again the functional $t\to V(t)=\int^1_0 x^2(t,z)dz$ ,  we derive\vspace{-1mm}
\[
\dot V(t)\leq -2\nu \pi^2  V(t)-2\nu \tfrac{ \int^1_0 \frac{\partial^2 \xi(z)}{\partial z^2} x(t,z)dz\int^1_0\beta(z)x(t,z)dz}{\int^1_0 \xi(z)\beta(z)dz}.\vspace{-1mm}
\]
The identity \eqref{eq:heam_SMC_dynamics2}  and Cauchy-Schwartz inequality give \vspace{-1mm}
\[\small \vspace{-1mm}
\begin{split}
\left| \int^1_0 \tfrac{\partial^2 \xi(z)}{\partial z^2} x(t,z)dz\right|=&
\left|\int^1_0 \left(\tfrac{\partial^2 \xi(z)}{\partial z^2}+\lambda  \xi(z)\right) x(t,z)dz\right|\\
 \leq &\left\|\tfrac{\partial^2 \xi}{\partial z^2}+\lambda \xi \right\|_{L^2} \sqrt{V(t)}.\vspace{-1mm}
\end{split}
\]
Using the inequality $\left|\int^1_0 \beta(z)x(t,z) dz\right|\leq \|\beta\|_{L^2}\sqrt{V(t)}$ and the condition  \eqref{eq:cond_lambda} we come up with the estimate \vspace{-1mm}
\[
\dot V(t)\leq 2V(t)\left(\tfrac{\left\|\frac{\partial^2 \xi}{\partial z^2}+\lambda \xi \right\|_{L^2} \|\beta\|_{L^2}}{|CB|}-\pi^2\right)<0\, \text{ for }\, V(t)\neq 0.\vspace{-1mm}\]
So, the system \eqref{eq:heam_SMC_dynamics1}, \eqref{eq:heam_SMC_dynamics2}, \eqref{eq:heat_ic} is globally exponentially stable.

The condition \eqref{eq:cond_lambda} is fulfilled, for example, if
\[
\beta(z)=\left\{
\begin{array}{ll}
\exp\left(\frac{1}{(0.3-x)(0.6-x)}\right) & \text{ if } x\in (0.3,0.6),\\
0  & \text{ otherwise},
\end{array}
\right.
\]  
\[
\xi(z)=\sin(\pi z)+z(1-z) \quad \text{ and } \quad \lambda=\pi^2.
\]
In this case, we have $CB\approx 0.33$,   $\|\beta\|_{L^2}\approx 0.55$ and $ \left\|\tfrac{\partial^2 \xi}{\partial z^2}+\lambda \xi \right\|_{L^2}\approx 0.8$ for $\lambda=\pi^2$. The figures \ref{fig:heat_state}, \ref{fig:heat_sliding}, \ref{fig:heat_control} show simulation results for the heat system 
\eqref{eq:heat}-\eqref{eq:heat_y} with the sliding mode control \eqref{eq:heat_control}, $q(t)=-\frac{\rho}{CB}, \rho=1$ and the matched perturbation $\gamma(t)=0.5$. The initial state $x^0\in \mathcal{D}(A)$ is  defined as follows $x_0(z)=10z(1-z),z\in [0,1]$. 
The  method of modal decomposition \cite{Balas1988:JMAA}, \cite{PrieurTrelat2019:TAC}, \cite{KatzFridman2021:EJC}, \cite{Ayamou_etal2025:IMA} with  60 modes is  utilized for numerical simulation of the closed-loop SMC system. The time sampling is $0.001$.
	\begin{figure}
	\centering
	\includegraphics[width=0.37\textwidth]{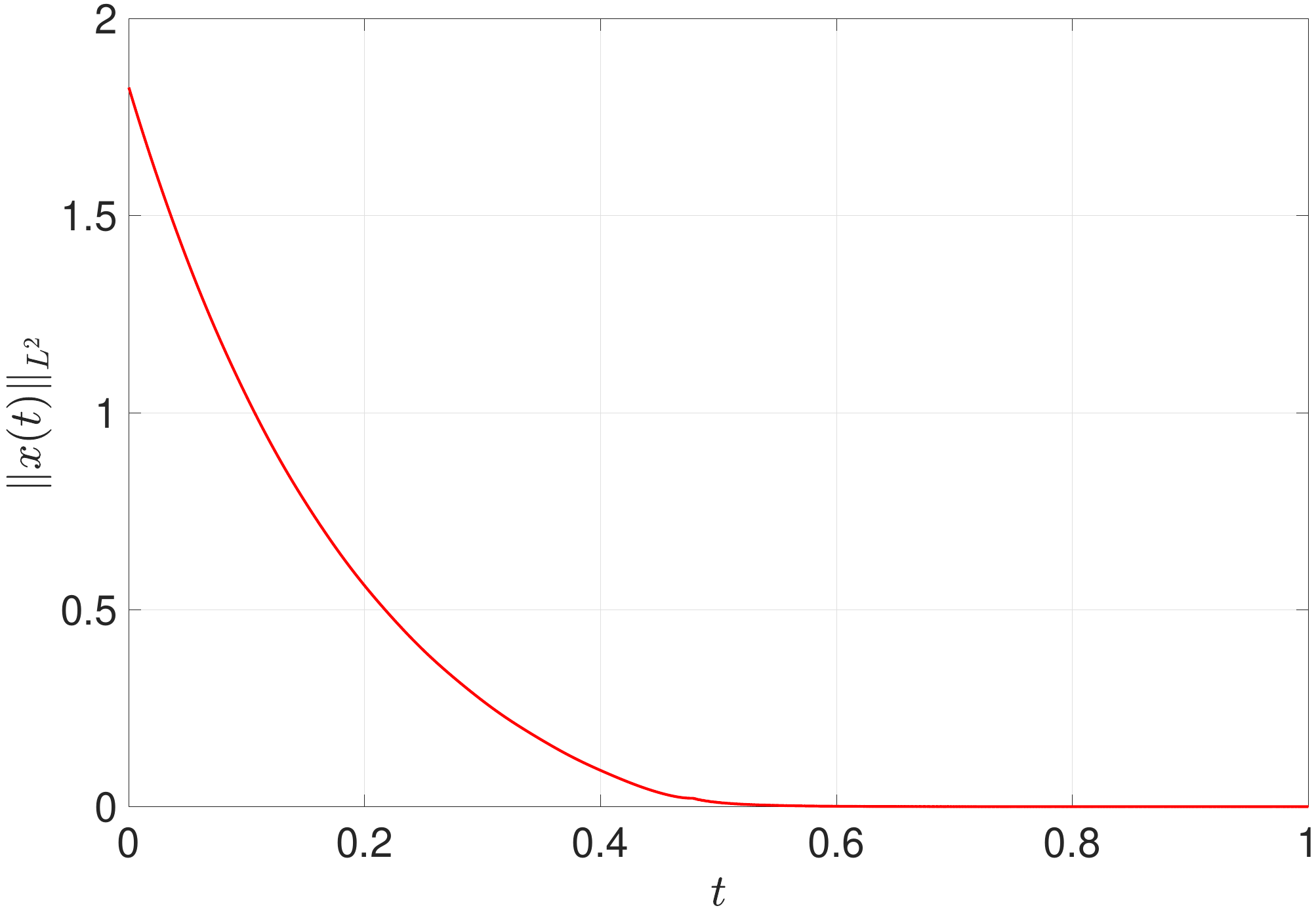}\vspace{-2mm}
	\caption{The states of the heat system \eqref{eq:heat}-\eqref{eq:heat_y} with the SMC \eqref{eq:heat_control}}
	\label{fig:heat_state}
\end{figure}
\begin{figure}
	\centering
	\includegraphics[width=0.37\textwidth]{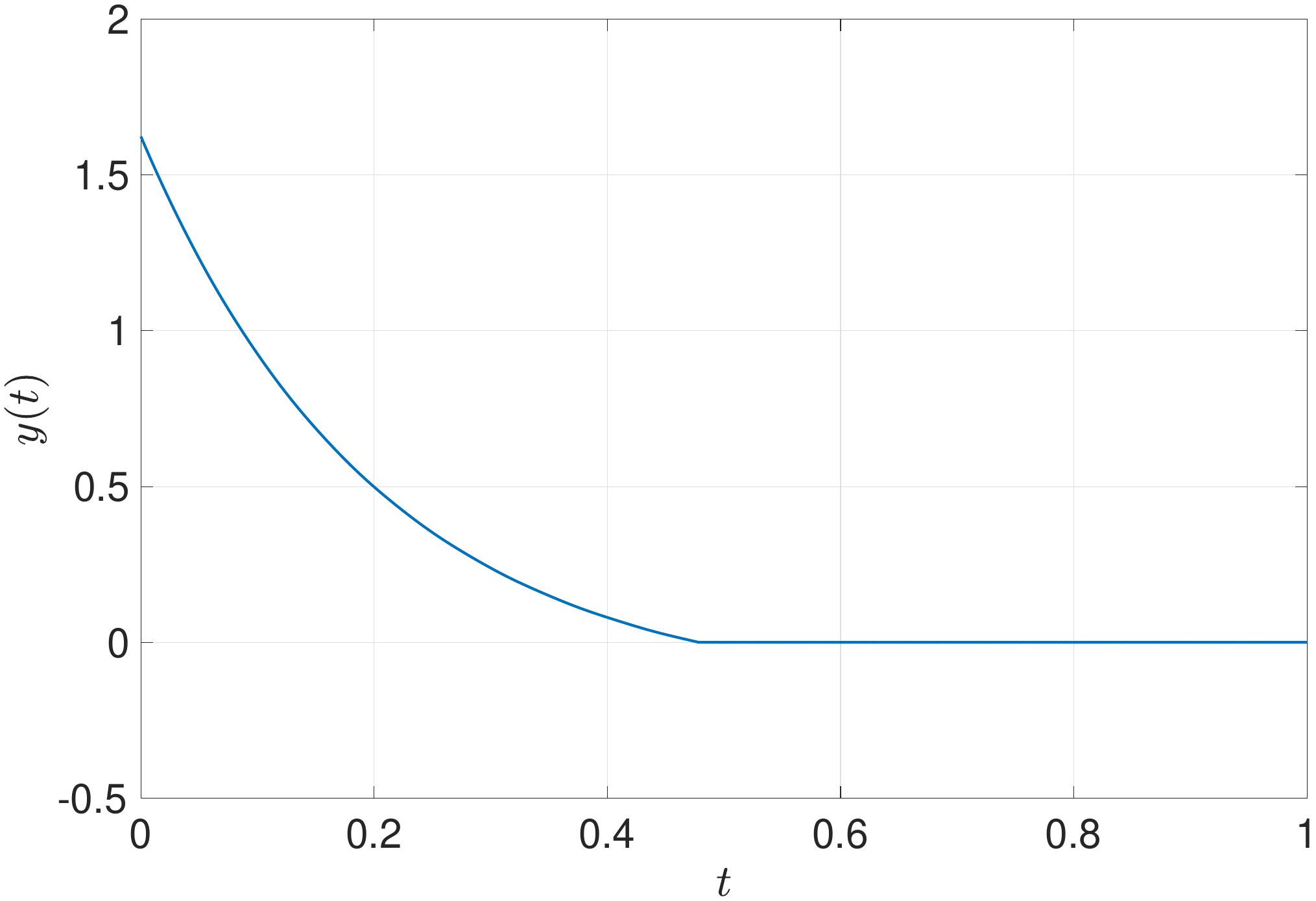}\vspace{-2mm}
	\caption{The output of the heat system \eqref{eq:heat}-\eqref{eq:heat_y} with the SMC \eqref{eq:heat_control}}
	\label{fig:heat_sliding}
\end{figure}
\begin{figure}
	\centering
	\includegraphics[width=0.37\textwidth]{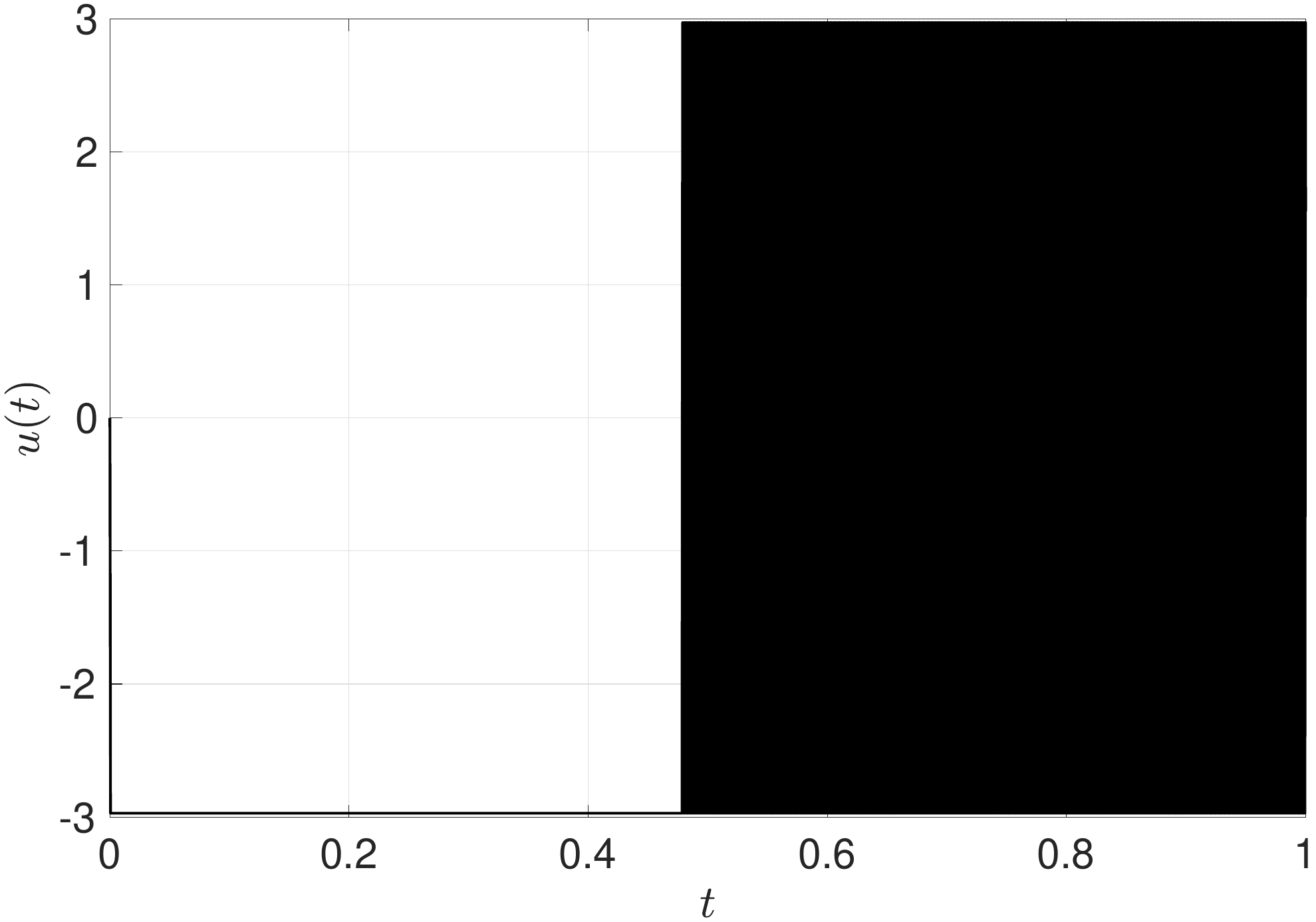}\vspace{-2mm}
	\caption{The control input of the heat system \eqref{eq:heat}-\eqref{eq:heat_y} with the SMC \eqref{eq:heat_control}}
	\label{fig:heat_control}
\end{figure}

\end{example}

The finite-dimensional IDE-based procedure of regularization of infinite-dimensional systems with sliding mode control can be extended to other systems  using, for example, 
the methodology suggested in \cite{Polyakov2025:CDC}.  Being interesting future research topic,
a further study of the SMC design for such systems goes out of the scope of this paper.

}

\section{Conclusions}\label{sec:con}

In this paper, Filippov method  is extended to systems modeled by integro-differential equations.
This method  allows an ill-posed discontinuous system to be regularized by means of its transformation to a well-posed differential inclusion.  Such ill-posed discontinuous systems appear in sliding mode control theory, which uses discontinuous feedback laws for regulation of systems and rejection of matched perturbations. The advantage of the Filippov method consists in a minimal (from set-theoretic point of view) extension of an ill-posed equation to a well-posed inclusion. This  implies a better matching of the dynamics of the control system by the regularized model.
 In the paper, the equivalent control method is developed for affine-in-control systems modeled by integro-differential equations.  The well-posedness of regularized systems is studied. Filippov method and  equivalent control approach are key tools for analysis of control systems with sliding modes.

To illustrate  theoretical results,  sliding mode control is designed for a class of systems with distributed input delay. The dynamics of the closed-loop system in the sliding mode is obtained in the form of an algebro-integro-differential equation. The equivalent control is defined as a solution of the Volterra  integral equation of the second kind.  

For a class of infinite-dimensional control systems  \eqref{eq:EE} with  finite-dimensional inputs and outputs, a sliding mode control can be designed and analyzed using the developed theory. Indeed, the input-output dynamics of such an infinite-dimensional system can always be represented as a finite-dimensional integro-differential equation. If the input-output dynamics admits the representation \eqref{eq:LTI_IDE}, \eqref{eq:LTI_y}, then  the SMC for the original system can be designed in the form \eqref{eq:SMC_LTI} provided that all conditions of Theorem \ref{thm:SMC_design_LTI_IDE} are fulfilled.
In Example \ref{ex:8}, this approach is demonstrated for a heat control system modeled by PDE. 

The obtained theoretical results  provide all necessary tools for extension  of other
methods of SMC theory  (e.g.,  high order SMC \cite{Levant2003:IJC} or implicit SMC discretization \cite{AcaryBrogliato2010:SCL})
to systems modeled by integro-differential equations  and by differential equations in Banach spaces (e.g.,  PDEs). 
\bibliographystyle{plain}
\bibliography{bib_all.bib}

\end{document}